\documentclass[a4paper,12pt,leqno]{article}
\usepackage[utf8]{inputenc}
\usepackage[usenames,dvipsnames,svgnames,table]{xcolor}
\definecolor{mygreen}{RGB}{0,108,0}
\definecolor{myblue}{RGB}{0,0,153}
\usepackage{t1enc}
\usepackage{amssymb}
\usepackage{latexsym}
\usepackage{mathrsfs}
\usepackage{amsmath}
\usepackage{amsfonts}
\usepackage{amsthm}
\usepackage{graphics}
\usepackage{graphicx}
\usepackage{wasysym}
\usepackage{stmaryrd}
\usepackage{enumerate}
\usepackage{authblk}
\usepackage{subfigure}
\usepackage{epstopdf}
\usepackage[colorlinks=true,citecolor=myblue,linkcolor=mygreen]{hyperref}
\usepackage{url}

\newcommand{\ud}{\,\mathrm{d}}

\newcommand{\rn}{\,\mathcal{R}}

\newcommand{\Z}{\mathbb{Z}}
\newcommand{\R}{\mathbb{R}}

\newtheorem{theorem}{Theorem}[section]
\newtheorem{corollary}[theorem]{Corollary}

\newtheorem{lemma}[theorem]{Lemma}

\newtheorem{proposition}[theorem]{Proposition}

\newtheorem{definition}[theorem]{Definition}

\newtheorem{example}{Example}

\usepackage[top=2.5 cm, bottom=2.5 cm, left=3.5 cm, right=2.5 cm]{geometry}

\usepackage{color}
\title{Large number of endemic equilibria for disease transmission models in patchy environment}
\author[a]{Di\' ana H. Knipl \thanks{Corresponding author. Tel.: +36 62 34 3883}}
\author[b]{Gergely R\"{o}st}
\affil[a]{\small MTA--SZTE Analysis and Stochastics Research Group, University of Szeged, Aradi v\'{e}rtan\'{u}k tere 1, Szeged, Hungary, H-6720 \textit{E-mail: knipl@math.u-szeged.hu}
}
\affil[b]{Bolyai Institute, University of Szeged, Aradi v\'{e}rtan\'{u}k tere 1, Szeged, Hungary, H-6720  \textit{E-mail: rost@math.u-szeged.hu}}
\date{}

\begin{document}

\thispagestyle{empty}
\maketitle

\begin{abstract}
We show that disease transmission models in a spatially heterogeneous environment can have a large number of coexisting endemic equilibria. A general compartmental model is considered to describe the spread of an infectious disease in a population distributed over several patches. For disconnected regions, many boundary equilibria may exist with mixed disease free and endemic components, but these steady states usually disappear in the presence of spatial dispersal. However, if backward bifurcations can occur in the regions, some partially endemic equilibria of the disconnected system move into the interior of the nonnegative cone and persist with the introduction of mobility  between the patches. We provide a mathematical procedure that precisely describes in terms of the local reproduction numbers and the connectivity network of the patches, whether a steady state of the disconnected system is preserved or ceases to exist for low volumes of travel. Our results are illustrated on a patchy HIV transmission model with subthreshold  endemic equilibria and backward bifurcation. We demonstrate the rich dynamical behavior (i.e., creation and destruction of steady states) and the presence of multiple stable endemic equilibria for various connection networks.

\end{abstract}

\section{Introduction}
Compartmental epidemic models have been considered widely in the mathematical literature since the pioneering works of Kermack, McKendrick and many others. Investigating fundamental properties of the models with analytical tools allows us to get insight into the spread and control of the disease by gaining information about the solutions of the corresponding system of differential equations. Determining steady states of the system and knowing their stability is of particular interest if one thinks of the long term behavior of the solution as final epidemic outcome. \\

In the great majority of the deterministic models for communicable diseases, two steady states exist: one disease free, meaning that the disease is not present in the population, and the other one is endemic, when the infection persists  with a positive state in some of the infected compartments. In such situation the basic reproduction number ($\rn_0$) usually works as a threshold for the stability of fixed points: typically the disease free equilibrium is locally asymptotically stable whenever this quantity, defined as the number of secondary cases generated by an index infected individual who was introduced into a completely susceptible population, is less than unity, and for values of $\rn_0$ greater than one, the endemic fixed point emerging at $\rn_0=1$ takes stability over by making the disease free state unstable. This phenomenon, known as forward bifurcation at $\rn_0=1$, is in contrary to some other cases when more than two equilibria coexist in certain parameter regions. Backward bifurcation presents such a scenario, when there is an interval for values of $\rn_0$ to the left of one where there is a stable and an unstable endemic fixed point besides the unique disease free equilibrium. Such dynamical structure of fixed points has been observed is several biological models considering multiple groups with asymmetry between groups and multiple interaction mechanisms (for an overview see, for instance, \cite{gumelbb} and the references therein). However, examples can also be found in the literature where the coexistence of multiple non-trivial steady states is not due to backward transcritical bifurcation of the disease free equilibrium; in the age-structured SIR model analyzed by Franceschetti {\it et al.} \cite{franceschetti} endemic equilibria arise through two saddle-node bifurcations of a positive fixed point, moreover Wang \cite{wang2} found backward bifurcation from an endemic equilibrium in a simple SIR model with treatment.  \\

In case of forward transcritical bifurcation, the classical disease control policy can be formulated: the stability of the endemic state typically accompanied with the persistence of the disease in the population as long as the reproduction number is larger than one, while controlling the epidemic in a way such that $\rn_0$ decreases below one successfully eliminates the infection, since every solution converges to the disease free equilibrium when $\rn_0$ is less than unity. On the other hand, the presence of backward bifurcation with a stable non-trivial fixed point for $\rn_0<1$ means that bringing the reproduction number below one is only necessary but not sufficient for disease eradication. Nevertheless, multiple endemic equilibria have further epidemiological implications, namely that stability and global behavior of the models that exhibit such structure are often not easy to analyze, henceforth little can be known about the final outcome of the epidemic.\\

\begin{figure}[t]
\centering
\subfigure[$e_1=2$, $e_2=2$, $e_3=1$, $\rn^{1}<1$, $\rn^{2}<1$, $\rn^{3}>1$.] { \includegraphics[width=4cm]{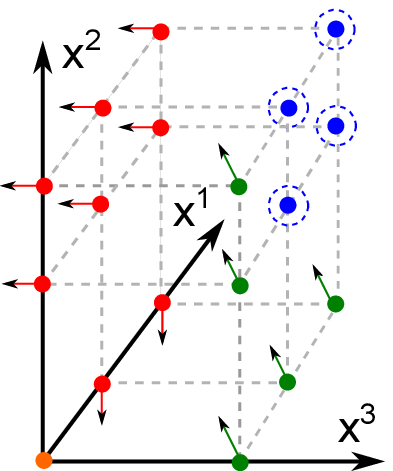}}\hspace{1cm}
\subfigure[$e_1=2$, $e_2=1$, $e_3=1$, $\rn^{1}<1$, $\rn^{2}>1$, $\rn^{3}>1$.]{\includegraphics[width=4cm]{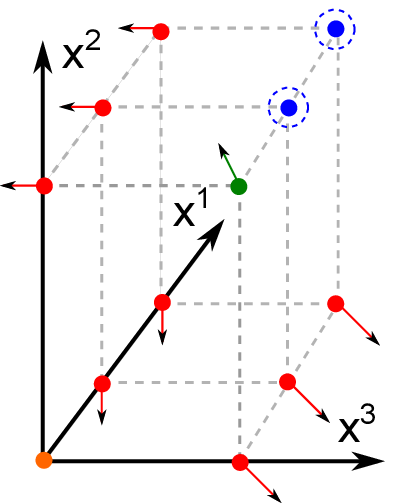}}\hspace{1cm}
\subfigure[$e_1=1$, $e_2=1$, $e_3=1$, $\rn^{1}>1$, $\rn^{2}>1$, $\rn^{3}>1$.]{\includegraphics[width=4cm]{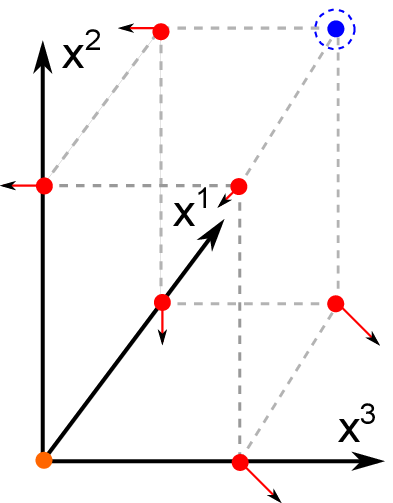}}
\caption{We illustrate the behavior of steady states in the system of three regions connected to each other by a complete mobility network. Dots on the schematic diagrams depict infected components of equilibria of the disconnected system for three different cases, $e_i$ denotes the number of positive fixed points in region $i$, $i=1,2,3$. Mobility has no impact on the disease free equilibrium (orange dot). Componentwise positive steady states (blue dots) are preserved  in the system with traveling as they continuously depend on the mobility parameter $\alpha$. A boundary endemic equilibrium moves out from the nonnegative octant with the introduction of traveling if the equilibrium has a component corresponding to a region, which is disease free in the absence of traveling and has local reproduction number ($\rn$) greater than one (red dot).
Other boundary steady states move into the interior of the nonnegative octant (green dots). \label{fig:schematic}}
\end{figure}

Multi-city epidemic models, where the population is distributed in space over several discrete geographical regions with the possibility of individuals' mobility between them, provide another example for rich dynamics. In the special case when the cities are disconnected the model possesses numerous steady states, the product of the numbers of equilibria in the one-patch models corresponding to each city. However, the introduction of traveling has a significant impact on steady states, as it often causes substantial technical difficulties in the fixed point analysis and, more importantly, makes certain equilibria disappear. Some works in the literature deal with models where the system with traveling exhibits only two steady states, one disease free with the infection not being present in any of the regions, and another one, which exists only for $\rn_0>1$, corresponding to the situation when the disease is endemic in each region (see, for instance, Arino \cite{arinometa}, Arino and van den Driessche \cite{arvdd}). Other studies which consider the spatial dispersal of infecteds between regions (Gao and Ruan \cite{gao}, Wang and Zhao \cite{wang} and the references therein) don't derive the exact number for the steady states but show the global stability of a single disease free fixed point for $\rn_0<1$ and claim the uniform persistence of the disease for $\rn_0>1$ with proving the existence of at least one (componentwise) positive equilibrium.\\  

The purpose of this study is to investigate the impact of individuals' mobility on the number of equilibria in multiregional epidemic models. A general deterministic model is formulated to describe the spread of infectious diseases with horizontal transmission. The framework enables us to consider models with multiple susceptible, infected and removed compartments, and more significantly, with several steady states. The model can be extended to an arbitrary number of regions connected by instantaneous travel, and we investigate how mobility creates or destroys equilibria in the system. First we determine the exact number of steady states for the model in disconnected regions, then give a precise condition in terms of the reproduction numbers of the regions and the connecting network for the persistence of equilibria in the system with traveling. The possibilities for a three patch scenario with backward bifurcations (i.e., when two endemic states are present for local reproduction numbers less than one) are sketched in Figure \ref{fig:schematic} (cf. Corollary \ref{cor:summaryee}). \\

The paper is organized as follows. A general class of compartmental epidemic models is presented in section \ref{sec:model}, including multigroup, multistrain and stage progression models. We consider $r$ regions which are connected by means of movement between the subpopulations and use our setting as a model building block in each region. Section \ref{sec:dfe} concerns with the unique disease free equilibrium of the multiregional system with small volumes of mobility, whilst in sections \ref{sec:ee}, \ref{sec:irred} and \ref{sec:nodirect} we consider the endemic steady states of the disconnected system and specify conditions on the connection network and the model equations for the persistence of fixed points in the system with traveling. We finish sections \ref{sec:ee}-\ref{sec:nodirect} with corollaries that summarize the achievements. The results are applied to a model for HIV transmission in three regions with various types of connecting networks in section \ref{sec:HIV}, then this model is used for the numerical simulations of section \ref{sec:richdyn} to give insight into the interesting dynamics with multiple stable endemic equilibria, caused by the possibility of traveling.  

\section{Model formulation}\label{sec:model}

We consider an arbitrary ($r$) number of regions, and use upper index to denote region $i$, $i \in \{1, \dots r\}$. Let $x^{i} \in \R^n$,  $y^{i} \in \R^m$ and  $z^{i} \in \R^k$ represent the set of infected, susceptible and removed (by means of immunity or recovery) compartments, respectively, for $n, m, k \in \Z^{+}$. The vectors $x^{i}$, $y^{i}$ and $z^{i}$ are functions of time $t$. We assume that all individuals are born susceptible, the continuous function $g^{i}(x^{i}, y^{i}, z^{i})$ models recruitment and also death of susceptible members. It is assumed that $g^{i}$ is $r-1$ times continuously differentiable. The $n \times n$ matrix $-V^{i}$ describes the transitions between infected classes as well as removals from infected states through death and recovery. It is reasonable to assume that all non-diagonal entries of $V^{i}$ are non-positive, that is, $V^{i}$ has the Z sign pattern \cite{vddwat}; moreover the sum of the components of $V^{i}u$ should also be nonnegative for any $u \geq 0$. It is shown in \cite{vddwat} that for such a matrix it holds that it is a non-singular M-matrix, moreover $(V^{i})^{-1} \geq 0$. Furthermore we let $D^{i}$ be a $k \times k$ diagonal matrix whose diagonal entries denote the removal rate in the corresponding removed class.\\
Disease transmission is described by the $m \times n$ matrix function $\mathcal{B}^{i}(x^{i}, y^{i}, z^{i})$, assumed $C^{r-1}$ on $\R^{n}_+ \times (\R^{m}_+ \setminus \{0\}) \times \R^{k}_+$, an element $\beta^{i}_{p,q} (x^{i}, y^{i}, z^{i})$ represents transmission between the $p$th susceptible class and the $q$th infected compartment. The term $(\text{diag} (y^{i}) \mathcal{B}^{i}(x^{i}, y^{i}, z^{i}) x^{i})_p$ thus has the form $(y^{i})_p \sum_{q=1}^{n} \beta^{i}_{p,q} (x^{i})_q$, $p \in \{1, \dots m\}$. For each pair $(p,q) \in \{1, \dots m\}\times \{1, \dots n\}$ we define a non-negative $n$-vector $\eta^{i}_{p,q}$ which distributes the term $(y^{i})_p \beta^{i}_{p,q} (x^{i})_q$ into the infected compartments; it necessarily holds that $\sum_{\substack{j=1} }^n (\eta^{i}_{p,q})_j=1$. Henceforth individuals who enter the $j$-th infected class when turning infected are represented by  $\sum_{p=1}^{m}\sum_{q=1}^{n} (\eta^{i}_{p,q})_{j} (y^{i})_p \beta^{i}_{p,q} (x^{i})_q$, which allows us to interpret the inflow of newly infected individuals into $x^{i}$ as $\mathcal{F}^{i}(x^{i}, y^{i}, z^{i}) x^{i}$ with $(\mathcal{F}^{i})_{j,q}=\sum_{p=1}^{m} (\eta^{i}_{p,q})_{j} (y^{i})_q \beta^{i}_{p,q}$, $j,q \in \{1, \dots n\}$.  Recovery of members of the $q$th disease compartment into the $p$th removed class is denoted by the $(p,q)$-th entry of the $k \times n$ nonnegative matrix $Z^{i}$.\\

In case of disconnected regions we can formulate the equations describing disease dynamics in region $i$, $i \in \{1, \dots r\}$, as
\begin{equation}\label{basismodel}\tag{$\textit{L}_i$}\aligned
\frac{\ud}{\ud t}x^{i}&=\mathcal{F}^{i}(x^{i}, y^{i}, z^{i}) x^{i}- V^{i} x^{i},\\
\frac{\ud}{\ud t}y^{i}&=g^{i}(x^{i}, y^{i}, z^{i})- \text{diag} (y^{i}) \mathcal{B}^{i}(x^{i}, y^{i}, z^{i}) x^{i},\\
\frac{\ud}{\ud t}z^{i}&=-D^{i} z^{i}+Z^{i} x^{i}.
\nonumber\endaligned\end{equation}
Due to its general formulation our system is applicable to describe a broad variety of epidemiological models in the literature. This is illustrated with some simple examples. 

\begin{example} Multigroup models \end{example}
Epidemiological models where, based on individual behavior, multiple homogeneous subpopulations (groups) are distinguished in the heterogeneous population are often called multigroup models. The different individual behavior is typically reflected in the incidence function as, for instance, by sexually transmitted diseases the probability of becoming infected depends on the number of contacts the individual makes, which is closely related to his / her sexual behavior. In terms of our system (\ref{basismodel}), such a model is realized if $n=m=k$ holds and the vector $\eta_{p,q}^{i}$ is defined as its $p$th component is one with all other elements zero, meaning that individuals who are in the $p$th susceptible group go into the $p$th infected class when contracting the disease. A simple SIR-type model with constant recruitment $\Lambda_j$ into the $j$th susceptible class, and $\mu_j$ and $\gamma_j$ as natural mortality rate of the $j$th subpopulation and recovery rate of individuals in $I_j$, $j\in \{1, \dots n\}$, becomes a multigroup model if its ODE system reads 
\begin{equation}\aligned
\frac{\ud}{\ud t}S_{j}(t)&=\Lambda_j-\sum_{q=1}^n \beta_{j,q} I_q(t) S_j(t)-\mu_j S_j(t),\\
\frac{\ud}{\ud t}I_{j}(t)&=\sum_{q=1}^n \beta_{j,q} I_q(t) S_j(t)-\gamma_j I_j(t)-\mu_j I_j(t),\\
\frac{\ud}{\ud t}R_{j}(t)&=\gamma_j I_j(t)-\mu_j R_j(t).
\nonumber\endaligned\end{equation}
See also the classical work of Hethcote and Ark \cite{hethcote} for epidemic spread in heterogeneous populations.   

\begin{example} Stage progression models \end{example}
These models are designed to describe the spread of infectious diseases where all newly infected individuals arrive to the same compartment and then progress through several infected stages until they recover or die. If we let $\eta_{p,q}^{i}=(1,0,\dots 0)$ for every $(p,q) \in \{1, \dots m\}\times \{1, \dots n\}$ then (\ref{basismodel}) becomes a stage progression model. The example 
\begin{equation}\aligned
\frac{\ud}{\ud t}S(t)&=\Lambda-\sum_{q=1}^n \beta_{q} I_q(t) S(t)-\mu_{_S} S(t),\\
\frac{\ud}{\ud t}I_{1}(t)&=\sum_{q=1}^n \beta_{q} I_q(t) S(t)-\gamma_1 I_1(t)-\mu_1 I_1(t),\\
\frac{\ud}{\ud t}I_{2}(t)&=\gamma_1 I_1(t)-\gamma_2 I_2(t)-\mu_2 I_2(t),\\
\vdots\\
\frac{\ud}{\ud t}I_{n}(t)&=\gamma_{n-1} I_{n-1}(t)-\gamma_n I_n(t)-\mu_n I_n(t),\\
\frac{\ud}{\ud t}R(t)&=\gamma_n I_n(t)-\mu_{_R} R(t)
\nonumber\endaligned\end{equation}
provides such a framework with one susceptible and one removed class. The more general model presented by Hyman {\it et al.} in \cite{hyman} considers different infected compartments to represent the phenomenon of changing transmission potential throughout the course of the infectious period.

\begin{example} Multistrain models \end{example}
Considering more than one infected class in an epidemic model might be necessary because of the coexistence of multiple disease strains. Individuals infected by different subtypes of pathogen belong to different disease compartments, and a new infection induced by a strain always arises in the corresponding infected class. Using the interpretation of $(\eta_{p,q})$ in (\ref{basismodel}) this can be modeled with the choice of $(\eta_{p,q}^{i})_q=1$, $p \in \{1, \dots m\}$, $q \in \{1, \dots n\}$, however it is not hard to see that the model described by the system \begin{equation}\aligned
\frac{\ud}{\ud t}S(t)&=\Lambda-\sum_{q=1}^n \beta_{q} I_q(t) S(t)-\mu_{_S} S(t),\\
\frac{\ud}{\ud t}I_{j}(t)&= \beta_{j} S(t)I_j(t) -\gamma_j I_j(t)-\mu_j I_j(t), \hspace{1cm} j=1, \dots n,\\
\frac{\ud}{\ud t}R(t)&= \sum_{q=1}^n \gamma_q I_q(t)-\mu_{_R} R(t)
\nonumber\endaligned\end{equation}
also exhibits such a structure. Van den Driessche and Wathmough refer to several works for multistrain models in section 4.4 in \cite{vddwat}, and they also provide a system with two strains and one susceptible class as an example; though we point out that their model incorporate the possibility of ``super-infection'' which is not considered in our framework. 

\bigskip
After describing our general disease transmission model in $r$ separated territories we connect the regions by means of traveling with the assumptions that travel occurs instantaneously.  We denote the matrices of movement rates from region $j$ to region $i$, $i,j \in \{1, \dots r\}$, $i \neq j$, of infected, susceptible and removed individuals by $\mathbb{A}_x^{ij}$, $\mathbb{A}_y^{ij}$ and $\mathbb{A}_z^{ij}$, respectively, which have the form $\mathbb{A}_x^{ij}=\text{diag}(\alpha_{x,1}^{ij}, \dots \alpha_{x,n}^{ij})$,  $\mathbb{A}_y^{ij}=\text{diag}(\alpha_{y,1}^{ij}, \dots \alpha_{y,m}^{ij})$ and $\mathbb{A}_z^{ij}=\text{diag}(\alpha_{z,1}^{ij}, \dots \alpha_{z,k}^{ij})$, where all entries are nonnegative. For connected regions, our model in region $i$ reads

\begin{equation}\label{genmodel}\tag{$\textit{T}_i$}\aligned
\frac{\ud}{\ud t}x^{i}&=\mathcal{F}^{i}(x^{i}, y^{i}, z^{i}) x^{i}- V^{i} x^{i}-\sum_{\substack{j=1 \\j \neq i}}^r \mathbb{A}_x^{ji} x^{i}+\sum_{\substack{j=1 \\j \neq i}}^r  \mathbb{A}_x^{ij} x^{j},\\
\frac{\ud}{\ud t}y^{i}&=g^{i}(x^{i}, y^{i}, z^{i})- \text{diag} (y^{i}) \mathcal{B}^{i}(x^{i}, y^{i}, z^{i}) x^{i}- \sum_{\substack{j=1 \\j \neq i}}^r \mathbb{A}_y^{ji} y^{i}+\sum_{\substack{j=1 \\j \neq i}}^r  \mathbb{A}_y^{ij} y^{j},\\
\frac{\ud}{\ud t}z^{i}&=-D^{i} z^{i}+Z^{i} x^{i}- \sum_{\substack{j=1 \\j \neq i}}^r \mathbb{A}_z^{ji} z^{i}+\sum_{\substack{j=1 \\j \neq i}}^r \mathbb{A}_z^{ij} z^{j}.
\nonumber\endaligned\end{equation}

\section{Disease free equilibrium and local reproduction numbers}\label{sec:dfe}

In the absence of traveling, i.e., when $\alpha_{x,\cdot}^{ij}$, $\alpha_{y,\cdot}^{ij}$, $\alpha_{z,\cdot}^{ij}=0$ for all $i, j \in \{1, \dots r\}$, the equations for a given region $i$ are independent of the equations of other regions. We assume that for each $i$ the equation $$g^{i}(0,y_0^i,0)=0$$ has a unique solution $y_0^i>0$; this yields that there exists a unique disease free equilibrium $(0, y_0^i, 0)$ in region $i$ since $x_0^{i}=0$ and the third equation of (\ref{basismodel}) implies $z_0^{i}=0$. We also suppose that all eigenvalues of the derivative $g_{y^i}^{i}(0, y_0^i,0)$ have negative real part, which establishes the local asymptotic stability of $(y_0^i,0)$ in the disease free system 
\begin{equation}\aligned
\frac{\ud}{\ud t}y^{i}&=g^{i}(0, y^{i}, z^{i}),\\
\frac{\ud}{\ud t}z^{i}&=-D^{i} z^{i}.
\nonumber\endaligned\end{equation}
When system (\ref{basismodel}) is close to the disease free equilibrium, the dynamics in the infected classes can be approximated by the linear equation 
$$\frac{\ud}{\ud t}x^{i}=(F^i - V^{i}) x^{i},$$ 
where we use the notation $F^i=\mathcal{F}^{i}(0, y_0^{i}, 0)$. The transmission matrix $F^i$ represents the production of new infections while $V^i$ describes transition between and out of the infected classes. Clearly  $F^i$ is nonnegative, which together with  $(V^{i})^{-1} \geq 0$ implies the non-negativity of  $F^i(V^i)^{-1}$. We recall that the spectral radius $\rho(A)$ of a matrix $A \geq 0$ is the largest real eigenvalue of $A$ (according to the Frobenius--Perron theorem such an eigenvalue always exists for non-negative matrices, and it dominates the modulus of all other eigenvalues).  We define the {\it local} reproduction number in region $i$ as
\begin{equation}\aligned
\rn^i&=\rho(F^i (V^i)^{-1}),
\nonumber\endaligned\end{equation}
and obtain the following result.
\begin{proposition}
The point $(0, y_0^i, 0)$ is locally asymptotically stable in (\ref{basismodel}) if $\rn^i<1$, and unstable if $\rn^i>1$. 
\end{proposition}
\begin{proof}
The stability of the disease free fixed point is determined by the eigenvalues of the Jacobian of (\ref{basismodel}) evaluated at the equilibrium. Linearizing the system at $(0, y_0^i,0)$ yields
\begin{equation}
J^i= 
\begin{pmatrix}
F^i-V^i & 0 & 0\\
g_{x^i}^{i}(0, y_0^i,0)-\text{diag} (y^{i}) \mathcal{B}^{i}(0, y_0^i,0) & g_{y^i}^{i}(0, y_0^i,0) & g_{z^i}^{i}(0, y_0^i,0)\\
Z^i & 0 & -D^i
\end{pmatrix},
\nonumber\end{equation}
where it holds that $-D^i$ has negative real eigenvalues, and by assumption the eigenvalues of $g_{y^i}^{i}(0, y_0^i,0)$ have negative real part.  The special structure of $J^i$ implies that $F^i-V^i$  determines the stability of the disease free equilibrium.\\ 
It is known \cite{vddwat} that all eigenvalues of the matrix $F^{i}-V^{i}$ have negative real part if and only if $\rho(F^i (V^i)^{-1})<1$, and there is an eigenvalue with positive real part if and only if $\rho(F^i (V^i)^{-1})>1$. Since $\rn^i$ was defined as the spectral radius of $F^i (V^i)^{-1}$, one obtains the statement of the proposition. 
\end{proof}
If the regions are disconnected, the basic (global) reproduction number arises as the maximum of the local reproduction numbers, hence we arrive to the following simple proposition.
\begin{proposition}\label{prop:stabalpha0}
The system $(L_1)$--$(L_r)$ has a unique disease free equilibrium $E_{df}^0=(0, y_0^1, 0,$ $\dots$ $0, y_0^r, 0)$, which is locally asymptotically stable if $\rn^B_0<1$ and is unstable if $\rn^B_0>1$, where we define
$$\rn^B_0=\max_{1 \leq i \leq r} \rn^i.$$
\end{proposition}

Let us suppose that all movement rates admit the form $\alpha_{x,\cdot}^{ij}=\alpha \cdot c_{x,\cdot}^{ij}$, $\alpha_{y,\cdot}^{ij}=\alpha \cdot c_{y,\cdot}^{ij}$, $\alpha_{z,\cdot}^{ij}=\alpha \cdot c_{z,\cdot}^{ij}$, where the non-negative constants $c_{x,\cdot}^{ij}$, $c_{y,\cdot}^{ij}$ and $c_{z,\cdot}^{ij}$ represent connectivity potential and we can think of $\alpha \geq 0$ as the general mobility parameter. Using the notation $C_w^{ij}=\text{diag}(c_{w,1}^{ij}, \dots c_{w,n}^{ij})$ makes $\mathbb{A}_w^{ij}=\alpha C_w^{ij}$, $w \in \{x,y,z\}$. With this formulation we can control all movement rates at once, through the parameter $\alpha$, moreover it allows us to rewrite systems $(T_1)$ -- $(T_r)$ in the compact form
\begin{equation}\aligned\label{compt}
\frac{\ud}{\ud t}\mathcal{X}&=\mathcal{T}(\alpha,\mathcal{X})
\endaligned\end{equation}
with $\mathcal{X}=(x^{1},y^{1},z^{1}, \dots x^{r},y^{r},z^{r})^T \in \R^{r(n+m+k)}$ and $\mathcal{T}=(\mathcal{T}^{1,x},\mathcal{T}^{1,y},\mathcal{T}^{1,z},$ $\dots$ $\mathcal{T}^{r,x},\mathcal{T}^{r,y},\mathcal{T}^{r,z})^T\colon \R \times \R^{r(n+m+k)} \shortrightarrow \R^{r(n+m+k)}$, where $\mathcal{T}^{i,x}$, $\mathcal{T}^{i,y}$ and $\mathcal{T}^{i,z}$ are defined as the right hand side of the first, second and third equation, respectively, of system (\ref{genmodel}), $i \in \{1, \dots r\}$. We note that $\mathcal{T}$ is an $r-1$ times continuously differentiable function on $\left(\R \times \R^{n}_+ \times (\R^{m}_+ \setminus \{0\}) \times \R^{k}_+ \times \dots \times \R^{n}_+ \times (\R^{m}_+ \setminus \{0\}) \times \R^{k}_+\right)$, and for $\alpha=0$ (\ref{compt}) gives system $(L_1)$--$(L_r)$.\\
As pointed out in Proposition \ref{prop:stabalpha0}, the point $E_{df}^0=(0, y_0^1, 0,$ $\dots$ $0, y_0^r, 0)$ is the unique disease free equilibrium of $(L_1)$--$(L_r)$. Since this system coincides with $(T_1)$ -- $(T_r)$ for $\alpha=0$, it holds that $\mathcal{T}(0,E_{df}^0)=0$, this is, $E_{df}^0$ is a disease free steady state of $(T_1)$ -- $(T_r)$ when $\alpha=0$, and it is unique. The following theorem establishes the existence of a unique disease free equilibrium of this system for small positive $\alpha$-s.

\begin{theorem}\label{th:iftdfe}
Assume that the matrix $\left(\frac{\partial \mathcal{T}}{\partial \mathcal{X}}\right) (0,E_{df}^0)$ is invertible. Then, by means of the implicit function theorem it holds that there exists an $\alpha_0>0$, an open set $U_0$ containing $E_{df}^0$, and a unique $r-1$ times continuously differentiable function $f_0=$ $(f_{x_0^1}, f_{y_0^1}, f_{z_0^1},$ $\dots f_{x_0^r}, f_{y_0^r}, f_{z_0^r})^T\colon$ $[0, \alpha_0) \shortrightarrow U_0$ such that $f_0(0)=E_{df}^0$ and $\mathcal{T}(\alpha,f_0(\alpha))=0$ for $\alpha \in [0, \alpha_0)$. Moreover, $\alpha_0$ can be defined such that $f_0$ is the unique disease free equilibrium of system $(T_1)$--$(T_r)$ on $[0, \alpha_0)$.
\end{theorem}
\begin{proof}
The existence of $f_0$, the continuous function which satisfies the fixed point equations of (\ref{compt}) for small $\alpha$-s, is straightforward so it remains to show that it defines a disease free steady state when $\alpha$ is sufficiently close to zero.\\ 
We consider the following system for the susceptible classes of the model with traveling
\begin{equation}\aligned\label{eq:suscsub}
\frac{\ud}{\ud t}y^{1}&=g^{1}(0, y^{1}, 0)- \sum_{\substack{j=1 \\j \neq 1}}^r \alpha C_y^{j1} y^{1}+\sum_{\substack{j=1 \\j \neq 1}}^r  \alpha C_y^{1j} y^{j},\\
\vdots&\\
\frac{\ud}{\ud t}y^{r}&=g^{r}(0, y^{r}, 0)- \sum_{\substack{j=1 \\j \neq r}}^r \alpha C_y^{jr} y^{r}+\sum_{\substack{j=1 \\j \neq r}}^r  \alpha C_y^{rj} y^{j}.
\endaligned\end{equation}
\begin{sloppypar}
The Jacobian evaluated at the disease free equilibrium and $\alpha=0$ reads  $\text{diag}(g^{i}_{y^{i}}(0,y_0^i,0))$,  its non-singularity follows from the assumption made earlier in this section that all eigenvalues of $g^{i}_{y^{i}}(0,y_0^i,0)$, $i \in \{1, \dots r\}$, have negative real part. We again apply the implicit function theorem and get that in the absence of the disease the susceptible subsystem obtains a unique equilibrium for small values of $\alpha$. More precisely, there is an $r-1$ times continuously differentiable function $\tilde{f}_0^{y}(\alpha) \in \R^{r m}$, which satisfies the steady-state equations of (\ref{eq:suscsub}) whenever $\alpha$ is in $[0,\tilde{\alpha}_0)$ with $\tilde{\alpha}_0$ close to zero, and it also holds that $\tilde{f}_0^{y}(0)=(y_0^1, \dots y_0^r)^T$. On the other hand, we note that the point $(0, (\tilde{f}_0^{y})_1,0, \dots 0,(\tilde{f}_0^{y})_r,0)^T$ is an equilibrium solution of system $(T_1)$--$(T_r)$, and by uniqueness it follows that $f_0=(0, (\tilde{f}_0^{y})_1,0, \dots 0,(\tilde{f}_0^{y})_r,0)^T$, and necessarily $(f_{y_0^1}, \dots f_{y_0^r})^T=\tilde{f}_0^{y}$, for $\alpha <\min\{\alpha_0,\tilde{\alpha}_0\}$. By continuity it is clear from $f_{y_0^i}(0)=y_0^{i}>0$,  $i \in \{1, \dots r\}$, that $\alpha_0$ can be defined such that $f_0$ is nonnegative, and thus, it is a disease free fixed point of $(T_1)$--$(T_r)$ which is biologically meaningful.
\end{sloppypar}
\end{proof}

If $E_{df}^0$ is locally asymptotically stable in system $(L_1)$--$(L_r)$ then  $\left(\frac{\partial \mathcal{T}}{\partial \mathcal{X}}\right) (0,E_{df}^0)$ has only eigenvalues with negative real part, and therefore is invertible. By continuity of the eigenvalues with respect to parameters all eigenvalues of $\left(\frac{\partial \mathcal{T}}{\partial \mathcal{X}}\right) (\alpha,f_0(\alpha))$ have negative real part if $\alpha$ is sufficiently small. Similarly, if $E_{df}^0$ is unstable and $\left(\frac{\partial \mathcal{T}}{\partial \mathcal{X}}\right) (0,E_{df}^0)$ has no eigenvalues on the imaginary axis then, for $\alpha$-s close enough to zero, $f_0(\alpha)$ has an eigenvalue with positive real part and thus, is unstable. We have learned from Proposition \ref{prop:stabalpha0} that $\rn^B_0$ works as a threshold for the stability of the disease free steady state for $\alpha=0$, and now we obtain that this is not changed when traveling is introduced with small volumes into the system.
\begin{proposition}
There exists an $\alpha^*_0 >0$ such that $f_0(\alpha)$ is locally asymptotically stable on $[0, \alpha^*_0)$ if $\rn^B_0<1$, and in case $\rn^B_0>1$ and $\det \left(\frac{\partial \mathcal{T}}{\partial \mathcal{X}}\right) (0,E_{df}^0)  \neq 0$, $\alpha^*_0$ can be chosen such that it also holds that $f_0(\alpha)$ is unstable for $\alpha < \alpha^*_0$.
 \end{proposition}

\section{Endemic equilibria}\label{sec:ee}
Next we examine endemic equilibria $(\hat{x}^i, \hat{y}^i, \hat{z}^i)$, $\hat{x}^i\neq0$, of system (\ref{basismodel}). We assume that the functions and matrices defined for the model are such that either $\hat{w}^i=0$ or $\hat{w}^i>0$ holds for $w \in \{x, y, z\}$, that is, in region $i$ if any of the infected (susceptible) (removed) compartments are at positive steady state  then so are the other infected (susceptible) (removed) classes. Endemic fixed points thus admit $\hat{x}^i>0$, which implies $\hat{y}^i>0$ and $\hat{z}^i>0$. Indeed, the equilibrium condition for system (\ref{basismodel})
\begin{equation}
-D^{i} z^{i}+Z^{i} x^{i}=0
\nonumber\end{equation}
and $Z^i\geq 0$, $Z^i\neq 0$ gives $\hat{z}^i \neq 0$ if $\hat{x}^i>0$, so our assumption above implies that $z^{i}$ is at positive steady state in endemic equilibria. On the other hand, $\hat{y}^i=0$ would make $\mathcal{F}^i=0$, so using the non-singularity of $V^{i}$ and the first equation of (\ref{basismodel}), $V^{i} \hat{x}^i=0$ contradicts $\hat{x}^i>0$. Endemic equilibria of the regions can thus be referred to as positive fixed points. \\
Without connections between the regions, let region $i$ have $e_i \geq 1$ positive fixed points  $(\hat{x}^i, \hat{y}^i, \hat{z}^i)_{1}$, $\dots$ $(\hat{x}^i, \hat{y}^i, \hat{z}^i)_{e_i}$. Then the disconnected system $(L_1)$--$(L_r)$ admits $\left(\prod_{i=1}^{r} (e_i+1)\right) -1$ endemic equilibria of the form $EE^0=(EE_1, \dots EE_r)$, $EE_i \in \{(0, y_0^i, 0), (\hat{x}^i, \hat{y}^i, \hat{z}^i)_{1}, \dots (\hat{x}^i, \hat{y}^i, \hat{z}^i)_{e_i}\}$, and $EE^0 \neq (0, y_0^1, 0,$ $\dots$ $0, y_0^r, 0)$, the disease free steady state. In the sequel we will use the general notation $EE^0=(\hat{x}^1, \hat{y}^1, \hat{z}^1, \dots \hat{x}^r, \hat{y}^r, \hat{z}^r)$, where $\hat{x}^i=0$ for an $i$ means $(\hat{x}^i, \hat{y}^i, \hat{z}^i)=(0, y_0^i, 0)$. The upper index `$0$' in $EE^0$ stands for $\alpha=0$. We note that $\mathcal{T}(0,EE^0)=0$ holds with $\mathcal{T}$ defined for system (\ref{compt}).\\

The implicit function theorem is also applicable for any of the endemic equilibria under the assumption that the Jacobian of system (\ref{compt}) evaluated at the fixed point and $\alpha=0$ has nonzero determinant. We remark that whenever  $EE^0$ is asymptotically stable, that is, $EE_i$ is asymptotically stable in (\ref{basismodel}) for all $i \in \{1, \dots r \}$, then $\left(\frac{\partial \mathcal{T}}{\partial \mathcal{X}}\right) (0,EE^0)$ has no eigenvalues on the imaginary axis and thus, is nonsingular.
\begin{sloppypar}
\begin{theorem}\label{th:iftendemic}
Assume that the matrix $\left(\frac{\partial \mathcal{T}}{\partial \mathcal{X}}\right) (0,EE^0)$ is invertible. Then, by means of the implicit function theorem it holds that there exists an $\alpha_E$, an open set $U_E$ containing $EE^0$, and a unique $r-1$ times continuously differentiable function $f=$ $(f_{\hat{x}^1}, f_{\hat{y}^1}, f_{\hat{z}^1},$ $\dots f_{\hat{x}^r}, f_{\hat{y}^r}, f_{\hat{z}^r})^T\colon$ $[0, \alpha_E)\shortrightarrow U_E$ such that $f(0)=EE^0$ and $\mathcal{T}(\alpha,f(\alpha))=0$ for $\alpha \in [0, \alpha_E)$. By continuity of eigenvalues with respect to parameters $\det \left(\frac{\partial \mathcal{T}}{\partial \mathcal{X}}\right) (0,EE^0) \neq 0$ implies $\det \left(\frac{\partial \mathcal{T}}{\partial \mathcal{X}}\right) (\alpha,f(\alpha)) \neq 0$ for $\alpha$-s sufficiently small, thus on an interval $[0,\alpha^*_E)$ it holds that $f(\alpha)$ is a locally asymptotically stable (unstable) steady state of $(T_1)$--$(T_r)$ whenever $EE^0$ is locally asymptotically stable (unstable) in $(L_1)$--$(L_r)$.
\end{theorem}
\end{sloppypar}
The last theorem means that, under certain assumptions on our system, it holds that for every equilibrium $EE^0$ of the disconnected system $(L_1)$--$(L_r)$ there is a fixed point $f(\alpha)$, $f(0)=EE^0$, of $(T_1)$--$(T_r)$ close to $EE^0$ when $\alpha$ is sufficiently small. If $EE^0$ has only positive components then so does $f(\alpha)$, so we arrive to the following result.
\begin{theorem}\label{th:posee}
If $EE^0$ is a positive equilibrium of $(L_1)$--$(L_r)$ then $\alpha_E$ in Theorem \ref{th:iftendemic} can be chosen such that $f(\alpha)>0$ holds for $ \alpha \in [0,\alpha_E)$. This means that the equilibrium $EE^0$ of the disconnected system is preserved for small volumes of movement by a unique function which depends continuously on $\alpha$.
\end{theorem}

On the other hand, it is possible that the $EE^0=f(0)$ has some zero components when there is a region $i$, $i\in \{1, \dots r\}$, where $\hat{x}^i=0$ and $\hat{z}^i=0$ hold, that is, the fixed point is on the boundary of the nonnegative cone of $\R^{r(n+m+k)}$; nevertheless we recall that $EE^0$ is an endemic equilibrium so there exists a $j \in \{1, \dots r\}$, $j \neq i$, such that $\hat{x}^j > 0$. In the sequel such fixed points will be referred to as {\it boundary} endemic equilibria. The biological interpretation of such a situation is that, when the regions are disconnected, the disease is endemic in some regions but is not present in others. In this case $f(\alpha)$ may move out of the nonnegative cone of $\R^{r(n+m+k)}$ as $\alpha$ increases, which means that, though $f(\alpha)$ is a fixed point of system $(T_1)$--$(T_r)$, it is not biologically meaningful. Henceforth it is essential to describe under which conditions is $f(\alpha) \geq 0$ fulfilled. This will be done in the following two lemmas but before we proceed let us introduce a definition to facilitate notations and terminology.

\begin{definition}
Consider an endemic equilibrium $EE^0$ of system $(L_1)$--$(L_r)$. \\
If there is a region $i$ which is at a disease free steady state in $EE^0$ then we say that region $i$ is DFAT (disease free in the absence of traveling) in the endemic equilibrium $EE^0$, that is, $\hat{x}^i=0$.\\
If there is a region $j$ which is at an endemic (positive) steady state in $EE^0$ then we say that region $j$ is EAT (endemic in the absence of traveling) in the endemic equilibrium $EE^0$, that is, $\hat{x}^j>0$.
\end{definition}

\begin{lemma}\label{lem:suffforpos}
Consider a boundary endemic equilibrium $EE^0$ of system $(L_1)$--$(L_r)$. For the function $f(\alpha)$ defined in Theorem \ref{th:iftendemic} to be nonnegative for small $\alpha$-s it is necessary and sufficient to ensure that $f_{\hat{x}^i} (\alpha)\geq0 $ holds  for all $i$-s such that $\hat{x}^i=0$ in $EE^0$, that is, $i$ is DFAT. 
\end{lemma}

\begin{proof}
We recall that in an endemic equilibrium $\hat{y}^j>0$ holds by assumption for any $j\in \{1, \dots r\}$, thus for an $i$ with $\hat{x}^i=0$ the positivity of $f_{\hat{y}^i}(\alpha)$ for small $\alpha$-s follows from $f_{\hat{y}^i}(0)=y_0^i$ and the continuity of $f$. From (\ref{genmodel}) we derive the fixed point equation
\begin{equation}\label{eq:fxfz}
\begin{pmatrix}
Z^1 & 0 & \dots& 0\\
0 & Z^2 & \dots& 0\\
\vdots & \vdots & \ddots &\vdots\\
0 & 0 & \dots & Z^r 
\end{pmatrix} \begin{pmatrix}
f_{\hat{x}^1}(\alpha) \\
f_{\hat{x}^2}(\alpha)\\
\vdots\\
f_{\hat{x}^r}(\alpha)
\end{pmatrix}=M_z \begin{pmatrix}
f_{\hat{z}^1}(\alpha) \\
f_{\hat{z}^2}(\alpha)\\
\vdots\\
f_{\hat{z}^r}(\alpha)
\end{pmatrix},
\end{equation}
where $M_z$ is defined as
\begin{equation}\aligned
M_z&=\begin{pmatrix}
D_1+\sum_{\substack{j=1 \\j \neq 1}}^r \alpha C_z^{j1} & -\alpha C_z^{12} & \dots & -\alpha C_z^{1r}\\
-\alpha C_z^{21}& D_2+\sum_{\substack{j=1 \\j \neq 2}}^r \alpha C_z^{j2}& \dots & -\alpha C_z^{2r}\\
\vdots & \vdots & \ddots &\vdots\\
-\alpha C_z^{r1} & -\alpha C_z^{r2} &\dots  & D_r+\sum_{\substack{j=1 \\j \neq r}}^r \alpha C_z^{jr} 
\end{pmatrix}.
\nonumber\endaligned\end{equation} 
All non-diagonal elements of this $r k \times r k$ matrix are non-positive, thus it has the Z sing pattern \cite{vddwat}, moreover we also note that in each column the diagonal element dominates the absolute sum of all non-diagonal entries since $D_i >0$,  $i\in \{1, \dots r\}$. Then, we can apply Theorem 5.1 in \cite{friedler} where the equivalence of properties 3 and 11 claims that $M_z$ is invertible with the inverse nonnegative. Using the non-negativity of $Z_i$, $i \in \{1, \dots r\}$, and equation (\ref{eq:fxfz}) we get that  $f_{\hat{z}^i}(\alpha) \geq 0$ for all $i \in \{1, \dots r\}$ whenever the vector $(f_{\hat{x}^1}(\alpha), \dots f_{\hat{x}^r}(\alpha))$ is nonnegative. If $\hat{x}^j>0$ in a region $j$, meaning that the region is endemic in the absence of traveling, then for $\alpha$-s close to zero it holds that $f_{\hat{x}^j}(\alpha)>0$  since $f$ is continuous and $f_{\hat{x}^j}(0)=\hat{x}^j$. It is therefore enough (though, clearly, also necessary as well) to guarantee the nonnegativity of $f_{\hat{x}^i}(\alpha)$ for each region $i$ where $\hat{x}^i=0$, that is, the region is DFAT.
\end{proof}

\begin{lemma}\label{lem:firstder}
Consider a boundary endemic equilibrium $EE^0$ of system $(L_1)$--$(L_r)$. If $\frac{\ud f_{\hat{x}^i}}{\ud \alpha} (0) > 0$ is satisfied for the function $f$ defined in Theorem \ref{th:iftendemic} whenever region $i$ is DFAT in $EE^0$, then $f_{\hat{x}^i}(\alpha)$ is positive for $\alpha$-s sufficiently small. On the other hand if there is a region $i$, which is DFAT and for which $\frac{\ud f_{\hat{x}^i}}{\ud \alpha} (0)$ has a negative component then there is no interval for $\alpha$ to the right of zero such that $f(\alpha)$ is nonnegative. The derivative arises as the solution of the equation
\begin{equation}\aligned\label{eq:derx}
\left(V^{i}-F^i \right)\frac{\ud f_{\hat{x}^i}}{\ud \alpha} (0)&=\sum_{\substack{j=1 \\j \neq i}}^r  C_x^{ij} \hat{x}^j.
\endaligned\end{equation}
\end{lemma}

\begin{proof}
We consider a region $i$ where $\hat{x}^i=0$, this is, $i$ is a DFAT region in $EE^0$. Using the equilibrium condition $\mathcal{T}^{i,x}(\alpha,f(\alpha))=0$  we obtain
\begin{equation}\aligned\label{eq:deralpha}
\frac{\ud}{\ud \alpha} \biggl( \mathcal{F}^{i}(f_{\hat{x}^i}(\alpha), f_{\hat{y}^i}(\alpha), f_{\hat{z}^i}(\alpha)) f_{\hat{x}^i}(\alpha)- V^{i} f_{\hat{x}^i}(\alpha)&\\
-\sum_{\substack{j=1 \\j \neq i}}^r \alpha C_x^{ji} f_{\hat{x}^i}(\alpha)+\sum_{\substack{j=1 \\j \neq i}}^r  \alpha C_x^{ij} f_{\hat{x}^j}(\alpha)\biggr)&=\\
\frac{\ud}{\ud \alpha} \biggl(\mathcal{F}^{i}(f_{\hat{x}^i}(\alpha), f_{\hat{y}^i}(\alpha), f_{\hat{z}^i}(\alpha))\biggr) f_{\hat{x}^i}(\alpha) +\mathcal{F}^{i}(f_{\hat{x}^i}(\alpha), f_{\hat{y}^i}(\alpha), f_{\hat{z}^i}(\alpha))\cdot &\\
\cdot \frac{\ud f_{\hat{x}^i}}{\ud \alpha} (\alpha)
- V^{i} \frac{\ud f_{\hat{x}^i}}{\ud \alpha}(\alpha)-\sum_{\substack{j=1 \\j \neq i}}^rC_x^{ji} f_{\hat{x}^i}(\alpha)&\\
-\sum_{\substack{j=1 \\j \neq i}}^r \alpha C_x^{ji}\frac{\ud f_{\hat{x}^i}}{\ud \alpha} (\alpha)+\sum_{\substack{j=1 \\j \neq i}}^r  C_x^{ij} f_{\hat{x}^j}(\alpha)+\sum_{\substack{j=1 \\j \neq i}}^r  \alpha C_x^{ij}\frac{\ud f_{\hat{x}^j}}{\ud \alpha} (\alpha)&=0,
\endaligned\end{equation}
where we remark that $\mathcal{T}^{i,x}$ is differentiable at fixed points since $f_{\hat{y}^i}(\alpha)>0$ and $\mathcal{T}^{i}\in C^{r-1}$ when $y^i \neq 0$. Evaluating (\ref{eq:deralpha}) at $\alpha=0$ gives
\begin{equation}\aligned
\left(\mathcal{F}^{i}(0, \hat{y}^i, \hat{z}^i)- V^{i} \right)\frac{\ud f_{\hat{x}^i}}{\ud \alpha} (0)&=-\sum_{\substack{j=1 \\j \neq i}}^r  C_x^{ij} \hat{x}^j,
\nonumber\endaligned\end{equation}
where we used that $f_{\hat{x}^j}(0)=\hat{x}^j$, $f_{\hat{y}^j}(0)=\hat{y}^j$ and $f_{\hat{z}^j}(0)=\hat{z}^j$ for  $j \in \{1, \dots r\}$ and $\hat{x}^i=0$. Note that $(0, \hat{y}^i,\hat{z}^i)$ is an equilibrium in (\ref{basismodel}) and, since its component for the infected classes is zero, it equals the unique disease free equilibrium $(0,y_0^i, 0)$. This makes $\mathcal{F}^{i}(0, \hat{y}^i,\hat{z}^i)=\mathcal{F}^{i}(0,y_0^i, 0)$, so applying the definition of $F^i$ in section \ref{sec:dfe} the above equations reformulate as
\begin{equation}\aligned
\left(V^{i}-F^i \right)\frac{\ud f_{\hat{x}^i}}{\ud \alpha} (0)&=\sum_{\substack{j=1 \\j \neq i}}^r  C_x^{ij} \hat{x}^j.
\nonumber\endaligned\end{equation}
\end{proof}

Before we investigate the solutions of equation (\ref{eq:derx}) let us point out a few things. When introducing traveling a fixed point of $(T_1)$--$(T_r)$ moves along the continuous function $f(\alpha)$. In the case when there are regions where the disease is not present  without traveling and the fixed point $f$ has zeros for $\alpha=0$, it is possible that $f(\alpha)$ is non-positive for small positive $\alpha$-s. The epidemiological implication of such a situation is that boundary equilibria of the disconnected system  might disappear when traveling is introduced. \\ 
Considering a boundary endemic equilibrium $EE^0$, Lemmas \ref{lem:suffforpos} and \ref{lem:firstder} describe when such a case is realized and give condition for the non-negativity of $f(\alpha)$, $f(0)=EE^0$, for small positive $\alpha$-s.  The equation (\ref{eq:derx}) is derived for an $i \in \{1, \dots r\}$ for which $f_{\hat{x}^i}(0)=\hat{x}^i=0$ holds; the right hand side of (\ref{eq:derx}) is a nonnegative $n$-vector with the $q$th component having the form  $\left(\sum_{\substack{j=1 \\j \neq i}}^r  C_x^{ij} \hat{x}^j\right)_q=\sum_{\substack{j=1 \\j \neq i}}^r  c_{x,q}^{i,j} (\hat{x}^j)_q$. It is clear that $\left(\sum_{\substack{j=1 \\j \neq i}}^r  C_x^{ij} \hat{x}^j\right)_q$ is positive if and only if there exists a $j_q \in \{1, \dots r\}, j_q \neq i$, such that $(\hat{x}^{j_q})_q>0$ and $c_{x,q}^{i,j_q}>0$, or with words, there is a region $j_q$ where the $q$th infected class is in a positive steady state in $EE^0$, and there is a connection from that class toward the $q$th infected class of region $i$ (we remark that $(\hat{x}^{j_q})_q>0$ implies $\hat{x}^{j_q}>0$, yielding that the region $j_q$ is EAT). We state two theorems.

\begin{theorem}\label{th:r0<1}
Assume that there is a region $i$, $i \in \{1, \dots r\}$, which is DFAT in the boundary endemic equilibrium $EE^0$ of system $(L_1)$--$(L_r)$. Then for the function $f_{\hat{x}^i}$ defined in Theorem  \ref{th:iftendemic} it is satisfied that $\frac{\ud f_{\hat{x}^i}}{\ud \alpha} (0) \geq 0$ if $\rn^i<1$. Furthermore, if we assume that $\sum_{\substack{j=1 \\j \neq i}}^r  C_x^{ij} \hat{x}^j>0$, then it follows that $\frac{\ud f_{\hat{x}^i}}{\ud \alpha} (0) > 0$.
\end{theorem}

\begin{proof}
From the properties of $V^{i}$ described in section \ref{sec:model} and the non-negativity of $F^{i}$ we get that $(V^{i}-F^{i})_{p,q}\leq 0$  holds for $p \neq q$, hence $(V^{i}-F^{i})$ has the Z sign pattern. Theorem 5.1 in \cite{friedler} says that $V^{i}-F^{i}$ is invertible and $(V^{i}-F^{i})^{-1} \geq 0$ if and only if all eigenvalues of $V^{i}-F^{i}$ have positive real part (properties 11 and 18 are equivalent); or analogously, $F^{i}-V^{i}$ is invertible and $(V^{i}-F^{i})^{-1} \geq 0$ if and only if all eigenvalues of $F^{i}-V^{i}$ have negative real part. We follow \cite{arino} and \cite{vddwat} and claim that, for all eigenvalues of $F^{i}-V^{i}$ to have negative real part it is necessary and sufficient that the spectral radius of $F^{i}\cdot(V^{i})^{-1}$ --- which is the local reproduction number $\rn^i$ --- is less than unity.\\
We conclude that if $\rn^i<1$ holds then the equality
\begin{equation}\aligned
\frac{\ud f_{\hat{x}^i}}{\ud \alpha} (0)&=\left(V^{i}-F^i\right)^{-1} \left(\sum_{\substack{j=1 \\j \neq i}}^r C_x^{ij} \hat{x}^j\right)
\nonumber\endaligned\end{equation}
derived from (\ref{eq:derx}) shows that $\frac{\ud f_{\hat{x}^i}}{\ud \alpha} (0)$ is nonnegative. If the sum on the right hand side is strictly positive (which is possible since $EE^0$ is an endemic equilibrium hence there is a region $j \in \{1, \dots r\}$, $j \neq i$, where $\hat{x}^j > 0$; furthermore the  matrix $C_x^{ij}$ is also nonnegative), then $\det (V^{i}-F^{i})^{-1} \neq 0$ yields $\frac{\ud f_{\hat{x}^i}}{\ud \alpha} (0)>0$. The proof is complete.
\end{proof}

\begin{theorem}\label{th:r0>1}
Assume that there is a region $i$, $i \in \{1, \dots r\}$, which is DFAT in the endemic equilibrium $EE^0$ of system $(L_1)$--$(L_r)$. If $\rn^i>1$, then for the function $f_{\hat{x}^i}$ defined in Theorem  \ref{th:iftendemic} it is satisfied that $\frac{\ud f_{\hat{x}^i}}{\ud \alpha} (0)$ has a non-positive component.  Furthermore, if we assume that $\sum_{\substack{j=1 \\j \neq i}}^r  C_x^{ij} \hat{x}^j>0$, then it holds that $\frac{\ud f_{\hat{x}^i}}{\ud \alpha} (0)$ has a strictly negative component.
\end{theorem}
\begin{proof}
Theorems 5.3 and 5.11 in \cite{friedler} state that if $A$ is a square matrix  which satisfies $(A)_{p,q} \leq 0$ for $p \neq q$ and if there exists a vector $x>0$ such that $Ax \geq 0$, then it holds that every eigenvalue of $A$ has nonnegative real part. It is known \cite{vddwat} that all eigenvalues of the matrix $F^{i}-V^{i}$ have negative real part if and only if $\rn^i< 1$, the maximum real part of the eigenvalues is zero if and only if $\rn^i= 1$, and there is an eigenvalue with strictly positive real part if and only if $\rn^i>1$. Hence, using the above result from \cite{friedler} with $A=V^{i}-F^{i}$ and the non-negativity of the right hand side of (\ref{eq:derx}) we get that if $\rn^i > 1$ then there exists no positive vector $x$ such that $(V^{i}-F^{i})x \geq 0$ since $V^{i}-F^{i}$ has an eigenvalue with negative real part. This implies the first statement of the theorem. \\
Theorem 5.1 in \cite{friedler} yields that there is no $x \geq 0$ such that $(V^{i}-F^{i})x >0$; it follows from the equivalence of properties 1 and 18 of Theorem 5.1 that for the existence of such $x$ all eigenvalues of $V^{i}-F^{i}$ should have positive real part. If we now suppose that the last assumption of our statement holds, which ensures the positivity of the right hand side of (\ref{eq:derx}), then we get that $\frac{\ud f_{\hat{x}^i}}{\ud \alpha} (0)$ should satisfy an inequality of the form $(V^{i}-F^{i}) x >0$, which in the light of the argument above is only possible if  $\frac{\ud f_{\hat{x}^i}}{\ud \alpha} (0)$ has a negative component.
\end{proof}

Theorems \ref{th:r0<1} and \ref{th:r0>1} together with Lemmas \ref{lem:suffforpos} and \ref{lem:firstder} give conditions for the persistence of endemic equilibria in system $(T_1)$ -- $(T_r)$ for small volumes of travel. If the fixed point $EE^0$ is a boundary endemic equilibrium of system $(L_1)$--$(L_r)$ with a DFAT region $i$ (that is, $\hat{x}^i=f_{\hat{x}^i}(0)=0$) but, once traveling is introduced, to every infected class in $i$ there is an inflow from another region which is EAT (i.e., if the right hand side of equation (\ref{eq:derx}) is positive), then $f(\alpha)$, $f(0)=EE^0$, leaves the nonnegative cone of  $\R^{r(n+m+k)}$ if $\rn^i>1$, since $\frac{\ud f_{\hat{x}^i}}{\ud \alpha} (0)$ has a negative component and hence, so does $f_{\hat{x}^i} (\alpha)$ for small $\alpha$-s. On the other hand, if for every DFAT region $i$, $i \in \{1, \dots r\}$, it holds that the local reproduction number is less than one, and to each infected class there is an inflow from an EAT region by means of individuals' movement, then $\frac{\ud f_{\hat{x}^i}}{\ud \alpha} (0)>0$ for each such $i$ implies that the endemic equilibrium is preserved in system $(T_1)$ -- $(T_r)$ when $\alpha$ is small.\\

We understand that there is a limitation in applying the results of the above stated theorems: to decide whether an endemic steady state of the disconnected system continues to exist in the system with traveling, we need to know the structure of the connecting network and require the pretty restrictive property that for each $i\in  \{1, \dots r\}$ with $\hat{x}^i=0$, for each $q \in \{1, \dots n\}$ there exists a $j_q \in \{1, \dots r\}, j_q\neq i$, such that $ (\hat{x}^{j_q})_q>0$ and $c_{x,q}^{i,j_q}>0$. In the next section we turn our attention to the case when this property doesn't hold, that is, there is a region $i$ which is DFAT and the right hand side of (\ref{eq:derx}) is not positive (nevertheless we emphasize that, considering the biological interpretation of the sum, it is always nonnegative). This section is closed with a corollary which summarizes our findings. The result covers the special case when the connecting network of all infected classes is a complete network.

\begin{corollary}\label{cor:summaryee}
Consider a boundary endemic equilibrium $EE^0$ of system $(L_1)$--$(L_r)$. Assume that $\sum_{\substack{j=1 \\j \neq i}}^r C_x^{ij} \hat{x}^j>0$ is satisfied whenever $i$, $i\in \{1, \dots r\}$, is a DFAT region in $EE^0$; we note that this condition always holds if the constant $c_{x,q}^{j,l}$ is positive for every $j,l \in \{1, \dots r\}$ and $q \in \{1, \dots n\}$, meaning that all possible connections are established between the infected compartments of the regions. Then, in case $\rn^i<1$ holds in all DFAT regions $i$  we get that $EE^0$ is preserved for small volumes of traveling by a unique function which depends continuously on $\alpha$. If there exists a region $i$ which is EAT and where $\rn^i>1$ then $EE^0$ moves out of the feasible phase space when traveling is introduced.
\end{corollary}

\section{The role of irreducibility of $V^{i}-F^{i}$}\label{sec:irred}
Knowing the steady states of the disconnected system $(L_1)$--$(L_r)$, we are interested in the effect of incorporating the possibility of individuals' movement on the equilibria. The differential system of connected regions $(T_1)$--$(T_r)$ reduces to $(L_1)$--$(L_r)$ when the general mobility parameter $\alpha$ equals zero, thus whenever the Jacobian of $(T_1)$--$(T_r)$ evaluated at an equilibrium of $(L_1)$--$(L_r)$ and $\alpha=0$, $\left(\frac{\partial \mathcal{T}}{\partial \mathcal{X}}\right) (0,EE^0)$, is nonsingular, the existence of a fixed point $f(\alpha)$, $f(0)=EE^0$, in $(T_1)$--$(T_r)$ is guaranteed for small $\alpha$-s by the implicit function theorem. Theorem \ref{th:posee} implies that if  $EE^0$ is a positive steady state of $(L_1)$--$(L_r)$ then so is $f(\alpha)$ in $(T_1)$--$(T_r)$. On the other hand in case $EE^0$ is a boundary endemic equilibrium and $\hat{x}^i=f_{\hat{x}^i}(0)=0$ holds for some $i \in \{1 \dots r\}$, meaning that region $i$ is at disease free state (DFAT) when the system is disconnected, the continuous dependence of $f$ on $\alpha$ allows that the fixed point might move out of the feasible phase space as $\alpha$ becomes positive. \\

In section \ref{sec:ee} we gave a full picture of the behavior of $f(\alpha)$ for small $\alpha$-s in the case when the condition $\sum_{\substack{j=1 \\j \neq i}}^r C_x^{ij} \hat{x}^j>0$ holds for each region $i$ which is DFAT (for a summary, see Corollary \ref{cor:summaryee}). If this condition is not satisfied, then Theorem \ref{th:r0<1} yields that the derivative $\frac{\ud f_{\hat{x}^i}}{\ud \alpha} (0)$ is nonnegative but may have some zero components if $\rn^i<1$, and though --- following Theorem \ref{th:r0>1} --- it cannot be positive if $\rn^i>1$, it might happen that it is still nonnegative. Following this argument it is clear that the problematic case is when $\frac{\ud f_{\hat{x}^i}}{\ud \alpha} (0) \geq 0$ and either the derivative is identically zero, or it has both positive and zero components. In both situations Lemmas \ref{lem:suffforpos} and \ref{lem:firstder} through equation (\ref{eq:derx}) don't provide enough information to decide whether the boundary endemic equilibrium will be preserved once traveling is incorporated. \\

In this section we investigate the question of under what conditions can the derivative be nonnegative but non-positive, and we recall that this can only happen if the right hand side of (\ref{eq:derx}) is not positive. It is convenient to work with the general equation $(V^{i}-F^{i}) v=u$ where $v, u \geq 0$, which gives (\ref{eq:derx}) for $u=\sum_{\substack{j=1 \\j \neq i}}^r  C_x^{ij} \hat{x}^j$ and $v=\frac{\ud f_{\hat{x}^i}}{\ud \alpha} (0)$. The statement of the next proposition immediately follows from the Z sign pattern property of $V^{i}-F^{i}$.
\begin{proposition}\label{prop:vu}
If $v$ is a nonnegative solution of $(V^{i}-F^{i}) v=u$ with $u \geq 0$, then $v_q=0$ implies $u_q=0$, $q \in \{1, \dots n\}$.
\end{proposition}
\begin{lemma}\label{lem:reduc}
If $v$ is a solution of $(V^{i}-F^{i}) v=u$ with $u \geq 0$ such that $v$ is nonnegative and has both zero and positive components, then the matrix $V^{i}-F^{i}$ is reducible.
\end{lemma}
\begin{proof}
If $v$ consists of zero and positive components then, without loss of generality we can assume that there are $r,s>0$, $r+s=n$ such that $v$ can be represented as $v=(v_1, \dots v_{r},$ $v_{r +1}, \dots v_{r+s})^T$ with $v_1, \dots v_{r}>0$ and $v_{r+1}, \dots v_{r+s}=0$. We decompose $V^{i}-F^{i}$ as
\begin{equation}\aligned
V^{i}-F^{i}=\begin{pmatrix}
R_{r \times r} & S_{r \times s}\\
S_{s \times r} & R_{s \times s}
\end{pmatrix}
\nonumber\endaligned\end{equation}
with the $r \times r, r\times s, s \times r$ and $s \times s$ dimensional matrices $R_{r \times r}, S_{r \times s}, S_{s \times r}$ and $R_{s \times s}$, and derive the equation
\begin{equation}\aligned
S_{s \times r} (v_1, \dots v_{r})^T+R_{s \times s}(v_{r +1}, \dots v_{r+s})^T&=(u_{r +1}, \dots u_{r+s})^T
\nonumber\endaligned\end{equation}
from $(V^{i}-F^{i}) v=u$. According to Proposition \ref{prop:vu} from $v_{r+1}, \dots v_{r+s}=0$ it follows that $u_{r+1}, \dots u_{r+s}=0$, thus the last equation reduces to
\begin{equation}\aligned
S_{s \times r} (v_1, \dots v_{r})^T&=0,
\nonumber\endaligned\end{equation}
which, considering that $S_{s \times r} \leq 0$ and $(v_1, \dots v_{r})^T>0$, immediately implies $S_{s \times r} = 0$ and thus the reducibility of $V^{i}-F^{i}$.
\end{proof}
\subsection{The case when $V^{i}-F^{i}$ is irreducible}
The last lemma has an important implication on equation $(V^{i}-F^{i}) v=u$, as it excludes certain solutions. We will also see that it enables us to answer the question posed at the beginning of this section, namely that the derivative in (\ref{eq:derx}) cannot have both positive and zero but no negative components if $V^{i}-F^{i}$ is irreducible.

\begin{lemma}\label{lem:uveq}
Assume that $V^{i}-F^{i}$ is irreducible. If $u \geq 0$, $u \neq 0$ then $(V^{i}-F^{i}) v=u$ has a unique positive solution if $\rn^i<1$, and it holds that $v \ngeq 0$ if $\rn^i>1$. In the case when $u=0$, $v=0$ is the only solution if $\rn^i<1$, and for $\rn^i>1$ it holds that either $v=0$ or $v$ has a negative component.
\end{lemma}
\begin{proof}
In the proof of Theorem \ref{th:r0<1} we have seen that $(V^{i}-F^{i})^{-1} \geq 0$ if $\rn^i<1$, which implies the uniqueness of $v \geq 0$ in $(V^{i}-F^{i}) v=u$. If $u=0$ then trivially $v=0$, and we use Lemma \ref{lem:reduc} to get that $v>0$ when $u\neq 0$. Similar arguments as in the proof of Theorem \ref{th:r0>1} yield that $v$ has a non-positive component if $\rn^i>1$, but Lemma \ref{lem:reduc} again makes only $v=0$ and $v \ngeq 0$ possible. However $v=0$ is a solution of  $(V^{i}-F^{i}) v=u$ if and only if $u=0$, otherwise $v$ must have a negative component.
\end{proof}
The following theorem and proposition are immediate from Lemma \ref{lem:uveq}. We remark that parts of the results of the theorem are to be found in Theorem 5.9 \cite{friedler}, that is, if $V^{i}-F^{i}$ is irreducible then equation (\ref{eq:derx}) has a positive solution.
\begin{theorem}\label{th:irred}
Assume that there is a region $i$, $i \in \{1, \dots r\}$, which is DFAT in the endemic equilibrium $EE^0$ of system $(L_1)$--$(L_r)$, and $V^{i}-F^{i}$ is irreducible. If  $\sum_{\substack{j=1 \\j \neq i}}^r  C_x^{ij} \hat{x}^j \neq 0$, then for the function $f_{\hat{x}^i}$ defined in Theorem  \ref{th:iftendemic} it is satisfied that $\frac{\ud f_{\hat{x}^i}}{\ud \alpha} (0) > 0$ if $\rn^i<1$, and $\frac{\ud f_{\hat{x}^i}}{\ud \alpha} (0) \ngeq 0$ if $\rn^i>1$. 
\end{theorem}

\begin{proposition}\label{prop:rhszero}
Assume that there is a region $i$, $i \in \{1, \dots r\}$, which is DFAT in the endemic equilibrium $EE^0$ of system $(L_1)$--$(L_r)$, and $V^{i}-F^{i}$ is irreducible. If $\sum_{\substack{j=1 \\j \neq i}}^r  C_x^{ij} \hat{x}^j = 0$, then $\frac{\ud f_{\hat{x}^i}}{\ud \alpha} (0) = 0$ is the only solution if $\rn^i<1$, and in the case when $\rn^i>1$ the derivative $\frac{\ud f_{\hat{x}^i}}{\ud \alpha}(0)$ is either zero or has a negative component.
\end{proposition}
We summarize our findings as follows. We consider every region $i$, $i\in \{1, \dots r\}$, which is DFAT in a boundary endemic equilibrium $EE^0$ of $(L_1)$--$(L_r)$. If the derivative in equation (\ref{eq:derx}) has some zero but no negative components then Lemmas \ref{lem:suffforpos} and \ref{lem:firstder} are insufficient to decide whether the fixed point $f(\alpha)$, for which $f(0)=EE^0$, will be biologically meaningful in the system of connected regions. In the case when $\sum_{\substack{j=1 \\j \neq i}}^r  C_x^{ij} \hat{x}^j \neq 0$ (with words, some infected classes of region $i$ have inflow of individuals from EAT regions), the statement of Theorems \ref{th:r0<1} and \ref{th:r0>1} can be sharpened if the extra assumption of $V^{i}-F^{i}$ being irreducible holds: as pointed out in Theorem \ref{th:irred}, the derivative in equation (\ref{eq:derx}) is positive if $\rn^i<1$, and has a negative component if $\rn^i>1$. Applying the results of Lemmas \ref{lem:suffforpos} and \ref{lem:firstder}, this means that if every DFAT region $i$ has inflow from an EAT region and $V^{i}-F^{i}$ is irreducible in all such regions $i$ then $f(\alpha)$, $f(0)=EE^0$, is a positive steady state of $(T_1)$--$(T_r)$ if $\rn^i<1$, and $f(\alpha)$ is not a biologically meaningful equilibrium if there is a region where $\hat{x}^i=0$ and the local reproduction number is greater than one. For conclusion we state a corollary which is similar to the one at the end of section \ref{sec:ee}.
\begin{corollary}\label{cor:summaryirred}
Consider a boundary endemic equilibrium $EE^0$ of system $(L_1)$--$(L_r)$. Let us assume that $\sum_{\substack{j=1 \\j \neq i}}^r C_x^{ij} \hat{x}^j \neq 0$ is satisfied whenever $i$, $i\in \{1, \dots r\}$, is a DFAT region in $EE^0$; we remark that this situation is realized if each DFAT region has at least one infected class with connection from an EAT region. In addition we also suppose that $V^{i}-F^{i}$ is irreducible for DFAT regions.  Then, in case $\rn^i<1$ holds in all regions $i$ which are DFAT we get that $EE^0$ is preserved for small volumes of traveling by a unique function which depends continuously on $\alpha$. If there exists a region $i$ which is DFAT and  where $\rn^i>1$ then $EE^0$ moves out of the feasible phase space when traveling is introduced.
\end{corollary}

\subsection{What if $V^{i}-F^{i}$ is reducible?}
An $n \times n$ square matrix $A$ is called reducible if the set $\{1, \dots n\}$ can be divided into two disjoint nonempty subsets $\{j_1, \dots j_s\}$ and $\{j_{s+1}, \dots j_n\}$ such that $(A)_{j_p,j_q}=0$ holds whenever $p \in \{1, \dots s\}$ and $q \in \{s+1, \dots n\}$. An equivalent definition is that, with simultaneous row and/or column permutations, the matrix can be placed into a form to have an $s \times (n-s)$  zero block. When an infectious agent is introduced into a fully susceptible population in some region $i$ then --- as pointed out in section \ref{sec:dfe} --- the matrix $F^{i}-V^{i}$ describes  disease propagation in the early stage of the epidemic since the change in the rest of the  population can be assumed negligible during the initial spread. If $F^{i}-V^{i}=$ $-(V^{i}-F^{i})$ is reducible then without loss of generality we can assume that it can be decomposed into 
\begin{equation}\aligned
F^{i}-V^{i}=\begin{pmatrix}
R_{r \times r} & S_{r \times s}\\
S_{s \times r} & R_{s \times s}
\end{pmatrix},
\nonumber\endaligned\end{equation}
where $r=n-s$, the dimensions of the sub-matrices are indicated in lower indexes and $S_{s \times r}$ is the zero matrix. This means that there are $s$ infected classes in region $i$ which have no inflow induced by the other $r=n-s$ infected classes of region $i$ in the initial stage of the epidemic (by the expression ``inflow induced by an infected class'' we mean either transition from the class described by matrix $V^{i}$, or the arrival of new infections generated by the infected class, described by $F^{i}$). \\

In the sequel we will assume that such dynamical separation of the infected classes is not realized in any of the regions, or with other words for each $i$ the matrices $F^{i}$ and $V^{i}$ are defined in the model such that $F^{i}-V^{i}$ is irreducible. The biological consequence of this assumption is that whenever a single infected compartment of a DFAT region imports infection via a link from the corresponding $x$-class of an EAT region then the disease will spread in {\it all} infected classes of the DFAT region, not only in the one which has connection from the EAT region. Furthermore we note that the irreducibility of $F^{i}-V^{i}$ also ensures by means of Lemma \ref{lem:uveq} that the fixed point equation $(F^i-V^i)x^i=0$ of system (\ref{basismodel}) has only componentwise positive solutions besides the disease free equilibrium, which is in conjunction with the assumption made for the equilibria in section \ref{sec:ee}. \\

The criterion on $F^{i}-V^{i}$ being irreducible is satisfied in a wide range of well-known epidemiological models, however we remark that our results obtained in sections \ref{sec:dfe} and \ref{sec:ee} also hold in the general case, i.e., when the matrix is reducible.

\section{When the first derivative doesn't help --- DFAT regions with no connection from EAT regions}\label{sec:nodirect}

We consider an endemic equilibrium $EE^0$ of system $(L_1)$--$(L_r)$, our aim is to investigate the solution $f(\alpha)$ of the fixed point equations of system $(T_1)$--$(T_r)$, for which $f(0)=EE^0$, when $\alpha$ is small but positive. The case of positive fixed points has been treated in Theorem \ref{th:posee}. If $EE^0$ is boundary endemic equilibrium, then we assume that the matrix $V^{i}-F^{i}$ is irreducible for every DFAT region $i$; if for each such $i$ it holds that $\sum_{\substack{j=1 \\j \neq i}}^r  C_x^{ij} \hat{x}^j \neq 0$ then Corollary \ref{cor:summaryirred} describes precisely under what conditions is $f(\alpha)$ a nonnegative steady state. It remains to handle the scenario when there exists a region $i$ which is DFAT but $\sum_{\substack{j=1 \\j \neq i}}^r  C_x^{ij} \hat{x}^j = 0$, that is, the region $i$ is disease free in the disconnected system and so are all the regions which have a direct connection to the infected classes of $i$ in $(T_1)$--$(T_r)$. We emphasize here that under {\it ``direct connection from a region $j$ to $i$''} we doesn't necessarily mean that {\it all} infected classes of $i$ have an inbound link from $j$; in the sequel we will use this term to describe the case when  $C_x^{ij}=\text{diag}(c_{x,1}^{ij}, \dots c_{x,n}^{ij}) \neq 0$, that is, there is an infected compartment of $j$  which is connected to $i$. See Figure \ref{fig:directcon} which further illustrates the definition.\\

Henceforth we proceed with the case when there is a region $i$ which is DFAT in $EE^0$ and has no direct connection from any EAT regions. For such $i$-s Proposition \ref{prop:rhszero}  yields that our approach of investigating the non-negativity of $f(\alpha)$ using Lemma \ref{lem:firstder} and the first derivative from equation (\ref{eq:derx}) fails.  However, we assume that $\frac{\ud f_{\hat{x}^l}}{\ud \alpha} (0) \geq 0$ holds for all DFAT regions where $\hat{x}^l=f_{\hat{x}^l}(0)=0$ and $\sum_{\substack{j=1 \\j \neq l}}^r  C_x^{lj} \hat{x}^j \neq 0$, since if the derivative has a negative component then, as pointed out in Corollary \ref{cor:summaryirred}, $f(\alpha)$ moves out of the feasible phase space when $\alpha$ increases and there is no further examination necessary. First we state a few results for later use.

\begin{proposition}\label{prop:derxn} For any positive integer $N$,  $N\leq r-1$, it holds that
\begin{equation}\aligned
(V^{i}-F^{i})\frac{\ud^N f_{\hat{x}^i}}{\ud \alpha^N}(0)&=N\sum_{\substack{j=1 \\j \neq i}}^r C_x^{ij}\frac{\ud^{N-1} f_{\hat{x}^j}}{\ud \alpha^{N-1}} (0)
\nonumber\endaligned\end{equation}
whenever region $i$, $i \in \{1, \dots r\}$, is DFAT in the boundary equilibrium $EE^0$, and $\frac{\ud^{l} f_{\hat{x}^i}}{\ud \alpha^l}(0)=0$ for every $l <N$.
\end{proposition}
\begin{figure}[t]
\centering
\includegraphics[width=7cm]{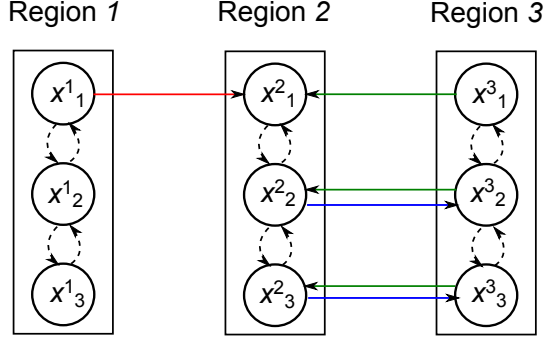}
\caption{We consider three regions with three infected classes ($r=3$, $n=3$). Every infected class of region 2 has an inbound link from region 3 (green arrows). This means that region 2 has direct connection from 3, but 3 also has direct connection from 2 since $c_{x,2}^{32}, c_{x,3}^{32}>0$, that is, there are links from the second and third infected classes of region 2 to the corresponding compartments of region 3 (blue arrows). Region 1 has no direct connection from either 2 or 3, and there is direct connection from region 1 to 2 (red arrow) but not to 3. On the other hand, 3 is reachable from 1 because there is a path from 1 to 3 via region 2. Region 1 is not reachable from any of the other two regions. \label{fig:directcon}}
\end{figure}
\begin{proof}
In case $N=1$, the equation in the proposition reads as (\ref{eq:derx}). Let us assume that $N \geq 2$ and $\frac{\ud f_{\hat{x}^i}}{\ud \alpha} (0)=0$. We return to equation (\ref{eq:deralpha}) to obtain the $N$th derivative of the equation of $x^i$ in (\ref{genmodel}) as
\begin{equation}\aligned\label{eq:deralphaN}
\frac{\ud^N}{\ud \alpha^N} \biggl( \mathcal{F}^{i}(f_{\hat{x}^i}(\alpha), f_{\hat{y}^i}(\alpha), f_{\hat{z}^i}(\alpha)) f_{\hat{x}^i}(\alpha)- V^{i} f_{\hat{x}^i}(\alpha)&\\
-\sum_{\substack{j=1 \\j \neq i}}^r \alpha C_x^{ji} f_{\hat{x}^i}(\alpha)+\sum_{\substack{j=1 \\j \neq i}}^r  \alpha C_x^{ij} f_{\hat{x}^j}(\alpha)\biggr)&=\\
\sum_{l=0}^{N} {N \choose l}\frac{\ud^{N-l}}{\ud \alpha^{N-l}} \biggl( \mathcal{F}^{i}(f_{\hat{x}^i}(\alpha), f_{\hat{y}^i}(\alpha), f_{\hat{z}^i}(\alpha))\biggr) \frac{\ud^l f_{\hat{x}^i}}{\ud \alpha^l} (\alpha)&\\
- V^{i} \frac{\ud^N f_{\hat{x}^i}}{\ud \alpha^N}(\alpha)-\sum_{l=0}^{N} {N \choose l}\sum_{\substack{j=1 \\j \neq i}}^r \frac{\ud^{N-l} (\alpha C_x^{ji})}{\ud \alpha^{N-l}} \cdot \frac{\ud^{l} f_{\hat{x}^i}}{\ud \alpha^{l}}(\alpha)&\\
+\sum_{l=0}^{N} {N \choose l} \sum_{\substack{j=1 \\j \neq i}}^r \frac{\ud^{N-l} (\alpha C_x^{ij})}{\ud \alpha^{N-l}}\cdot \frac{\ud ^{l} f_{\hat{x}^j}}{\ud \alpha^{l}}(\alpha)&=0.
\endaligned\end{equation}
As $f_{\hat{y}^i}(\alpha)>0$, it is satisfied by assumption that $\mathcal{F}^{i}$ is $r-1$ times continuously differentiable in the respective point. Clearly $\frac{\ud^{N-l} (\alpha C_x^{ij})}{\ud \alpha^{N-l}}=0$ whenever $N-l\geq 2$, moreover $\frac{\ud (\alpha C_x^{ij})}{\ud \alpha}=C_x^{ij}$, so if $\frac{\ud^{l} f_{\hat{x}^i}}{\ud \alpha^{l}} (0)= 0$ holds for all $l < N$ then (\ref{eq:deralphaN}) at $\alpha=0$ reads
\begin{equation}\aligned\label{eq:derxN}
(V^{i}-F^{i})\frac{\ud^N f_{\hat{x}^i}}{\ud \alpha^N}(0)&=N\sum_{\substack{j=1 \\j \neq i}}^r C_x^{ij}\frac{\ud^{N-1} f_{\hat{x}^j}}{\ud \alpha^{N-1}} (0)
\endaligned\end{equation}
since $(f_{\hat{x}^i}(0), f_{\hat{y}^i}(0), f_{\hat{z}^i}(0))=(0, y_0^i,0)$ and $F^{i}=\mathcal{F}(0, y_0^i,0)$. 
\end{proof}

Our interpretation of the term {\it ``direct connection from a region $j$ to the infected classes of $i$''} can be extended to the expression {\it ``path from a region $j$ to the infected classes of $i$''}, representing a chain of direct connections via other regions, starting at $j$ and ending in $i$. Figure \ref{fig:directcon} provides an example for three regions, where there is a path from region 1 to 3 via 2 (this is, $c_{x,1}^{21}, c_{x,2}^{32}>0$). We note, however, that the path doesn't necessarily consist of the same type of infected classes in the regions: in terms of the above example, infection imported to region 2 via the link from $x^1_1$ to $x^2_1$ spreads in other infected classes of region 2 as well by means of the irreducibility of $V^2-F^2$ (represented by dashed arrows in the figure), enabling the disease to reach region 3 via the links from $x^2_2$ to $x^3_2$ and from $x^2_3$ to $x^3_3$.  We also remark that the notation ``path from a region $j$ to the infected classes of $i$'' includes the special case when the path consists of $i$ and $j$ only, i.e., there is a direct connection from $j$ to $i$. We now define the shortest distance from EAT regions to a DFAT region.
\begin{definition}\label{def:mi}
Consider a region $i$ which is DFAT in the boundary endemic equilibrium $EE^0$. We define $M_i$ as the least nonnegative integer such that in system $(T_1)$--$(T_r)$ there is a path starting with an EAT region $j$, ending with region $i$ and containing $M_i$ regions in-between. If there is no such path then let $M_i= r-1$.
\end{definition}
If there is a direct connection from an EAT region $j$ to the infected classes of $i$ then this definition implies $M_i=0$. We also note that $M_i \leq r-2$ always holds whenever the path described above exists. In the sequel we omit the words ``infected classes'' from the expression ``direct connection (path) for $j$ to $i$'' for convenience. Clearly infection from endemic regions to disease free territories are never imported via links between non-infected compartments of different regions, so to decide whether the disease arrives to a region it is enough to know the graph connecting infected compartments.  
\begin{lemma}\label{lem:derivativezero}
Assume that $f_{\hat{x}^j}(\alpha) \geq 0$ is satisfied on an interval $[0,\alpha^*)$ whenever a region $j$, $j \in \{1, \dots r\}$, is DFAT in the boundary endemic equilibrium $EE^0$. Then for any DFAT region $i$, $i \in \{1, \dots r\}$, it holds that $\frac{\ud^{l} f_{\hat{x}^i}}{\ud \alpha^l}(0)=0$ for $l \leq M_i$.
\end{lemma}

\begin{proof}
The inequality $M_{i_0} \geq 0$ is satisfied in every region $i_0$ with $\hat{x}^{i_0}=0$. The case when $M_{i_0}=0$ is trivial, so we consider a region $i_1$ for which $M_{i_1} \geq 1$, and using that $\hat{x}^{i_1}=0$ we derive
\begin{equation}\aligned
(V^{i_1}-F^{i_1})\frac{\ud f_{\hat{x}^{i_1}}}{\ud \alpha}(0)&=\sum_{\substack{j=1 \\j \neq i_1}}^r C_x^{{i_1},j} f_{\hat{x}^j} (0),
\nonumber\endaligned\end{equation}
which is similar to equation (\ref{eq:derx}). For every $j$ such that $C_x^{{i_1},j} \neq 0$ it follows from  $M_{i_1} \neq 0$ that $f_{\hat{x}^j}(0)= 0$, thus the right hand side is zero. Lemma \ref{lem:uveq} yields that $\frac{\ud f_{\hat{x}^{i_1}}}{\ud \alpha}(0) $ is either zero (in case $\rn^{i_1}<1$ this is the only possibility) or has a negative component (this can be realized only if $\rn^{i_1}>1$). Nevertheless, the derivative having a negative component together with $\hat{x}^{i_1}= 0$ contradicts the assumption that $f_{\hat{x}^{i_1}}(\alpha) \geq 0$ for small $\alpha$-s, this observation makes  $\frac{\ud f_{\hat{x}^{i_1}}}{\ud \alpha}(0)=0$ the only possible case.\\

Next consider a region $i_2$ where $\hat{x}^{i_2}= 0$ and $M_{i_2} \geq 2$. We have $\frac{\ud f_{\hat{x}^{i_2}}}{\ud \alpha}(0)=0 $ since $M_{i_2} \geq 2 \geq 1$, so Proposition \ref{prop:derxn} yields the equation
\begin{equation}\aligned
(V^{i_2}-F^{i_2})\frac{\ud^2 f_{\hat{x}^{i_2}}}{\ud \alpha^2}(0)&=2 \sum_{\substack{j=1 \\j \neq i_2}}^r C_x^{{i_2},j} \frac{\ud f_{\hat{x}^j}}{\ud \alpha} (0).
\nonumber\endaligned\end{equation}
We note that each region $j$ for which $C_x^{{i_2},j} \neq 0$ is DFAT since $M_{i_2} \geq 1$. Thus, for $M_j$ it follows that $M_j \geq 1$, henceforth $\frac{\ud f_{\hat{x}^j}}{\ud \alpha} (0)= 0$ holds by induction, and the right hand side of the last equation is zero. Using Lemma \ref{lem:uveq} there are again two possibilities for $\frac{\ud^2 f_{\hat{x}^{i_2}}}{\ud^2 \alpha}(0)$, namely that it is either zero or has a negative component; but $f_{\hat{x}^{i_2}}(0)=0$, $\frac{\ud f_{\hat{x}^{i_2}}}{\ud \alpha}(0)=0$ and $\frac{\ud^2 f_{\hat{x}^{i_2}}}{\ud^2 \alpha}(0) \ngeq 0$ would imply the existence of an $\alpha^{**}$ such that $f_{\hat{x}^{i_2}}(\alpha)\ngeq 0$ for $\alpha <\alpha^{**}$ which is impossible. We conclude that $\frac{\ud^2 f_{\hat{x}^{i_2}}}{\ud \alpha^2}(0)=0$ holds for all regions where $M_{i_2} \geq 2$. \\
The continuation of these procedure yields that $\frac{\ud^l f_{\hat{x}^{i_l}}}{\ud \alpha^l}(0)=0$ for any region $i_l$ where $M_{i_l} \geq l$ holds. This proves the lemma.
\end{proof}
We say that {\it region $i$ is reachable from region $j$} if there is a path from (the infected classes of) $j$ to (the infected classes of) $i$. Directly connected regions are clearly reachable. Now we are in the position to prove one of the main results of this section.
\begin{theorem}\label{th:pluszeroro>1}
Assume that in the boundary endemic equilibrium $EE^0$ there is a region $i$ which is DFAT and for which $\rn^i>1$ holds, furthermore $i$ is reachable from an EAT region. Then there is an $\alpha^* >0$ such that $f(\alpha)$ has a negative component for $\alpha \in (0,\alpha^*)$, meaning that $f(0)=EE^0$ moves out of the feasible phase space when traveling is introduced.
\end{theorem}

\begin{proof}
The proof is by contradiction. We assume that $EE^0$ is such that there are regions $i_0$ and $i_+$ where $\hat{x}^{i_0}=0$, $\hat{x}^{i_+}>0$, $\rn^{i_0}>1$ and $i_0$ is reachable from $i_+$, moreover there exists an $\alpha^{**}>0$ such that $f(\alpha) \geq 0$ for $0 \leq \alpha \leq \alpha^{**}$, this is, the equilibrium $EE^0=f(0)$ of the disconnected system remains biologically meaningful in the system with traveling. This also means that for all $j$ with $\hat{x}^j=0$ it necessarily holds that $\frac{\ud f_{\hat{x}^{j}}}{\ud \alpha}(0)\geq 0$. \\

If regions $i_0$ and $i_+$, as described above, exist then there is a minimal distance between such regions, this is, there exists a least nonnegative integer $L\leq r-2$ such that there is a path (connecting infected compartments of regions) from an EAT region via $L$ regions to a region which is DFAT in $(L_1)$--$(L_r)$. In the case when $L=0$ Theorem \ref{th:irred} immediately yields contradiction, so we can assume that $L \geq 1$. We label the regions which are part of the minimal-length path by $i$, $i_1^{*}$, $\dots$ $i_{L}^{*}$, $i_{L+1}^{*}$, where $\hat{x}^i=\hat{x}^{i_1^{*}}= \dots \hat{x}^{i_{L}^{*}}=0$, $\hat{x}^{i_{L+1}^{*}}> 0$, moreover note that $\rn^i>1$ and $\rn^{i_j^*}<1$ hold for $j = 1, \dots L$. See the path depicted in Figure \ref{fig:Lpath} in the Appendix.\\

The fact that $\hat{x}^{i_{L}^{*}}=f_{\hat{x}^{i_{L}^{*}}}(0)=0$ gives
\begin{equation}\aligned
(V^{i_{L}^{*}}-F^{i_{L}^{*}})\frac{\ud f_{\hat{x}^{i_{L}^{*}}}}{\ud \alpha}(0)&=\sum_{\substack{j=1 \\j \neq i_{L}^{*}}}^r C_x^{{i_{L}^{*}},j} f_{\hat{x}^j} (0)
\nonumber\endaligned\end{equation}
by Proposition \ref{prop:derxn}. The equation has a non-zero right hand side since $\hat{x}^{i_{L+1}^{*}}=f_{\hat{x}^{i_{L+1}^{*}}} (0)>0$, so Lemma \ref{lem:uveq} and $\rn^{i_{L}^{*}}<1$ imply $\frac{\ud f_{\hat{x}^{i_{L}^{*}}}}{\ud \alpha}(0)>0$. A similar equation
\begin{equation}\aligned
(V^{i_{L-1}^{*}}-F^{i_{L-1}^{*}})\frac{\ud f_{\hat{x}^{i_{L-1}^{*}}}}{\ud \alpha}(0)&=\sum_{\substack{j=1 \\j \neq i_{L-1}^{*}}}^r C_x^{{i_{L-1}^{*}},j} f_{\hat{x}^j} (0)
\nonumber\endaligned\end{equation}
follows from $\hat{x}^{i_{L-1}^{*}}=0$. We note that $M_{i_{L-1}^{*}}= 1$, where $M$ was defined in Definition \ref{def:mi}, hence $f_{\hat{x}^j} (0)=0$ holds for every $j$ such that $C_x^{{i_{L-1}^{*}},j} \neq 0$.  The zero right hand side, Lemma \ref{lem:uveq} and $\rn^{i_{L-1}^{*}}<1$ yield $\frac{\ud f_{\hat{x}^{i_{L-1}^{*}}}}{\ud \alpha}(0)=0$, so we can apply Proposition \ref{prop:derxn} to derive
\begin{equation}\aligned
(V^{i_{L-1}^{*}}-F^{i_{L-1}^{*}})\frac{\ud^2 f_{\hat{x}^{i_{L-1}^{*}}}}{\ud \alpha^2}(0)&=2 \sum_{\substack{j=1 \\j \neq {i_{L-1}^{*}}}}^r C_x^{{i_{L-1}^{*}},j} \frac{\ud f_{\hat{x}^j}}{\ud \alpha} (0).
\nonumber\endaligned\end{equation}
If there is a $j$ such that $C_x^{{i_{L-1}^{*}},j} \neq 0$ and $\frac{\ud f_{\hat{x}^j}}{\ud \alpha} (0) \ngeq 0$ then $f_{\hat{x}^j} (0)=0$ would mean that for small $\alpha$-s $f_{\hat{x}^j} (\alpha)$ has a negative component and $f(\alpha)$, $f(0)=EE^0$, is not in the nonnegative cone, which violates our assumption that $f(\alpha)\geq 0$ for $\alpha$ sufficiently small. Thus each such derivative is necessarily nonnegative, moreover we have showed that $\frac{\ud f_{\hat{x}^{i_{L}^{*}}}}{\ud \alpha}(0)>0$ is satisfied, which makes the right hand side of the last equation positive; this, with the use Lemma \ref{lem:uveq}, implies  $\frac{\ud^2 f_{\hat{x}^{i_{L-1}^{*}}}}{\ud^2 \alpha}(0)>0$ since $\rn^{i_{L-1}^{*}}<1$.\\ 

Next we consider region $i_{L-2}^{*}$, where $M_{i_{L-2}^{*}}=2$. For any region $j$ for which $C_x^{{i_{L-2}^{*}},j} \neq 0$ it holds that $M_j \geq 1$, thus $f_{\hat{x}^j} (0)=0$ and $\frac{\ud f_{\hat{x}^j}}{\ud \alpha} (0) = 0$ hold by Lemma \ref{lem:derivativezero} and the assumption that $f(\alpha) \geq 0$ for small $\alpha$-s. Thus, the right hand side of equation
\begin{equation}\aligned
(V^{i_{L-2}^{*}}-F^{i_{L-2}^{*}})\frac{\ud f_{\hat{x}^{i_{L-2}^{*}}}}{\ud \alpha}(0)&=\sum_{\substack{j=1 \\j \neq i_{L-2}^{*}}}^r C_x^{{i_{L-2}^{*}},j} f_{\hat{x}^j} (0)
\nonumber\endaligned\end{equation}
is zero, from $\rn^{i_{L-2}^{*}}<1$ and  Lemma \ref{lem:uveq} it follows that $\frac{\ud f_{\hat{x}^{i_{L-2}^{*}}}}{\ud \alpha}(0)=0$ and thus Proposition \ref{prop:derxn} yields
\begin{equation}\aligned
(V^{i_{L-2}^{*}}-F^{i_{L-2}^{*}})\frac{\ud^2 f_{\hat{x}^{i_{L-2}^{*}}}}{\ud \alpha^2}(0)&=2 \sum_{\substack{j=1 \\j \neq i_{L-2}^{*}}}^r C_x^{{i_{L-2}^{*}},j} \frac{\ud f_{\hat{x}^j}}{\ud \alpha} (0).
\nonumber\endaligned\end{equation}
We get again that $\frac{\ud^2 f_{\hat{x}^{i_{L-2}^{*}}}}{\ud \alpha^2}(0)=0$ since, as we have seen above, all derivatives in the right hand side are zero and $\rn^{i_{L-2}^{*}}<1$ also holds, so Lemma \ref{lem:uveq} makes the second derivative of $f_{\hat{x}^{i_{L-2}^{*}}}$ zero.
Finally, using that $\frac{\ud^l f_{\hat{x}^{i_{L-2}^{*}}}}{\ud \alpha^l}(0)=0$ for $l=0,1,2$, we derive
 \begin{equation}\aligned
 (V^{i_{L-2}^{*}}-F^{i_{L-2}^{*}})\frac{\ud^3 f_{\hat{x}^{i_{L-2}^{*}}}}{\ud \alpha^3}(0)&=3 \sum_{\substack{j=1 \\j \neq i_{L-2}^{*}}}^r C_x^{{i_{L-2}^{*}},j} \frac{\ud^2 f_{\hat{x}^j}}{\ud \alpha^2} (0),
 \nonumber\endaligned\end{equation}
where $C_x^{{i_{L-2}^{*}},i_{L-1}^{*}} \neq 0$ and  $\frac{\ud^2 f_{\hat{x}^{i_{L-1}^{*}}}}{\ud \alpha^2}(0)>0$. If there is a $j$, $C_x^{{i_{L-2}^{*}},j} \neq 0$, for which $\frac{\ud^2 f_{\hat{x}^j}}{\ud \alpha^2} (0)$ has a negative component then so does $f_{\hat{x}^j}(\alpha)$ and $f(\alpha)$ for small $\alpha$-s since $\frac{\ud f_{\hat{x}^j}}{\ud \alpha} (0) = 0$ and $f_{\hat{x}^j}(0) = 0$, which is a contradiction. Otherwise the right hand side of the last equation is positive (it holds that $\frac{\ud^2 f_{\hat{x}^{i_{L-1}^{*}}}}{\ud \alpha^2}(0)>0$), thus the positivity of  $\frac{\ud^3 f_{\hat{x}^{i_{L-2}^{*}}}}{\ud \alpha^3}(0)$ follows from $\rn^{i_{L-2}^{*}}<1$ and Lemma \ref{lem:uveq}.\\

Following these arguments one can prove that  $\frac{\ud^{l+1} f_{\hat{x}^{i_{L-l}^{*}}}}{\ud \alpha^{l+1}}(0)>0$ for $l=0, 1, \dots L-1$ (we remark that for $l=L-1$ this reads $\frac{\ud^{L} f_{\hat{x}^{i_{1}^{*}}}}{\ud \alpha^{L}}(0)>0$), and that for any fixed $l$ and $k \leq l$ it holds that $\frac{\ud^{k} f_{\hat{x}^{i_{L-l}^{*}}}}{\ud \alpha^{k}}(0)=0$.
We note that $M_i=L$, which according to Lemma \ref{lem:derivativezero} also means that $\frac{\ud^l f_{\hat{x}^i}}{\ud \alpha^l} (0)=0$ for $l \leq M_i=L$ since $f_{\hat{x}^i}(\alpha) \geq 0$ holds for small $\alpha$-s by assumption. Henceforth we can apply Proposition \ref{prop:derxn} and derive
  \begin{equation}\aligned
   (V^{i}-F^{i})\frac{\ud^{L+1} f_{\hat{x}^{i}}}{\ud \alpha^{L+1}}(0)&=L \sum_{\substack{j=1 \\j \neq i}}^r C_x^{i,j} \frac{\ud^L f_{\hat{x}^j}}{\ud \alpha^L} (0).
   \nonumber\endaligned\end{equation}
$M_i=L$ implies $M_j \geq L-1$ for any $j$ for which $C_x^{i,j} \neq 0$, hence $\frac{\ud^l f_{\hat{x}^j}}{\ud \alpha^l} (0)=0$ is satisfied for $l=0, 1, \dots L-1$.  The assumption $f(\alpha) \geq 0$ for small $\alpha$-s yields  $ f_{\hat{x}^j}(\alpha) \geq 0$ for any region $j$ with $C_x^{i,j} \neq 0$, so $\frac{\ud^L f_{\hat{x}^j}}{\ud \alpha^L} (0) \ngeq 0$ is impossible; this together with $\frac{\ud^{L} f_{\hat{x}^{i_{1}^{*}}}}{\ud \alpha^{L}}(0)>0$ results in the positivity of the right hand side of the above equation. As $\rn^{i}>1$ holds, it follows from Lemma \ref{lem:uveq} that $\frac{\ud^{L+1} f_{\hat{x}^{i}}}{\ud \alpha^{L+1}}(0)$ has a negative component, but we showed that $\frac{\ud^{l} f_{\hat{x}^{i}}}{\ud \alpha^{l}}(0)=0$ when $0\leq l \leq L$, so for small $\alpha$-s $f_{\hat{x}^i} (\alpha) \ngeq 0$ follows, a contradiction. The proof is complete.
\end{proof}

Theorem \ref{th:pluszeroro>1} ensures that, for a boundary endemic equilibrium $EE^0$ of $(L_1)$--$(L_r)$, the point $f(\alpha)$ defined by Theorem \ref{th:iftendemic} with $f(0)=EE^0$ will not be a biologically meaningful fixed point of system $(T_1)$--$(T_r)$ if there is a DFAT region $i$ in $EE^0$ which has local reproduction number greater than one and is reachable from another region which is EAT in $EE^0$. The question, whether the condition $\rn^i>1$ is crucial, comes naturally. We need the following result which is similar to Lemma \ref{lem:derivativezero}.

\begin{lemma}\label{lem:derivzeroro<1}
Assume that in the boundary endemic equilibrium $EE^0$ there is no DFAT region $j$ for which $\rn^j>1$ and $M_j < r-1$. Then for a region $i$ which is DFAT it holds that $\frac{\ud^{l} f_{\hat{x}^{i}}}{\ud \alpha^{l}}(0)=0$ for $l \leq M_i$.
\end{lemma}

\begin{proof}
If $i$ is disease free for $\alpha=0$ and the region is not reachable from any region $j$ with $\hat{x}^j>0$ (that is, $M_i =r-1$), then $i$ doesn't import any infection by means of traveling and hence we have  $f_{\hat{x}^i}(\alpha)=0$ for all $\alpha>0$. This also means that $\frac{\ud^{l} f_{\hat{x}^{i}}}{\ud \alpha^{l}}(0)=0$ holds for all $0 \leq l \leq r-1$. The case when $M_i=0$ is trivial, for $1 \leq M_i < r-1$ we use the method of induction. \\

We claim that for any $1 \leq l \leq r-2$ it holds that $\frac{\ud^{l} f_{\hat{x}^{i_l}}}{\ud \alpha^{l}}(0)=0$ whenever a region $i_l$ is such that $\hat{x}^{i_l}=0$, $\rn^{i_l}<1$ and $M_{i_l} \geq l$. If so, the statement of the lemma follows for region $i$ with the choice of $i:=i_l$ for $l=1, 2, \dots M_i$. For a region $i_1$ where $\hat{x}^{i_1}=0$, $M_{i_1} \geq 1$ and $\rn^{i_1}<1$, we get $\frac{\ud f_{\hat{x}^{i_1}}}{\ud \alpha}(0)=0$ from
  \begin{equation}\aligned
   (V^{i_1}-F^{i_1})\frac{\ud f_{\hat{x}^{i_1}}}{\ud \alpha}(0)&= \sum_{\substack{j=1 \\j \neq i_1}}^r C_x^{i_1,j} f_{\hat{x}^j}(0)
   \nonumber\endaligned\end{equation}
and Lemma \ref{lem:uveq} since the right hand side is zero because of $M_{i_1} \geq 1$. Let us assume that there exists an $L < r-2$ such that the statement holds for all $l \leq L$. We consider a region $i_{L+1}$ where  $\hat{x}^{i_{L+1}}=0$, $\rn^{i_{L+1}}<1$ and $M_{i_{L+1}} \geq L+1$, here clearly $M_{i_{L+1}} \geq 1, 2, \dots L$ so $\frac{\ud f_{\hat{x}^{i_{L+1}}}}{\ud \alpha}(0)=$ $\frac{\ud^{2} f_{\hat{x}^{i_{L+1}}}}{\ud \alpha^{2}}(0)=\dots =$ $\frac{\ud^{L} f_{\hat{x}^{i_{L+1}}}}{\ud \alpha^{L}}(0)=0$ holds and thus Proposition \ref{prop:derxn} yields
 \begin{equation}\aligned
   (V^{i_{L+1}}-F^{i_{L+1}})\frac{\ud^{L+1} f_{\hat{x}^{i_{L+1}}}}{\ud \alpha^{L+1}}(0)&=(L+1) \sum_{\substack{j=1 \\j \neq i_{L+1}}}^r C_x^{i_{L+1},j} \frac{\ud^L f_{\hat{x}^j}}{\ud \alpha^L} (0).
   \nonumber\endaligned\end{equation}
For any $j$ with $C_x^{i_{L+1},j} \neq 0$ it holds that the region is DFAT and $M_j\geq M_{i_{L+1}} -1 \geq L$, thus $\frac{\ud^L f_{\hat{x}^j}}{\ud \alpha^L} (0)=0$ makes the right hand side zero, and using Lemma \ref{lem:uveq} we get that $\frac{\ud^{L+1} f_{\hat{x}^{i_{L+1}}}}{\ud \alpha^{L+1}}(0)=0$ since $\rn^{i_{L+1}}<1$.
\end{proof}

The next theorem is the key to answer the question stated earlier, that is, an endemic equilibrium $EE^0$ of $(L_1)$--$(L_r)$ will persist in the system of connected regions via the uniquely defined function $f(\alpha)$, $f(0)=EE^0$, for small volumes of traveling if $\rn^i<1$ holds in all DFAT regions of $EE^0$ which are reachable from an EAT region. In what follows, we prove that $f_{\hat{x}^{i}}$ has a positive derivative whenever region $i$ is DFAT with local reproduction number less than one, and reachable from a region $j$ which is EAT. Then, with the help of Lemma \ref{lem:derivzeroro<1}, the statement yields that $f_{\hat{x}^{i}}(\alpha)$ is positive for small $\alpha$-s, and thus so is $f(\alpha)$ by Lemma \ref{lem:suffforpos}.

\begin{theorem}\label{th:pluszeroro<1}
Assume that in the boundary endemic equilibrium $EE^0$ there is no DFAT region $j$ for which $\rn^j>1$ and $M_j < r-1$. Then for a DFAT region $i$ where $\rn^i<1$, it holds that $\frac{\ud^{M_i+1} f_{\hat{x}^{i}}}{\ud \alpha^{M_i+1}}(0)>0$ if $M_i < r-1$.
\end{theorem}

\begin{proof}

The proof is by induction. For any $i_0$ such that $\hat{x}^{i_0}=0$, $\rn^{i_0}<1$ and $M_{i_0}=0$, Theorem \ref{th:irred} yields $\frac{\ud f_{\hat{x}^{i_0}}}{\ud \alpha}(0)>0$. Whenever $M_{i_1}=1$ is satisfied in a region $i_1$ where $\hat{x}^{i_1}=0$ and $\rn^{i_1}<1$, Lemma \ref{lem:derivzeroro<1} implies $\frac{\ud f_{\hat{x}^{i_1}}}{\ud \alpha}(0)=0$, so using Proposition \ref{prop:derxn} we derive
  \begin{equation}\aligned
   (V^{i_1}-F^{i_1})\frac{\ud^2 f_{\hat{x}^{i_1}}}{\ud \alpha^2}(0)&= 2\sum_{\substack{j=1 \\j \neq i_1}}^r C_x^{i_1,j} \frac{\ud f_{\hat{x}^{j}}}{\ud \alpha}(0).
\nonumber\endaligned\end{equation}
For every $j$ with $C_x^{i_1,j} \neq 0$ it holds that $M_j \geq 0$ (we remark that $M$ is well-defined for such regions because $M_{i_1} \neq 0$ implies that all such $j$-s are DFAT regions); if either $M_j=r-1$ (this always holds if $\rn^j>1$) or $1 \leq M_j < r-1$ then Lemma \ref{lem:derivzeroro<1} gives $\frac{\ud f_{\hat{x}^{j}}}{\ud \alpha}(0)=0$, and whenever $M_j=0$ then necessarily $\rn^j<1$ so $\frac{\ud f_{\hat{x}^{j}}}{\ud \alpha}(0)>0$ holds by induction. Nevertheless, the positivity of the right hand side of the last equation is guaranteed because we know from $M_{i_1}=1$ that there must exist a $j$ with $M_j=0$ and $\rn^j<1$, hence the inequality $\frac{\ud^{2} f_{\hat{x}^{i_1}}}{\ud \alpha^{2}}(0)>0$ follows using Lemma \ref{lem:uveq}.\\

We assume that the statement of the theorem holds for an $L$, $0<L<r-2$, that is, $\frac{\ud^{L+1} f_{\hat{x}^{i_L}}}{\ud \alpha^{L+1}}(0)>0$ if $M_i=L$, $\hat{x}^{i_L}=0$ and $\rn^{i_L}<1$. We take a region $i_{L+1}$, $M_{i_{L+1}}=L+1$, $\hat{x}^{i_{L+1}}=0$ and $\rn^{i_{L+1}}<1$, and obtain the equation
 \begin{equation}\aligned
   (V^{i_{L+1}}-F^{i_{L+1}})\frac{\ud^{L+2} f_{\hat{x}^{i_{L+1}}}}{\ud \alpha^{L+2}}(0)&=(L+1) \sum_{\substack{j=1 \\j \neq i_{L+1}}}^r C_x^{i_{L+1},j} \frac{\ud^{L+1} f_{\hat{x}^j}}{\ud \alpha^{L+1}} (0)
   \nonumber\endaligned\end{equation}
using Lemma \ref{lem:derivzeroro<1} and Proposition \ref{prop:derxn}. $M_{i_{L+1}}=L+1$ makes $M_j \geq L$ for each $j$ where $C_x^{i_{L+1},j}\neq 0$, and by examining the derivatives on the right hand side of this equation we get from Lemma \ref{lem:derivzeroro<1} that $\frac{\ud^{L+1} f_{\hat{x}^j}}{\ud \alpha^{L+1}} (0)=0$ for each $j$, $ C_x^{i_{L+1},j}\neq 0$, whenever $M_j \geq L+1$.  The case when $M_j=L$ is only possible if $\rn^j<1$, and for all such $j$-s the inequality $\frac{\ud^{L+1} f_{\hat{x}^j}}{\ud \alpha^{L+1}} (0)>0$ holds by induction. Hence, the right hand side of the last equation is positive because all the derivatives in it are nonnegative and $M_{i_{L+1}}=L+1$ implies there is a $j$ with $M_j=L$. We apply Lemma \ref{lem:uveq} to get that $\frac{\ud^{L+2} f_{\hat{x}^{i_{L+1}}}}{\ud \alpha^{L+2}}(0)>0$, which completes the proof.
\end{proof}

Let us now summarize what we have learned about steady states of system $(T_1)$--$(T_r)$ for small volumes of traveling (represented by the parameter $\alpha$) between the regions. With some conditions on the model equations described in Theorems \ref{th:iftdfe} and \ref{th:iftendemic}, for every equilibrium of the disconnected system there exists a unique continuous function of $\alpha$ on an interval to the right of zero, which satisfies the fixed point equations of $(T_1)$--$(T_r)$. As discussed in Theorems \ref{th:iftdfe} and \ref{th:posee}, $f_0$ corresponding to the unique disease free equilibrium of $(L_1)$--$(L_r)$ defines a disease free fixed point for $\alpha \in [0,\alpha_0)$, moreover if $f(0)$ is positive then $f(\alpha)>0$ holds for $\alpha$ sufficiently close to zero. With other words the connected system $(T_1)$--$(T_r)$ admits a single infection-free equilibrium and also several positive fixed points for small $\alpha$-s, regardless of  the connections between the regions.\\ 

On the other hand, the structure of the connection network plays an important role when considering boundary endemic equilibria, i.e., when some regions are disease free for $\alpha=0$. If there are regions $i$ and $j$ such that $i$ is reachable from $j$ then, by increasing $\alpha$ the fixed point $f(\alpha)$ moves out of the nonnegative cone whenever $f(0)=EE^0$ is such that $\hat{x}^i=0$, $\rn^i>1$, and $\hat{x}^j>0$, this is,  $j$ is an EAT region and $i$ is a DFAT region with local reproduction number greater than one. However, a boundary equilibrium of the disconnected system will persist through $f$ for small volumes of traveling in $(T_1)$--$(T_r)$ if the local reproduction number is less than one in all DFAT regions which are reachable from EAT regions. These last conclusions are stated below in the form of a corollary as well.
\begin{corollary}\label{cor:summarypluszero}
Consider a boundary endemic equilibrium $EE^0$ of system $(L_1)$--$(L_r)$. Assume that there is a DFAT region $i$ in $EE^0$ with $\rn^i>1$, and $i$ is reachable from a region which is EAT. Then $EE^0$ moves out of the feasible phase space when traveling is introduced. On the other hand, if there is no such region $i$ in the system, then $EE^0$ is preserved for small volumes of traveling, and given by a unique function which depends continuously on $\alpha$.
\end{corollary}

\section{Application to an HIV model on three patches}\label{sec:HIV}

\begin{sloppypar}
Human immunodeficiency virus infection/acquired immunodeficiency syndrome (HIV/AIDS) is one of the greatest public health concerns of the last decades worldwide. UNAIDS, the Joint United Nations Programme on HIV/AIDS reports an estimated 35.3 (32.2--38.8) million people living with HIV in 2012 \cite{unaids}. Though the data of 2.3 (1.9–-2.7) million infections acquired in 2012 show a decline in the number of new cases compared to 2001, enormous effort is devoted to halt and begin to reverse the epidemic. Developing vaccine which provides partial or complete protection against HIV infection remains a striking challenge of modern times. IAVI --- The International AIDS Vaccine Initiative \cite{iavi} believes that the earlier results on combining the two major approaches of stimulating antibody production and HIV infection clearance in the human body provides grounds for optimism and confidence in designing HIV vaccines.\\

There are several compartmental models (see, for instance, \cite{blower95, elbasha, massad, mclean93}) which deal with the mathematical modeling of HIV infection dynamics. The following model for the transmission of HIV with differential infectivity was given by Sharomi{\it~et~al.}~\cite{sharomi}\end{sloppypar}
\begin{equation}\label{HIVmod}\tag{$\textit{H}$}\aligned
\frac{\ud}{\ud t}S&=(1-p)\Lambda-\mu S-\lambda S+\gamma S_V,\\
\frac{\ud}{\ud t}S_V&=p\Lambda-\mu S_V-q\lambda S_V-\gamma S_V,\\
\frac{\ud}{\ud t}Y_1&=\rho_1\lambda S-(\mu+\sigma_1)Y_1,\\
\frac{\ud}{\ud t}Y_2&=\rho_2\lambda S-(\mu+\sigma_2)Y_2,\\
\frac{\ud}{\ud t}W_1&=\pi_1q\lambda S_V-(\mu+\theta_1\sigma_1)W_1,\\
\frac{\ud}{\ud t}W_2&=\pi_2q\lambda S_V-(\mu+\theta_2\sigma_2)W_2,\\
\frac{\ud}{\ud t}A&=\sigma_1 Y_1+\sigma_2 Y_2+\theta_1\sigma_1 W_1+\theta_2\sigma_2 W_2-(\delta+\mu)A,
\nonumber\endaligned\end{equation}
where the population is divided into the disjoint classes of unvaccinated ($S$) and vaccinated ($S_V$) susceptibles, unvaccinated infected individuals with high ($Y_1$) and low ($Y_2$) viral load, vaccinated infected individuals with high ($W_1$) and low ($W_2$) viral load, and individuals in AIDS stage of infection ($A$). Note that instead of the notation  $X$ and $V$ of the unvaccinated and vaccinated susceptible classes applied in \cite{sharomi} we use $S$ and $S_V$ to avoid confusion with the matrix $V^{i}$ and vector $\mathcal{X}$ used in section \ref{sec:dfe}. The total population of individuals not in the AIDS stage is denoted by $N$, $N=S+S_V+Y_1+Y_2+W_1+W_2$. Disease transmission is modeled by standard incidence, with transmission coefficients $\beta_1$ and $\beta_2$ in the infected classes with high and low viral load, the force of infection $\lambda$ arises as $\lambda=\sum_{j=1}^2 \left(\beta_j \frac{Y_j}{N}+\beta_j s_j \frac{W_j}{N}\right)$. Relative infectiousness of members of the $W_1$ and $W_2$ compartments is represented by $s_1$ and $s_2$. Parameter $\Lambda$ is the constant recruitment rate into the population, while $\mu$ stands for natural mortality. Susceptible individuals are immunized by vaccination with probability $p$, and $\gamma$ is the rate of waning immunity.  In the classes of infected individuals with high and low viral load the progression of the disease is modeled by $\sigma_1$ and $\sigma_2$, modification parameters $\theta_1$ and $\theta_2$ are used to account for the reduction of the progression rates in $W_1$ and $W_2$. The  disease-induced mortality rate $\delta$ is introduced into the equation of $A$, the individuals in the AIDS stage. All model parameters are assumed positive.\\

It holds that the system (\ref{HIVmod}) has a unique disease free equilibrium $E_{df}^H=$ $(S_0,$ $(S_V)_0,$ $\lambda_0,$ $A_0)$ with $S_0=\frac{(\gamma+(1-p)\mu)\Lambda}{\mu(\mu+\gamma)}>0$, $(S_V)_0=\frac{p \Lambda}{\mu+\gamma}>0$ and $\lambda_0=0$, $A_0=0$, which is globally asymptotically stable in the disease free subspace, moreover by Lemma 3 \cite{sharomi} $E_{df}^H$ is a locally asymptotically stable (unstable) steady state of (\ref{HIVmod}) if $\rn_H<1$ ($\rn_H>1$), where the reproduction number $\rn_H$ is defined by
\begin{equation}\aligned
\rn_H&=\frac{1}{N_0} \left(\frac{B_1 X_0}{(\mu+\sigma_1)(\mu+\sigma_2)}+\frac{B_2 V_0}{(\mu+\theta_1\sigma_1)(\mu+\theta_2\sigma_2)}\right)
\nonumber\endaligned\end{equation}
with $N_0=S_0+(S_V)_0$ and $B_1=\beta_2 \rho_2 (\mu+\sigma_1)+\beta_1 \rho_1 (\mu+\sigma_2)$, $B_2=q(\pi_2 s_2\beta_2 (\mu+\theta_1\sigma_1)+\pi_1 s_1\beta_1 (\mu+\theta_2\sigma_2))$. It easily follows from the model equations that in an equilibrium an infected compartment is at a positive steady state if and only if all components of the fixed point are positive. According to Theorem 4 \cite{sharomi} system (\ref{HIVmod}) has a unique endemic equilibrium if $\rn_H>1$, nevertheless positive fixed points can exist for $\rn_H<1$ as well; under certain conditions on the parameters the model exhibits backward bifurcation at $\rn_H=1$, that is, a critical value $\rn_{c} <1$ can be defined such that there are two distinct positive equilibria for values of $\rn_H$ in $(\rn_c, 1)$ (see \cite{sharomi} for details).\\

We consider $r$ patches and investigate the dynamics of HIV infection by incorporating the possibility of traveling into model (\ref{HIVmod}). In each region the same model compartments as in the one-patch model can be defined, upper index `$i$' is used to label the classes of region $i$, $i \in \{1, \dots r\}$. In terms of our notations in system (\ref{genmodel}), $n=4$, $m=2$, $k=1$ and we let $x^i=(Y_1^i, Y_2^i, W_1^i, W_2^i)^T$, $y^i=(S^i, S_V^i)^T$, $z^i=A^i$. The equalities $D^i=-(\delta^i+\mu^i)z^i$, $Z^i=(\sigma^i_1,\sigma^i_2,\theta^i_1\sigma^i_1,\theta^i_2\sigma^i_2)$ and
\begin{equation}\aligned\label{eq:HIVvari}
g^i(x^i, y^i, z^i)&=\begin{pmatrix}
(1-p^i) \Lambda^i\\
p^i \Lambda^i
\end{pmatrix}+\begin{pmatrix}
-\mu^i & \gamma^i\\
0 &  -\gamma^i-\mu^i
\end{pmatrix} y^i,\\
V^i&=\begin{pmatrix}
\mu^i+\sigma_1^i& 0&0&0\\
0&\mu^i+\sigma_2^i&0&0\\
0&0&\mu^i+\theta_1^i \sigma_1^i&0\\
0&0&0&\mu^i+\theta_2^i \sigma_2^i
\end{pmatrix},\\
\mathcal{B}^i&=\frac{1}{N^i}\begin{pmatrix}
\beta^i_1&\beta^i_2& s_1^i \beta^i_1& s^i_2 \beta^i_2\\
q^i \beta^i_1& q^i \beta^i_2&q^i s_1^i \beta^i_1&q^i s^i_2 \beta^i_2\\
\end{pmatrix}, \\
\eta^{i}_{1,\cdot}&=(\rho^i_1,\rho^i_2,0,0)^T, \eta^{i}_{2,\cdot}=(0,0,\pi^i_1,\pi^i_2)^T
\endaligned\end{equation}
put the multiregional HIV model  $(H_1)$--$(H_r)$ into the form of system $(L_1)$--$(L_r)$, moreover  $F^i$ arises as
\begin{equation}\aligned
F^i&=\begin{pmatrix}
\frac{\beta^i_1\rho^i_1S_0^{i}}{N^i_0}& \frac{\beta^i_2\rho^i_1S_0^{i}}{N^i_0}&\frac{s^i_1\beta^i_1\rho^i_1S_0^{i}}{N^i_0}&\frac{s^i_2\beta^i_2\rho^i_1S_0^{i}}{N^i_0}\\
\frac{\beta^i_1\rho^i_2S_0^{i}}{N^i_0}& \frac{\beta^i_2\rho^i_2S_0^{i}}{N^i_0}&\frac{s^i_1\beta^i_1\rho^i_2S_0^{i}}{N^i_0}&\frac{s^i_2\beta^i_2\rho^i_2S_0^{i}}{N^i_0}\\
\frac{\beta^i_1\pi^i_1q^i(S_V^{i})_0}{N^i_0}& \frac{\beta^i_2\pi^i_1q^i(S_V^{i})_0}{N^i_0}&\frac{s^i_1\beta^i_1\pi^i_1q^i(S_V^{i})_0}{N^i_0}&\frac{s^i_2\beta^i_2\pi^i_1q^i(S_V^{i})_0}{N^i_0}\\
\frac{\beta^i_1\pi^i_2 q^i(S_V^{i})_0}{N^i_0}& \frac{\beta^i_2\pi^i_2 q^i(S_V^{i})_0}{N^i_0}&\frac{s^i_1\beta^i_1\pi^i_2 q^i(S_V^{i})_0}{N^i_0}&\frac{s^i_2\beta^i_2\pi^i_2 q^i(S_V^{i})_0}{N^i_0}
\end{pmatrix}.
\nonumber\endaligned\end{equation}
By introducing parameter $c_w^{ij}$ to represent the connectivity potential from class $w^j$ to $w^i$, $w \in \{S, S_V, Y_1, Y_2, W_1, W_2, A\}$ and $i, j \in \{1, \dots r\}$, $i \neq j$, $\alpha$ as general mobility parameter, system $(H_1)$--$(H_r)$ can be extended to $(T_1)$--$(T_r)$ in the same way as described in section \ref{sec:model} to get an epidemic model with HIV dynamics in $r$ regions connected by traveling.

\subsection{Disease free equilibrium for arbitrary volumes of travel}
We recall that system (\ref{HIVmod}) has a single disease free fixed point $(0,y_0,0)$ with $y_0=\left(S_0,(S_V)_0\right)^T$, which is locally asymptotically stable if $\rn_H<1$ and unstable if $\rn_H>1$. This also means that the system of the regions connected with traveling $(T_1)$--$(T_r)$ admits a single disease free steady state when the general mobility parameter $\alpha$ equals zero. We now show that in case of the HIV model the connected system has a disease free equilibrium for every $\alpha>0$ as well.
\begin{theorem}\label{th:dfehiv}
The connected system of $r$ regions with HIV dynamics admits a unique disease free equilibrium for any $\alpha \geq 0$. It also holds that the classes of individuals in the AIDS stage are at zero steady state.
\end{theorem}
\begin{proof}
When the infected classes are at zero steady state in the HIV model we obtain the fixed point equations 
\begin{equation}\aligned\label{eq:msvmsma}
\begin{pmatrix}
p^1 \Lambda^1\\  \vdots\\p^r \Lambda^r
\end{pmatrix}&=M_{S_V} \begin{pmatrix}
\hat{S_V}^1(\alpha)\\
\vdots \\\hat{S_V}^r(\alpha)
\end{pmatrix},\\
\text{diag}(\gamma^i)\cdot\begin{pmatrix}
\hat{S_V}^1(\alpha)\\
\vdots \\\hat{S_V}^r(\alpha)
\end{pmatrix}+\begin{pmatrix}
(1-p^1) \Lambda^1\\  \vdots\\(1-p^r) \Lambda^r
\end{pmatrix}&=M_{S}\begin{pmatrix}
\hat{S}^1(\alpha)\\
\vdots \\\hat{S}^r(\alpha)
\end{pmatrix},\\
0&=M_{A}\begin{pmatrix}
\hat{A}^1(\alpha)\\
\vdots \\\hat{A}^r(\alpha)
\end{pmatrix}\\
\endaligned\end{equation}
with
\begin{equation}\aligned
M_{S_V}&=\begin{pmatrix}
\sum_{\substack{j=1 \\j \neq 1}}^r \alpha c_{S_V}^{j1}+\mu^1+\gamma^1 & \dots & -\alpha c_{S_V}^{1r}\\
\vdots  & \ddots  &\vdots\\
-\alpha c_{S_V}^{r1}& \dots &  \sum_{\substack{j=1 \\j \neq r}}^r \alpha c_{S_V}^{jr} +\mu^r+\gamma^r
\end{pmatrix},\\
M_{S}&=\begin{pmatrix}
\sum_{\substack{j=1 \\j \neq 1}}^r \alpha c_{S}^{j1} +\mu^1& \dots & -\alpha c_{S}^{1r}\\
\vdots & \ddots  &\vdots\\
-\alpha c_{S}^{r1}& \dots &  \sum_{\substack{j=1 \\j \neq r}}^r \alpha c_{S}^{jr} +\mu^r
\end{pmatrix},\\
M_{A}&=\begin{pmatrix}
\sum_{\substack{j=1 \\j \neq 1}}^r \alpha c_{A}^{j1}+\mu^1+\delta^1 &  \dots & -\alpha c_{A}^{1r}\\
\vdots &  \ddots  &\vdots\\
-\alpha c_{A}^{r1}& \dots &  \sum_{\substack{j=1 \\j \neq r}}^r \alpha c_{A}^{jr} +\mu^r+\delta^r
\end{pmatrix}.
\nonumber\endaligned\end{equation}
Similarly as discussed in section \ref{sec:ee} for the matrix $M_z$, Theorem 5.1 \cite{friedler} implies that the inverses of $M_{S_V}$, $M_{S}$ and $M_{A}$ exist and are nonnegative. It immediately follows that $\hat{A}^i(\alpha)=0$, $\hat{S_V}^i(\alpha)>0$ and $\hat{S}^i(\alpha)>0$, $i \in \{1, \dots r\}$, hold for the unique solution of (\ref{eq:msvmsma}).
\end{proof}

\subsection{Endemic equilibria}
In section \ref{sec:ee} we required $\hat{x}^i >0$ (that is,  $\hat{x}^i \geq 0$ with both zero and positive components not possible) for endemic steady states, we recall that this is fulfilled in the HIV model since the model parameters are assumed positive. At positive fixed points $g^{i}$ and $\mathcal{B}^{i}$ defined in (\ref{eq:HIVvari}) are infinitely many times continuously differentiable, hence it is possible to derive equations (\ref{eq:derx}) and (\ref{eq:derxN}). The analysis in section \ref{sec:nodirect} has been carried out with the extra condition that the matrix $V^i-F^i$ is irreducible, which is indeed the case by the HIV model.\\

Theorem \ref{th:iftendemic} contains condition on the non-singularity of the Jacobian of the system evaluated at an endemic fixed point and $\alpha=0$. The matrix  $\left(\frac{\partial \mathcal{T}}{\partial \mathcal{X}}\right) (0,\cdot)$ has block diagonal form with the block $\left(\frac{\partial \mathcal{T}^i}{\partial \mathcal{X}^i}\right) (0,\cdot)$ corresponding to region $i$, where we denote $\mathcal{X}^i=(x^i,y^i,z^i)^T$ and  $\mathcal{T}^i=(\mathcal{T}^{i,x},\mathcal{T}^{i,y},\mathcal{T}^{i,z})^T$. This gives $\det \left(\frac{\partial \mathcal{T}}{\partial \mathcal{X}}\right) (0,\cdot)=$ $\prod_{i=1}^{r} \det \left(\frac{\partial \mathcal{T}^i}{\partial \mathcal{X}^i}\right) (0,\cdot) $, so we conclude that the Jacobian of the system of $r$ regions is non-singular at a fixed point if and only if $\det \left(\frac{\partial \mathcal{T}^i}{\partial \mathcal{X}^i}\right) (0,\cdot)\neq 0$ holds in each region $i$. It is not hard to see that the matrix $\left(\frac{\partial \mathcal{T}^i}{\partial \mathcal{X}^i}\right) (0,\cdot)$ gives the Jacobian of $(H_i)$ without traveling, that is, it suffices to consider the steady state--components in each region separately. The Jacobian evaluated at a stable equilibrium has only eigenvalues with negative real part, which guarantees the non-singularity of the matrix; although in the case when the fixed point is unstable we only know that the determinant has an eigenvalue with positive real part, which doesn't exclude the existence of an eigenvalue on the imaginary axis. \\

It is conjectured  from an example of  \cite{sharomi} that if $\rn_H>1$ in the one-patch HIV model then the positive fixed point is locally asymptotically stable and the disease free equilibrium is unstable, furthermore in case the model exhibits backward bifurcation, one of the endemic steady states is locally asymptotically stable whilst the other one is unstable for $\rn_c^i<\rn_H^i<1$. As noted above, the matrix $\left(\frac{\partial \mathcal{T}^i}{\partial \mathcal{X}^i}\right) (0,\cdot)$ is always invertible at stable equilibria, and we use the same set of parameter values as the example in \cite{sharomi} to illustrate a case when the determinant of $\left(\frac{\partial \mathcal{T}^i}{\partial \mathcal{X}^i}\right) (0,\cdot)$ is non-zero at unstable fixed points. The continuous dependence of the determinant on parameters implies that the situation when the Jacobian is singular is realized only in isolated points of the parameter space. In fact, for $\rho^i_1=0.3$, $\rho^i_2=0.7$, $\sigma^i_1=0.45$, $\sigma^i_2=17$, $\beta^i_1=0.85$, $\beta^i_2=0.1$, $s^i_1=1$, $s^i_2=1$, $\pi^i_1=0.9$, $\pi^i_2=0.1$, $\theta^i_1=0.5$, $\theta^i_1=0.5$, $q^i=0.5$, $\mu^i=0.05$, $\gamma^i=0.05$, $\Lambda^i=1$ and $p^i=0.999$, the condition for backward bifurcation holds and $\rn_c^i<\rn_H^i<1$ \cite{sharomi}, moreover the positive equilibria $(\hat{X}^i, \hat{V}^i, \hat{\lambda}^i,\hat{A}^i)_{1,2}$ with $(\hat{\lambda}^i)_1= 0.0195$ and $(\hat{\lambda}^i)_2= 0.1492$ are unstable and stable, respectively, with the Jacobian evaluated at $(\hat{S}^i, \hat{S_V}^i, \hat{\lambda}^i,\hat{A}^i)_{1}$ non-singular. Letting  $\beta^i_1=1$ makes $\rn_H^i=1.12 >1$ and the disease free steady state $E_{df}^{H,i}=(S^i_0,(S_V^i)_0,0,0)$ is unstable with no eigenvalues of the Jacobian having zero real part.\\

In the sequel we assume that the model parameters are set such that $\left(\frac{\partial \mathcal{T}}{\partial \mathcal{X}}\right) (0,\cdot)\neq 0$ at the fixed points and thus the conditions of Theorem \ref{th:iftendemic} hold. Then, as discussed above, all the assumptions made throughout sections \ref{sec:model}, \ref{sec:dfe}, \ref{sec:ee}, \ref{sec:irred} and \ref{sec:nodirect} are satisfied and we conclude that the results obtained in these sections for the general model are applicable for the multiregional HIV model with traveling. We use this model to demonstrate our findings in the case when $r=3$. Let us assume that the necessary conditions for backward bifurcation are satisfied in all three regions. Then each region $i$ can have one (the case when $\rn_H^i<\rn_c^i$), three (the case when $\rn_c^i<\rn_H^i<1$) or two (the case when $\rn_H^i>1$) equilibria, including the disease free steady state. Thus, without traveling the united system of three regions with HIV dynamics has a disease free equilibrium, and $1^\phi \cdot 3^\psi \cdot 2^\omega-1$ endemic steady states where for the integers $\phi, \psi$ and $\omega$ it holds that $0 \leq \phi, \psi, \omega \leq 3$ and $\phi+\psi+\omega=3$; it is easy to check that the possibilities for the number of equilibria are 1, 2, 3, 4, 6, 8, 9, 12, 18 and 27.\\

Theorem \ref{th:dfehiv} guarantees the existence and uniqueness of the disease free fixed point when traveling is incorporated into the system.  Theorem \ref{th:posee} and Corollary \ref{cor:summarypluszero} give a full picture about the (non-)persistence of endemic steady states: a boundary endemic equilibrium of the disconnected system, where there is a DFAT region $i$ with $\rn_H^i>1$ which is reachable from an EAT region, will not be preserved in the connected system for any small volumes of traveling, however all other endemic fixed points of the disconnected system will exist if the mobility parameter $\alpha$ is small enough. It is obvious that the movement network connecting the regions plays an important role in deriving the exact number of steady states of the system with traveling; in what follows we give a complete description of the possible cases.

\subsection{Irreducible connection network}
\begin{figure}[t]
\centering
\subfigure[Reducible network.] { \includegraphics[width=3cm]{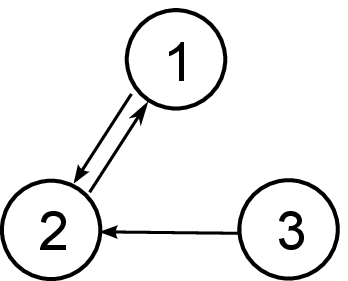}}\hspace{1cm}
\subfigure[Irreducible network.]{\includegraphics[width=3cm]{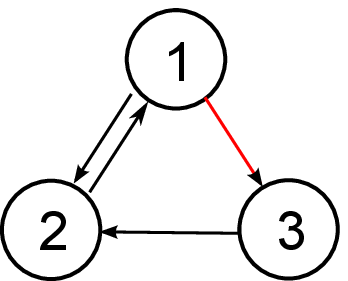}}\hspace{1cm}
\subfigure[Complete network.]{\includegraphics[width=3cm]{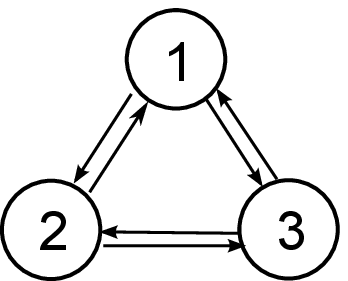}}
\caption{Example of reducible, irreducible and complete travel networks for three regions. Though both 1 and 2 are reachable from region 3, by (a) there is no connection to 3 from the other two regions. Adding a link from 1 to 3  on (b) makes the network irreducible, though not complete. In the example depicted in (c) the regions are directly connected to each other, which means that the network is complete and also clearly irreducible.\label{fig:irredrednetwork}}
\end{figure}

First we consider the case when each region is reachable from any other region, that is, the graph consisting of nodes as regions and directed edges as direct connections from (the infected classes of) one region to (the infected classes of) another region, is irreducible. Such network is realized if we think of the nodes as distant territories and the edges as one-way air travel routes. Note that the irreducibility of the network doesn't mean that each region is directly accessible from any other one; as experienced by the global airline network of the world, some territories are linked to each other via the correspondence in a third region. Figure \ref{fig:irredrednetwork} is presented to give examples of irreducible an reducible connection networks.

\begin{theorem}
If the network connecting three regions with HIV dynamics is irreducible then the number of fixed points of the disconnected system, which persists in $(T_1)$--$(T_3)$ for small volumes of traveling, can be 1, 2, 3, 4, 9, 10 or 27, depending on the local reproduction numbers in the regions. As pointed out in Theorem \ref{th:dfehiv} the unique disease free equilibrium always exists in $(T_1)$--$(T_3)$.
\end{theorem}

\begin{proof}
We distinguish four cases on the number of regions with local reproduction number greater than one.\\

Case 1: No regions with $\rn_H^i >1$.\\
This case is easy to treat: if in all three regions it holds that the local reproduction number is less than one, then Theorem \ref{th:pluszeroro<1} implies that all fixed points of the disconnected system of three regions are preserved for some small positive $\alpha$-s. If $\rn_H^i <1$ for $i=1,2,3$, the system $(L_1)$--$(L_3)$ may have 1 (if $\rn_H^i <\rn_c^i$ for $i=1,2,3$), 3 (if $\rn^{i_1} <\rn_c^{i_1}$, $\rn^{i_2} <\rn_c^{i_2}$ and $\rn_c^{i_3} <\rn^{i_3}<1$, $\{i_1,i_2,i_3\}=\{1,2,3\}$) , 9 (if $\rn^{i_1} <\rn_c^{i_1}$ and $\rn_c^{i_2} <\rn^{i_2}<1$, $\rn_c^{i_3} <\rn^{i_3}<1$, $\{i_1,i_2,i_3\}=\{1,2,3\}$) or 27 (if $\rn_c^{i} <\rn^{i}<1$ for $i=1,2,3$) equilibria.\\

Case 2: Exactly one region with $\rn_H^i >1$.\\
Let this region be labeled by $i_1$, system $(H_{i_1})$ has a disease free and a positive fixed point. By Theorem \ref{th:pluszeroro>1} and the assumption that $i_1$ is reachable from both other regions, we get that no endemic equilibrium of  $(L_1)$--$(L_3)$, where $i_1$ is DFAT, persists with traveling. It follows that besides the disease free equilibrium (when none of the regions is endemic), only fixed points with $\hat{x}^{i_1}=f_{\hat{x}^{i_1}}(0)>0$ will exist for small volumes of traveling, which makes the total number of equilibria 2 (1 disease free + 1 endemic if $\rn^{i_2} <\rn_c^{i_2}$, $\rn^{i_3} <\rn_c^{i_3}$), 4 (1 disease free + 3 endemic if either $\rn^{i_2} <\rn_c^{i_2}$ and $\rn_c^{i_3}<\rn^{i_3}<1$, or $\rn^{i_3} <\rn_c^{i_3}$ and $\rn_c^{i_2}<\rn^{i_2}<1$) or 10 (1 disease free + 9 endemic if $\rn_c^{i_2}<\rn^{i_2}<1$, $\rn_c^{i_3}<\rn^{i_3}<1$). \\

Case 3: Exactly two regions with $\rn_H^i >1$.\\
We let the reader convince him- or herself that if $\rn^{i_1} >1$ and $\rn^{i_2} >1$ ($i_1, i_2 \in \{1,2,3\}$) hold then a total number of 2 or 4 fixed points of the disconnected regions may persist in system $(T_1)$--$(T_3)$ for small $\alpha$-s. The proof can be led in a similar way as by Case 2, considering the two possibilities $\rn^{i_3} <\rn_c^{i_3}$ and $\rn_c^{i_3}<\rn^{i_3}<1$ for the local reproduction number of the third region. One again gets that the equilibrium where all the regions are disease free will exists, moreover it is worth recalling that no region with $\rn_H^i>1$ can be DFAT while another region is EAT. \\

Case 4: All three regions with $\rn_H^i >1$.\\
We apply Theorems \ref{th:pluszeroro>1} and \ref{th:pluszeroro<1} to get that if any of the regions is DFAT then so should be the other two for an equilibrium to persist $(T_1)$--$(T_3)$ and $\alpha$ small. This implies that only 2 fixed points of $(H_1)$--$(H_3)$, the disease free and the endemic with all three regions at positive steady state, will be preserved once traveling is incorporated.

\end{proof}

To summarize our findings, we note that the introduction of traveling via an irreducible network into $(H_1)$--$(H_3)$ never gives rise to situations when precisely 6, 8, 12 and 18 fixed points of the disconnected system continues to exist with traveling. Nevertheless evidence has been showed that new dynamical behavior (namely, the case when 10 equilibria coexist) can occur when connecting the regions by means of small volume--traveling. We conjecture that lifting the irreducibility restriction on the network results in even more new scenarios. This is proved in the next subsection.

\subsection{General connection network}
It is clear that, with the help of Theorems \ref{th:pluszeroro>1} and \ref{th:pluszeroro<1}, the number of fixed points in the disconnected system which persist with traveling can be easily determined for any given (not necessarily irreducible) connecting network.  The next theorem discusses all the possibilities on the number of equilibria. Examples are also provided to illustrate the cases.

\begin{theorem}
Depending on the local reproduction numbers and the connections between the regions, the system of three regions for HIV dynamics with traveling $(T_1)$--$(T_3)$ preserves  1--7, 9, 10, 12, 18 or 27 fixed points of the disconnected system for small volumes of traveling. As pointed out in Theorem \ref{th:dfehiv} the unique disease free equilibrium always exists in $(T_1)$--$(T_3)$.
\end{theorem}

\begin{proof}
The existence of the unique disease free steady state is guaranteed by Theorem \ref{th:dfehiv}. The proof will be done in the following steps:
\begin{enumerate}
\item[Step 1] We show that there is no travel network which results in the persistence of 13-17 or 19-26 equilibria.
\item[Step 2] We prove that the system of three regions with traveling cannot have 8 or 11 fixed points.
\item[Step 3] We demonstrate through examples that all other numbers of equilibria up to 27 can be realized.
\end{enumerate}

Step 1: \\
We note that if either $\rn_H^i<\rn_c^i$ holds in any of the regions, or there are two or more regions where $\rn_H^i>1$, then the number of fixed points doesn't exceed 12. Thus, to have at least 13 equilibria there must be two regions $i_1, i_2$ with $\rn_c^{i_1}<\rn_H^{i_1}<1$ and $\rn_c^{i_2}<\rn_H^{i_2}<1$. If the third region also has three fixed points, that is,  $\rn_c^{i_3}<\rn_H^{i_3}<1$ then there is no region with local reproduction number greater than one, and thus Theorem \ref{th:pluszeroro<1} yields the existence of 27 steady states. Otherwise $\rn_H^{i_3}$ is greater than one and region $i_3$ has two equilibria, one disease free and one endemic. In this case by Theorem \ref{th:iftendemic} there are 9 fixed points where $\hat{x}_{i_3}>0$, all of which preserved for small volumes of traveling. The possible number of equilibria with $\hat{x}_{i_3}=0$, which exist with traveling, are one (if $i_3$ is reachable from both regions), 3 (if $i_3$ is reachable from only one of them) and 9 (if $i_3$ is unreachable). We conclude that there are only two values greater than 12 for the possible number of fixed points in the travel system, which are 18 and 27. \\

Step 2: \\
We distinguish 5 cases to consider:
\begin{enumerate}[(i)]
\item $\rn_c^{i}<\rn_H^{i}<1$ for $i=1,2,3$;
\item $\rn_c^{i_1}<\rn_H^{i_1}<1$, $\rn_c^{i_2}<\rn_H^{i_2}<1$, $\rn_H^{i_3}>1$, $\{i_1, i_2, i_3\}=\{1,2,3\}$; 
\item $\rn_c^{i_1}<\rn_H^{i_1}<1$, $\rn_H^{i_2}>1$, $\rn_H^{i_3}>1$, $\{i_1, i_2, i_3\}=\{1,2,3\}$; 
\item $\rn_H^{i}>1$ for $i=1,2,3$;
\item there is an $i_1$ such that $\rn_H^{i_1}<\rn_c^{i_1}$, $i_1 \in \{1,2,3\}$.
\end{enumerate}
In case (i) each region has three equilibria, hence the connected system obtains 27 fixed points for small $\alpha$-s. We have seen in Step 1 that there are 10, 12 or 18 equilibria in a network with the regions such that case (ii) holds. \\
Let us assume that case (iii) is realized, and henceforth the system has maximum 12 fixed points. If neither $i_2$ nor $i_3$ is reachable from $i_1$ then the persistence of an equilibrium for small $\alpha$-s is independent of the steady state--value $\hat{x}_{i_1}$, thus the number of possible fixed points is a multiple of three, which doesn't hold for any of 8 and 11. On the other hand if there is a connection from $i_1$ to any of $i_2$ and $i_3$ then some equilibria may vanish once traveling is incorporated. More precisely, let $i_2$ be reachable from $i_1$. By Theorem  \ref{th:pluszeroro>1}, steady states where region $i_2$ is DFAT and $i_1$ is EAT don't exist in the connected system, which means that the connection from $i_1$ to $i_2$ destroys $2 \cdot 1 \cdot 2=4$ fixed points out of the maximum 12 (note that in region $i_1$ there are two positive equilibria, and the steady state--value of $i_3$ doesn't change the non-persistence of fixed points of the type $\hat{x}_{i_2}=0$, $\hat{x}_{i_1}>0$). This immediately makes 11 equilibria impossible. By means of the above arguments, either $\hat{x}_{i_1}=\hat{x}_{i_2}=0$ or  $\hat{x}_{i_2}>0$ must be satisfied for each fixed point which persists, and their number can be maximum 8. In particular the equilibria, $E_1$ where $\hat{x}_{i_1}=\hat{x}_{i_2}=0$, $\hat{x}_{i_3}>0$ and $E_2$ where $\hat{x}_{i_2}>0$, $\hat{x}_{i_1}>0$, $\hat{x}_{i_3}=0$, are such fixed points. $E_1$ persists with traveling only if the network is such that $i_2$ is unreachable from $i_3$, so in this case there must be a path from $i_1$ to $i_3$ due to the connectedness of the network (recall that we assumed that $i_2$ is reachable from $i_1$, so any link from $i_3$ would make $i_2$ reachable from $i_3$). However this structure makes the persistence of $E_2$ for small $\alpha$-s impossible, and we get that there cannot be 8 steady states in the case when $\rn_H^{i_2}, \rn_H^{i_3}>1$ and $\rn_c^{i_1}<\rn_H^{i_1}<1$.\\ 
The maximum number of equilibria by cases (iv) and (v) are 8 and 9, respectively, which observation finishes the investigation of the persistence of precisely 11 steady states in the system with traveling. By case (iv) some of the 8 fixed points obtained in the disconnected system clearly won't persist in the connected system --- if, for instance, there is a link from $i_1$ to $i_2$ then the equilibrium where $\hat{x}_{i_1}>0$ and $\hat{x}_{i_2}=0$ won't be preserved for positive $\alpha$-s. If case (v) is realized and there is a region with local reproduction number greater than one then the system cannot have more than 6 steady states. Otherwise $\rn_H^{i}<1$ holds for all $i\in \{1,2,3\}$ in case (v) and Theorem \ref{th:pluszeroro<1} yields that all fixed points of the disconnected system continues to exist once traveling is incorporated. It is not hard to check that the number of equilibria is never 8.\\

\begin{figure}[t]
\centering
\subfigure[4 equilibria]{\includegraphics[width=3cm]{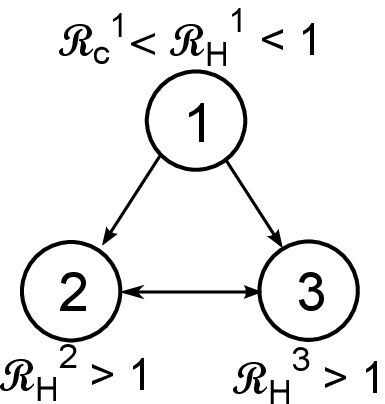}}\hspace{0.4cm}
\subfigure[5 equilibria]{\includegraphics[width=3cm]{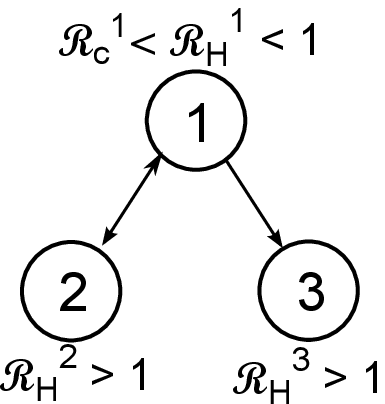}}\hspace{0.4cm}
\subfigure[6 equilibria]{\includegraphics[width=3cm]{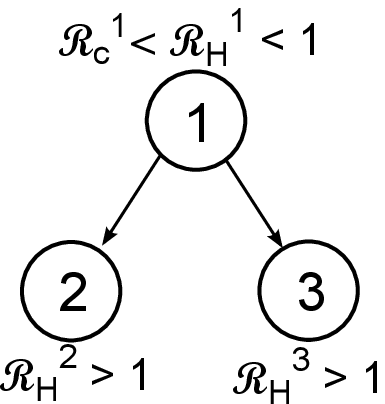}}\hspace{0.4cm}
\subfigure[7 equilibria]{\includegraphics[width=3cm]{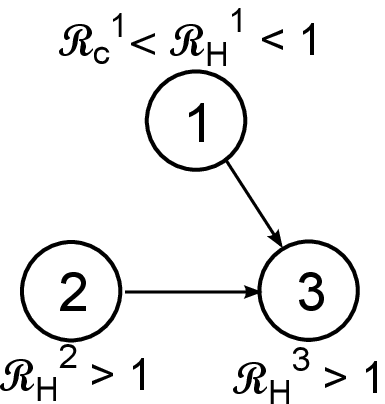}}
\caption{Examples of the travel network for three regions with $\rn_c^{1}<\rn_H^{1}<1$ and $\rn_H^{2}, \rn_H^{3}>1$. \label{fig:4567}}
\end{figure}
Step 3:\\
Any network where $\rn_H^{i}<\rn_c^{i}$ is satisfied for all $i$ exhibits only one, the disease free equilibrium. It is straightforward to see that the complete network of three regions has 2 fixed points when $\rn_H^{i}>1$ for $i \in \{1,2,3\}$, and if there is one, two or three region(s) where $\rn_c^{i}<\rn_H^{i}<1$ whilst $\rn_H^{j}<\rn_c^{j}$ holds in the remaining region(s) then, independently of the connections, the connected system preserves 3, 9 or 27 equilibria, respectively, of the disconnected system from small volume--traveling. \\
Any network where $\rn_c^{i_1}<\rn_H^{i_1}<1$, $\rn_c^{i_2}<\rn_H^{i_2}<1$, $\rn_H^{i_3}>1$ and $i_3$ is reachable from both other regions works as a suitable example for the case of 10 fixed points, since this way the disease free equilibrium coexists with 9 steady states where $\hat{x}_{i_3}>0$. A way to obtain 12 and 18  fixed points has been described in Step 1, and we use the case when $\rn_c^{1}<\rn_H^{1}<1$ and $\rn_H^{2}, \rn_H^{3}>1$ to construct examples for 4, 5, 6 and 7 steady states. Figure \ref{fig:4567} depicts one possibility for the network of each case, though it is clear that there might be several ways to get the same number of equilibria. \\
If both regions 2 and 3 are reachable from 1, then fixed points where $\hat{x}_{i_1}>0$ are preserved with traveling only if $\hat{x}_{i_2}>0$ and $\hat{x}_{i_3}>0$ also hold. On the top of these 2 positive equilibria, there surely exists the disease free steady state plus 1, 2 or 3 non-zero fixed point(s) with $\hat{x}_{i_1}=0$, depending on whether region 2 is reachable from 3 and vice versa, as illustrated in Figure \ref{fig:4567} (a), (b) and (c). If region 3 is reachable from both regions 1 and 2 then $\hat{x}_{i_3}=0$ is only possible in the disease free equilibrium; although all 6 fixed points where region 3 is at the endemic steady state persist for small volumes of traveling if there is no connection from 1 to 2 (Figure \ref{fig:4567} (d) shows such a situation).
\end{proof}

The dynamics of the HIV model in connected regions is worth investigating in more depth, although this is beyond the scope of this study. However the numerical simulations presented in the next section reveal some interesting behavior of the model. 

\section{Rich dynamical behavior}\label{sec:richdyn}

This section is devoted to illustrate the rich dynamical behavior in the model. The epidemiological consequence of the existence of multiple positive equilibria in one-patch models is that the epidemic can have various outcomes, because solutions with different initial values might converge to different steady states. Stable fixed points are of particular interest as they usually attract solutions starting in the neighborhood of other (unstable) steady states. For instance, in case of backward bifurcation the presence of a stable positive equilibrium for $\rn<1$ makes it possible that the disease sustains itself even if the number of secondary cases generated by a single infected individual in a fully susceptible population is less than one. However, considering multiple patches with connections from one to another deeply influences local disease dynamics, since the travel of infected agents induces outbreaks in disease free regions. The inflow of infected individuals might change the limiting behavior when pushing a certain solution into the attracting region of a different steady state, and it also may modify the value of fixed points.\\

Henceforth, knowing the stability of equilibria in the connected system of regions is of key importance. For small volumes of traveling not only the number of fixed points  but also their stability can be determined: whenever a steady state of the disconnected regions continues to exists in the system with traveling by means of the implicit function theorem, its stability is not changed on a small interval of the mobility parameter $\alpha$. This means that equilibria of $(T_1)$--$(T_r)$, which have all $r$ components stable in the disconnected system, are stable; although every steady state which contains an unstable fixed point as a component is unstable when $\alpha=0$ and thus, also for small positive $\alpha$. In this paper the conditions for the persistence of steady states with the introduction of small--volume traveling has been described:  by a continuous function of $\alpha$, all fixed points of $(L_1)$--$(L_r)$ will exist in the connected system but those for which there is a DFAT region with local reproduction number greater than one, and to which the connecting network establishes a connection from an EAT territory. However, infection-free steady states are typically unstable for $\rn>1$, thus the above argument yields that incorporating traveling with low volumes preserves all stable fixed points of the disconnected system, since the equilibria which disappear when $\alpha$ exceeds zero are unstable. \\
The dependence of the dynamics on movement is illustrated for the HIV model. To focus our attention to how $\alpha$ influences the fixed points, their stability and the long time behavior of solutions, we let all model parameters but the local reproduction numbers in the three regions to be equal. In the figures which we present in the Appendix, the evolution of four solutions with different initial conditions were investigated as $\alpha$ increases from zero through small volumes to larger values.\\

If all three regions exhibit backward bifurcation, and the local reproduction numbers are set such that besides the disease free fixed point, there are two positive equilibria $(\hat{X}^i, \hat{V}^i, \hat{\lambda}^i,\hat{A}^i)_{1,2}$, $\hat{\lambda}^1 < \hat{\lambda}^2$, then, as described in section \ref{sec:HIV}, 27 steady states exist for small $\alpha$-s. Assuming that the conjectures of section \ref{sec:HIV} about the stability of the disease free equilibrium and the steady state with $\hat{\lambda}_2$, and the instability of the positive fixed point with $\hat{\lambda}_1$ hold in each region, we get that system $(T_1)$--$(T_r)$ with HIV dynamics exhibits 8 stable and 19 unstable steady states on an interval for $\alpha$ to the right of zero. This is confirmed by Figures \ref{fig:HIVrl1irred} and \ref{fig:HIVrl1red}, where two cases of irreducible and reducible travel networks were considered (see the Appendix for more detailed description of the networks), and $\rn_H^{1}=\rn_H^{2}=\rn_H^{3}$ holds. Introducing low volume traveling (e.g., letting $\alpha=10^{-5}$ in our examples) effects neither the stability of steady states nor the limiting behavior of solutions, however the difference in the type of the connecting network manifests for larger movement rates, as the conditions for disease eradication clearly change along with the equilibrium values (see Figures \ref{fig:HIVrl1irred} and \ref{fig:HIVrl1red} (c) and (d) where $\alpha$ were chosen as $10^{-3}$ and $10^{-1}$, respectively). \\

When there are regions with local reproduction numbers larger than one in the network, certain fixed points of the disconnected system disappear with the introduction of traveling; this phenomenon is reasonably expected to have an impact on the final outcome of the epidemic. For all three networks used for the simulations in Figures \ref{fig:HIVrg1nofb}, \ref{fig:HIVrg1feedback} and \ref{fig:HIVrg1full}, presented in the Appendix, the number of infected individuals takes off in regions with $\rn^{i}_H>1$ for small $\alpha$, regardless of the initial conditions (see figures (b) where $\alpha=10^{-5}$ and, in particular, the cases when $\lambda_2(0)=\lambda_3(0)=0$). The results for larger travel volumes (in the simulations the two settings of $\alpha=10^{-3}$ and $10^{-1}$ were considered) further support the conjecture that solutions converge to positive steady states in regions with reproduction number greater than one. However, the case when regions 2 and 3, $\rn^{2}_H, \rn^{3}_H>1$, are not reachable from each other and a single endemic equilibrium seems to attract all solutions (illustrated in Figure \ref{fig:HIVrg1nofb}) is in contrary to the situation experienced in Figure \ref{fig:HIVrg1feedback}, since we see that establishing a path from region 2 to 3 via region 1 results in the emergence of another positive (possibly locally stable) steady state in region 3. Nevertheless, comparing Figures \ref{fig:HIVrg1nofb} and \ref{fig:HIVrg1feedback} with reducible networks to Figure \ref{fig:HIVrg1full}, where a complete connecting network was considered, highlights the role of the irreducibility of the network on the dynamics. Whereas in case of reducible networks, the final epidemic outcome in a region with $\rn_H$ strongly depends on initial conditions and connections to other regions, making each region reachable from another sustains the epidemic in region 1 (where $\rn^{1}_{c}<\rn_{H}^1<1$) by giving rise to a single positive steady state of the system. This has an implication on the long term behavior of solutions in regions with $\rn_H>1$ as well, since direct connections seem to stabilize only one endemic equilibrium in region 2 and region 3, and exclude the existence of other steady states.   \\

In summary, the theoretical analysis performed throughout the paper is in accordance with the numerical simulations for small values of the general mobility parameter $\alpha$, but more importantly, it provides full information about the fixed points of $(L_1)$--$(L_r)$: it determines their (non-)persistence, along with their stability, in the system with traveling incorporated. On the other hand, little is known about the solutions of the model when the travel volume is larger, as the structure of the connecting network and initial values deeply influence the dynamics.

\section{Conclusion}

In this paper a general class of differential epidemic models with multiple susceptible, infected and  removed compartments was considered. We provided examples of multigroup, multistrain and stage progression models to illustrate the broad range of applicability of our framework to describe the spread of infectious diseases in a population of individuals. The model setup allows us to investigate disease dynamics models with multiple endemic steady states. Such models have been considered in various works in the literature including studies which deal with the phenomenon of backward bifurcation. We extended our framework to an arbitrary number of regions and incorporated the possibility of mobility of individuals (e.g., traveling) between the regions into the model. Motivated by well known multiregional models, where the exact number of steady states have not been explored, our aim in this work was to reveal the implication of mobility between the regions on the structure of equilibria in the system.  \\

We introduced a parameter $\alpha$ to express the general mobility intensity, whilst differences in the connectedness of the regions were modeled by constants, each describing the relative connectivity of one territory to another. Considering the model equations of the connected system as a function of the model variables and $\alpha$, the implicit function theorem enabled us to represent steady states as continuous functions of the mobility parameter. We showed that the unique disease free equilibrium of the disconnected system along with all componentwise positive fixed points continues to exist in the system with traveling for small $\alpha$, with their stability unchanged. On the other hand, boundary equilibria of the system with no traveling (this is, steady states with some regions without infection and others endemic for $\alpha=0$)  may disappear when $\alpha$ becomes positive, as they might move out of the nonnegative cone along the continuous function established by means of the implicit function theorem, and thus, become no longer biologically meaningful.\\

Throughout the analysis performed in the paper we gave necessary and sufficient condition for the persistence of such equilibria in the system with traveling for various types of the connecting network. It turned out that the local reproduction numbers and the structure of the graph describing connections between the infected compartments of the regions play an important role. If each infected compartment is connected to every other infected class of the same type of other regions, implying that the connecting network includes every possible link, then a boundary equilibrium of the disconnected system won't persist with traveling if and only if there is a component of the fixed point corresponding to a disease free region with local reproduction number greater than one. Assuming an extra condition on the infected subsystem in each region we showed that the same statement holds in the case when the connection network of infected classes is not complete but is still irreducible, meaning that each region is reachable from any other one via a series of links between any of the infected classes -- see Figure \ref{fig:directcon} which illustrates such a situation. The result also extends to the most general case of arbitrary connection network of the infected classes: it was proved that steady states of the disconnected system which have a disease free region with $\rn>1$  disappear from the system if the possibility of mobility establishes a connection to this region (maybe via several other regions) from a territory where the disease is endemic; nevertheless all other equilibria of the system without traveling continue to exist for small values of the mobility parameter $\alpha$. The epidemiological implication of this behavior is that, even for small volumes of traveling, all regions with local reproduction number greater than one will be invaded by the disease unless they are unreachable from endemic territories. Direct or indirect connections from regions with positive disease state make the inflow of infecteds possible and then the imported cases spread the disease in the originally disease free region due to $\rn>1$. \\

In the most common situation of forward transcritical bifurcation of the disease free equilibrium at $\rn=1$, when the disease cannot be sustained for values of $\rn$ less than one, our results yield that only connections from regions with $\rn>1$ have impact on the equilibria of the disconnected system. If a region with local reproduction number greater than one is susceptible in the absence of traveling then isolating it from endemic territories keeps the region free of infection, so a successful intervention strategy can be to deny all connections from regions with $\rn>1$. However, the dynamics becomes more complicated when small--volume traveling is incorporated into a system of multiple regions with some exhibiting the phenomenon of backward bifurcation: in a case when endemic equilibria exist for $\rn<1$ as well, protecting a region with $\rn>1$ from the disease by denying the entrance of individuals from areas where the reproduction number is greater than one is no longer sufficient (though still necessary) to prevent the outbreak. Such a situation was illustrated by an HIV transmission model for three regions where, under certain conditions, the dynamics undergoes backward bifurcation in each region. We calculated the possible number of steady states of the disconnected system which persist with the introduction of traveling with small volumes into the system, and illustrated by several examples on the network structure and model parameter setting that mobility of individuals between the regions gives rise to various scenarios for the limiting behavior of solutions, and thus makes the outcome of the epidemic difficult to predict.

\section*{Acknowledgment}

DHK acknowledges support by the European Union and the State of Hungary, co-financed by the European Social Fund in the framework of T\'{A}MOP 4.2.4. A/2-11-1-2012-0001 ``National Excellence Program''. GR was supported by the European Union and the European Social Fund through project FuturICT.hu (grant T\'{A}MOP--4.2.2.C-11/1/KONV-2012-0013), European Research Council StG Nr. 259559,
and  Hungarian Scientific Research Fund OTKA K109782.

\newpage
\section*{Appendix}
The following figure illustrates the path of $L+2$ regions considered in the proof of Theorem \ref{th:pluszeroro>1}.
\begin{figure}[h!]
\centering
\includegraphics[width=7cm]{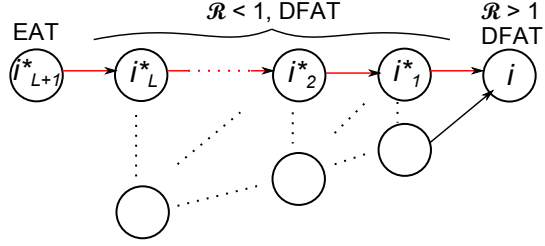} 
\caption{A path of regions $i_{L+1}^*$, $i_L^*$, $\dots$ $i_2^*$, $i_1^*$ and $i$, having the property that regions $i$ and $i_1^*$, $i_2^*$, $\dots$ $i_L^*$ are DFAT, $\rn^i>1$ and $\rn^j<1$ for $j \in \{i_1^*, i_2^*, \dots i_L^*\}$, furthermore region $i_{L+1}^*$ is EAT.}\label{fig:Lpath}
\end{figure}

We present the results of the simulations considered in section \ref{sec:richdyn} in the following figures, which depict solutions of system $(T_1)$--$(T_3)$ with HIV dynamics for four different sets of initial values. \\
For Figures \ref{fig:HIVrl1irred} and \ref{fig:HIVrl1red} initial values were chosen as $S^i(0)=10$, $S_v^i(0)=5$, $Y_2^i(0)=0$, $W_2^i(0)=0$ for $i=1, 2, 3$, and $Y_1^1(0)=1$, $W_1^1(0)=1$, $Y_1^2(0)=0.1$, $W_1^2(0)=0.5$, $Y_1^3(0)=0.1$, $W_1^3(0)=1$ (blue curve), $Y_1^1(0)=0.1$, $W_1^1(0)=1$, $Y_1^2(0)=1$, $W_1^2(0)=1$, $Y_1^3(0)=0.1$, $W_1^3(0)=0.1$ (red curve), $Y_1^1(0)=0.1$, $W_1^1(0)=0.1$, $Y_1^2(0)=1$, $W_1^2(0)=0$, $Y_1^3(0)=1$, $W_1^3(0)=0$ (black curve), $Y_1^1(0)=0.1$, $W_1^1(0)=0.1$, $Y_1^2(0)=1$, $W_1^2(0)=0$, $Y_1^3(0)=0.4$, $W_1^3(0)=0.3$ (green curve).\\

For Figures \ref{fig:HIVrg1nofb}, \ref{fig:HIVrg1feedback} and \ref{fig:HIVrg1full} initial values were chosen as $S^i(0)=10$, $S_v^i(0)=5$, $Y_2^i(0)=0$, $W_2^i(0)=0$ for $i=1, 2, 3$, and $Y_1^1(0)=0.1$, $W_1^1(0)=0.5$, $Y_1^2(0)=0$, $W_1^2(0)=0$, $Y_1^3(0)=0$, $W_1^3(0)=0$ (blue curve), $Y_1^1(0)=1$, $W_1^1(0)=1$, $Y_1^2(0)=0$, $W_1^2(0)=0$, $Y_1^3(0)=0.2$, $W_1^3(0)=0$ (red curve), $Y_1^1(0)=0.4$, $W_1^1(0)=0.3$, $Y_1^2(0)=0$, $W_1^2(0)=0$, $Y_1^3(0)=0$, $W_1^3(0)=0$ (black curve), $Y_1^1(0)=1$, $W_1^1(0)=0$, $Y_1^2(0)=5$, $W_1^2(0)=5$, $Y_1^3(0)=0$, $W_1^3(0)=0$ (green curve).

\begin{figure}[p]
\centering
\subfigure[$\alpha=0$]{\includegraphics[width=4.8cm]{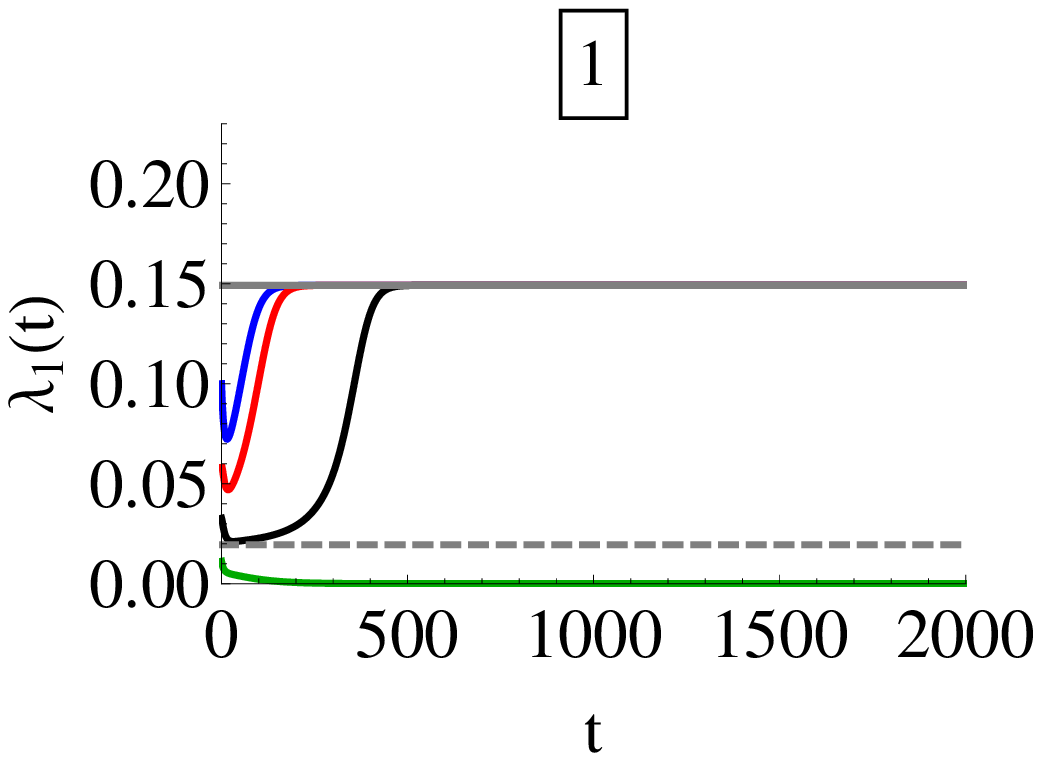} \includegraphics[width=4.8cm]{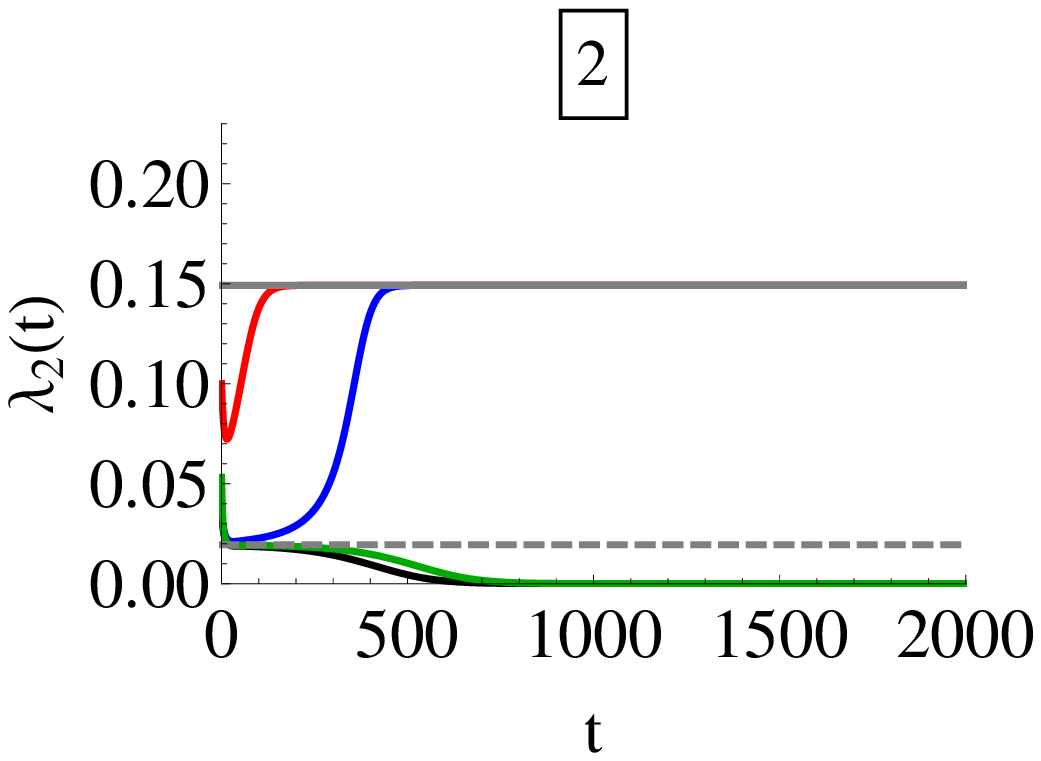} \includegraphics[width=4.8cm]{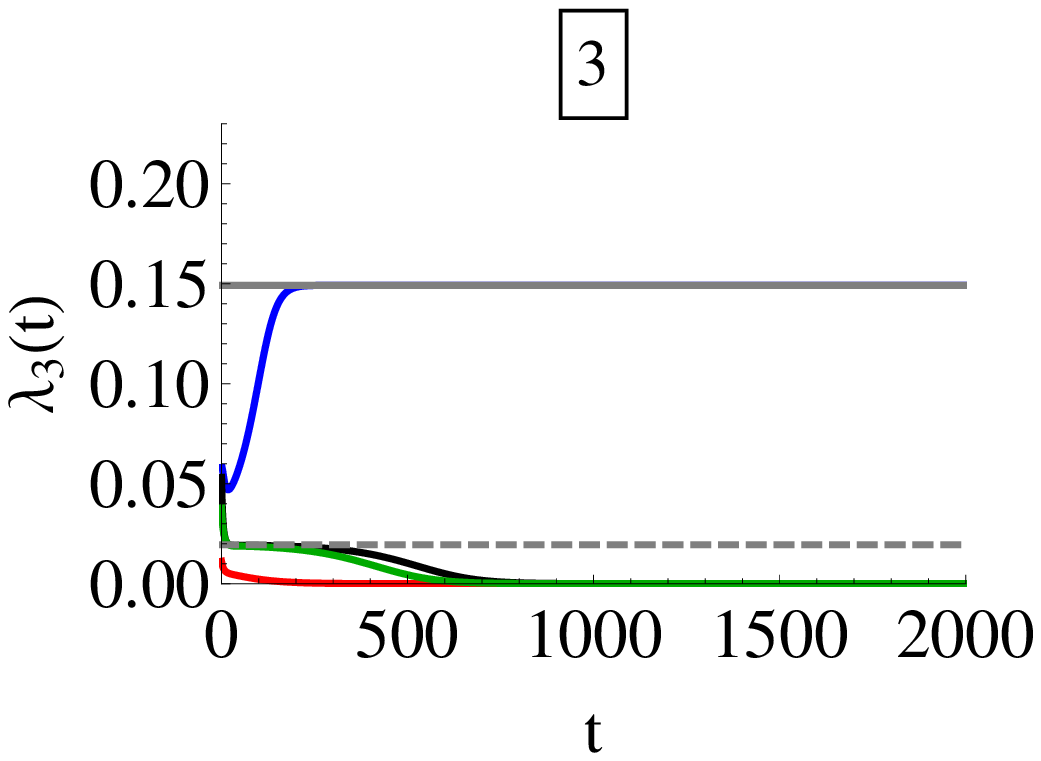}}\\
\subfigure[$\alpha=10^{-5}$]{\includegraphics[width=4.8cm]{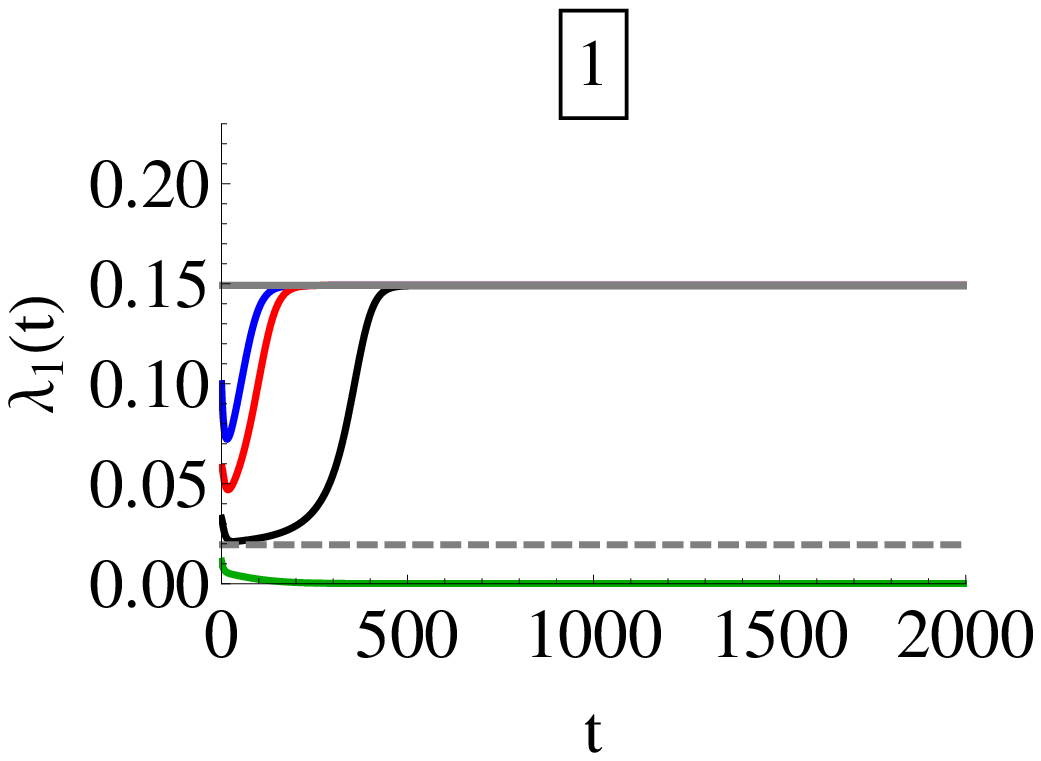} \includegraphics[width=4.8cm]{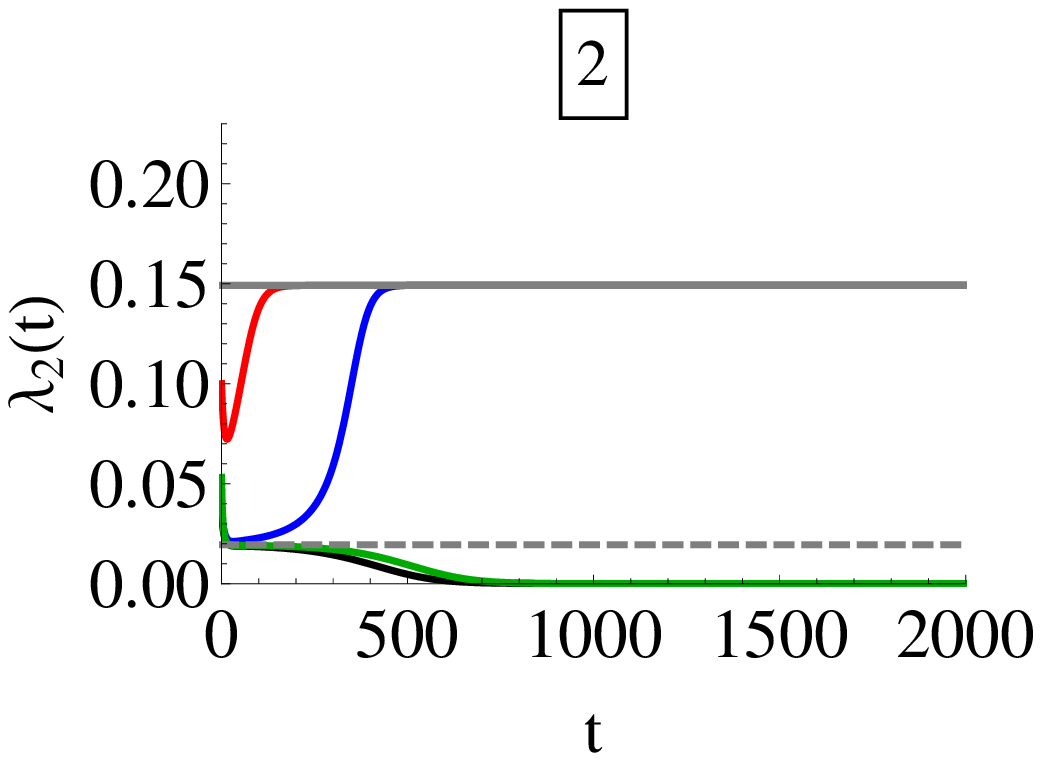} \includegraphics[width=4.8cm]{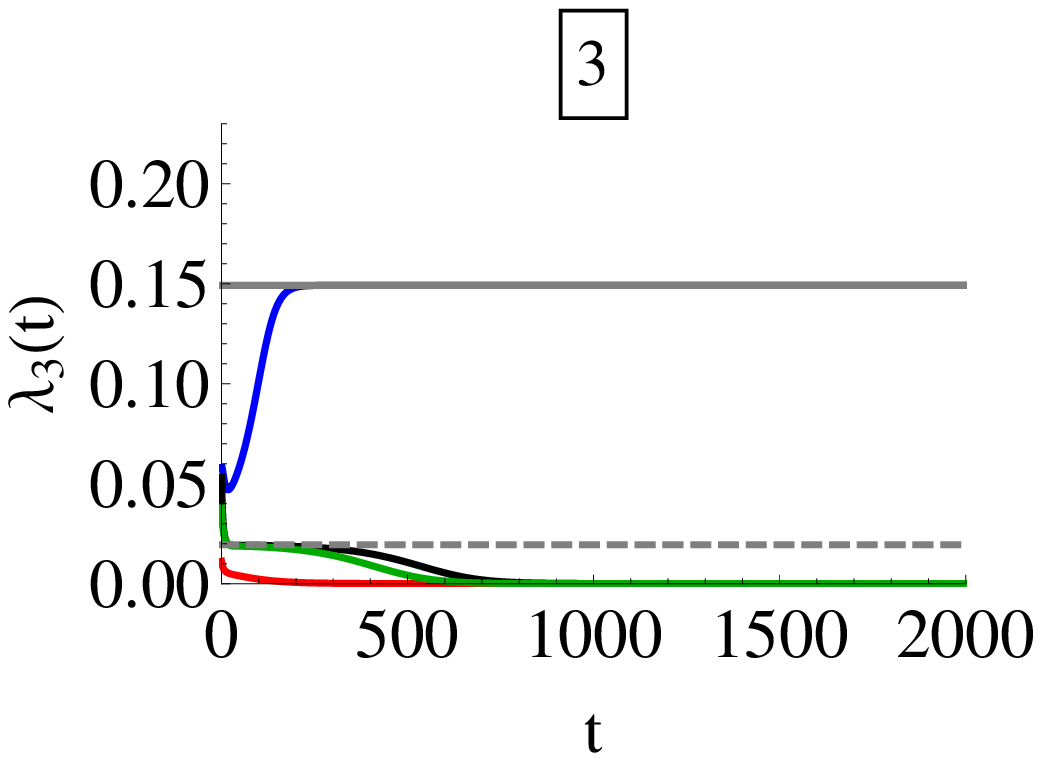}}\\
\subfigure[$\alpha=10^{-3}$]{\includegraphics[width=4.8cm]{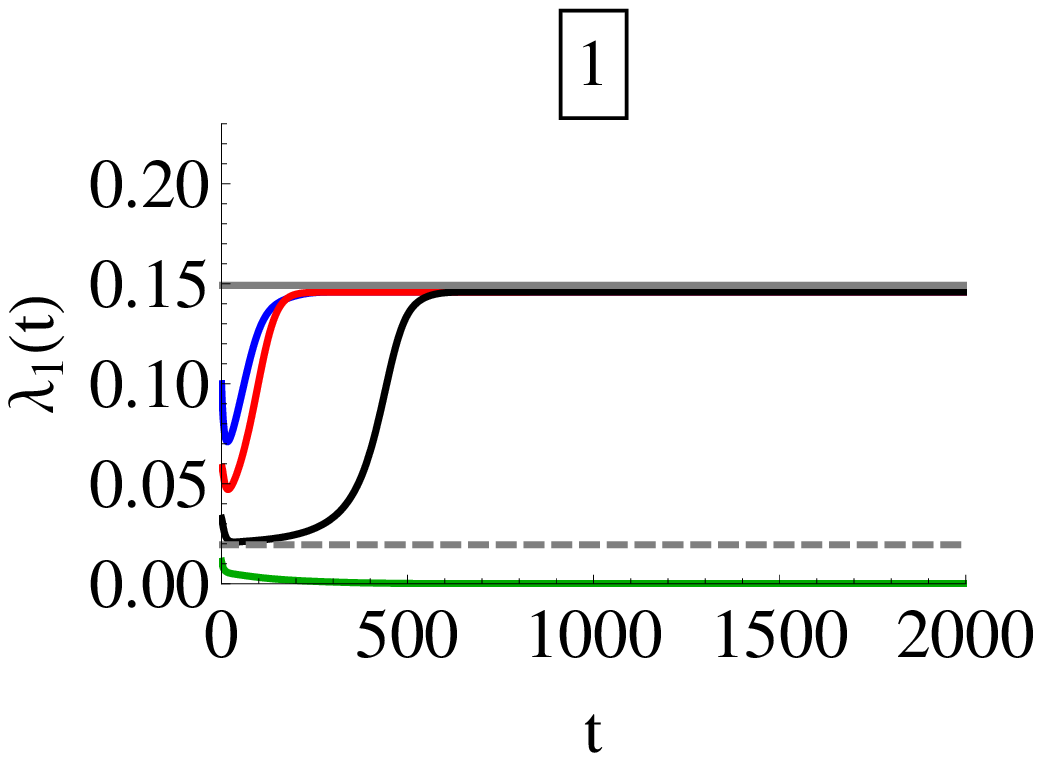} \includegraphics[width=4.8cm]{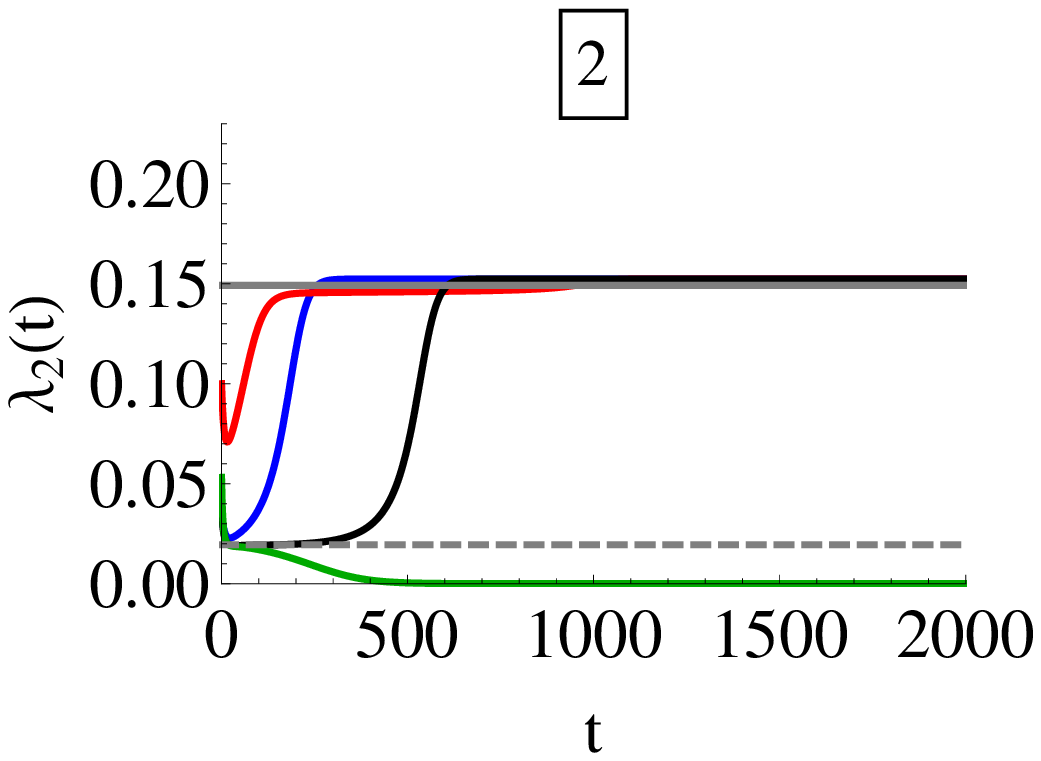} \includegraphics[width=4.8cm]{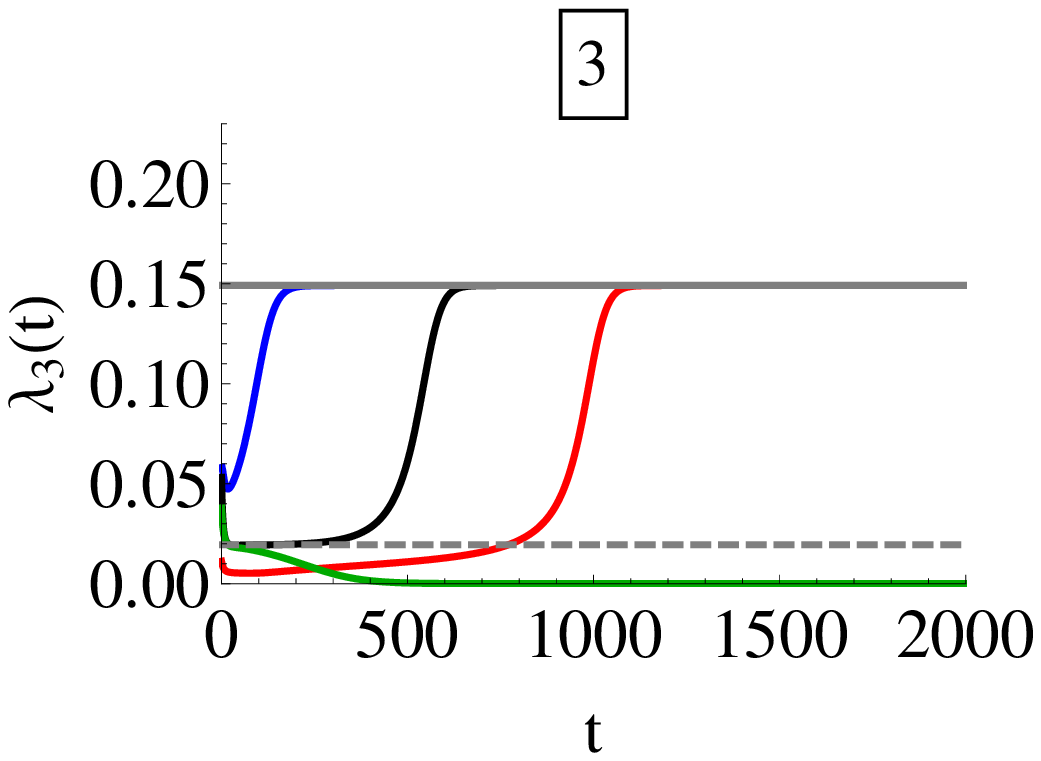}}\\
\subfigure[$\alpha=10^{-1}$]{\includegraphics[width=4.8cm]{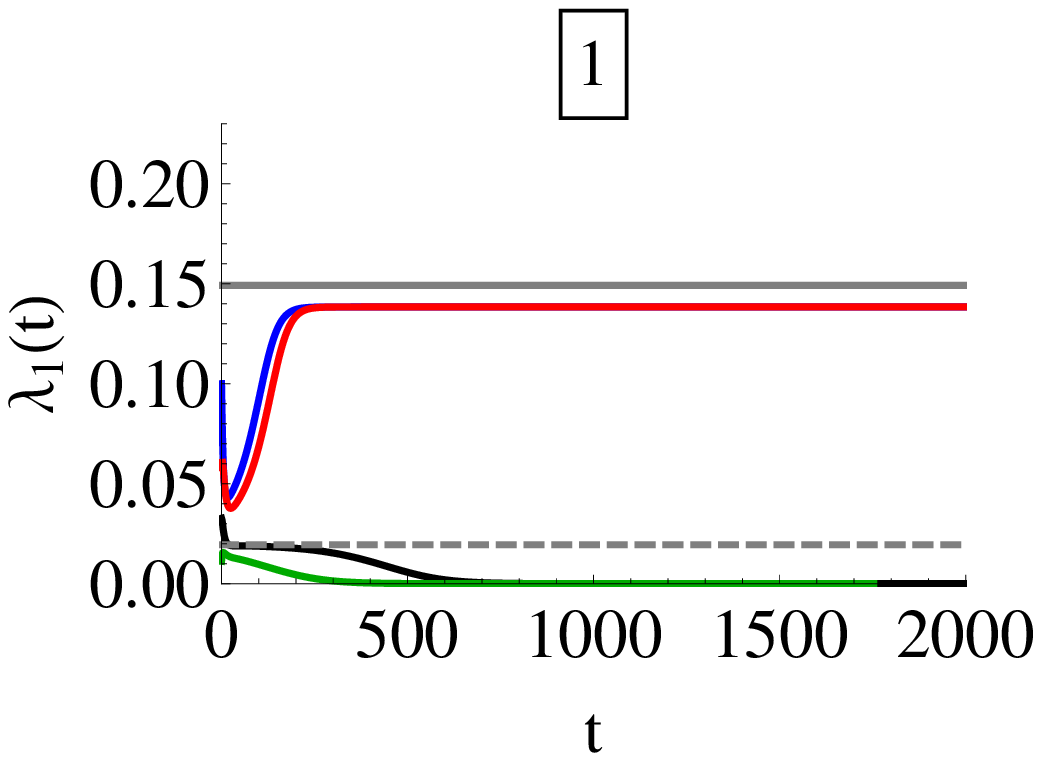} \includegraphics[width=4.8cm]{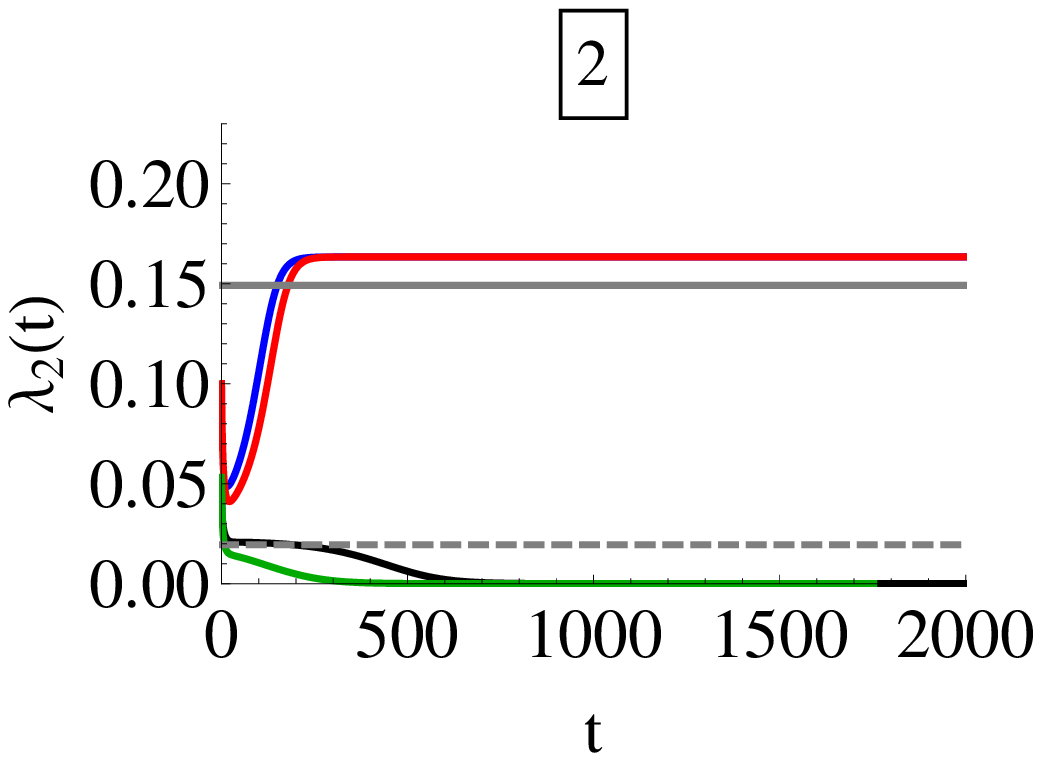} \includegraphics[width=4.8cm]{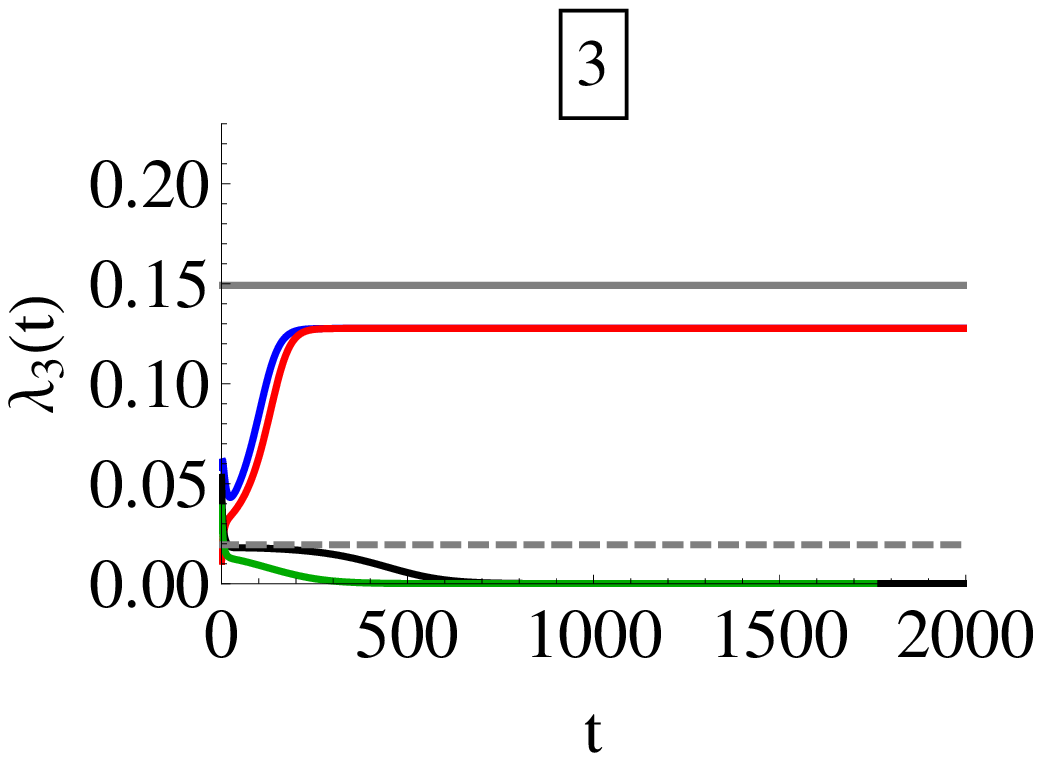}}
\caption{Solutions of system $(T_1)$--$(T_3)$ with HIV dynamics for different travel volumes. The three regions are considered to be symmetric in every parameter value, we applied the parameter set given in section \ref{sec:HIV} so that $\rn_c^{i}<\rn_H^{i}<1$ is satisfied for $i=1, 2, 3$. We use the irreducible connection network depicted in Figure \ref{fig:irredrednetwork} (b), where the connectivity potential parameters $c^{12}$, $c^{21}$, $c^{23}$, $c^{31}$ are equal to one in all model classes. Solid and dashed gray lines correspond to steady state solutions in the regions in the absence of traveling. }\label{fig:HIVrl1irred}
\end{figure}

\begin{figure}[p]
\centering
\subfigure[$\alpha=0$]{\includegraphics[width=4.8cm]{rl1_irred_a0_1.eps} \includegraphics[width=4.8cm]{rl1_irred_a0_2.eps} \includegraphics[width=4.8cm]{rl1_irred_a0_3.eps}}\\
\subfigure[$\alpha=10^{-5}$]{\includegraphics[width=4.8cm]{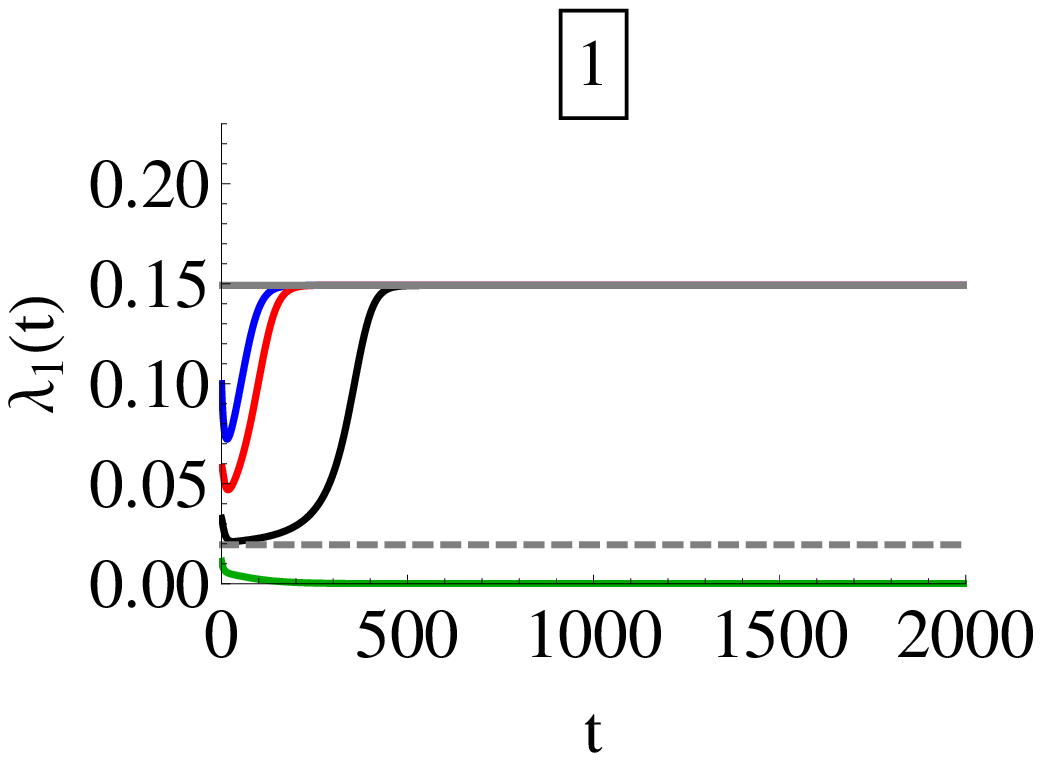} \includegraphics[width=4.8cm]{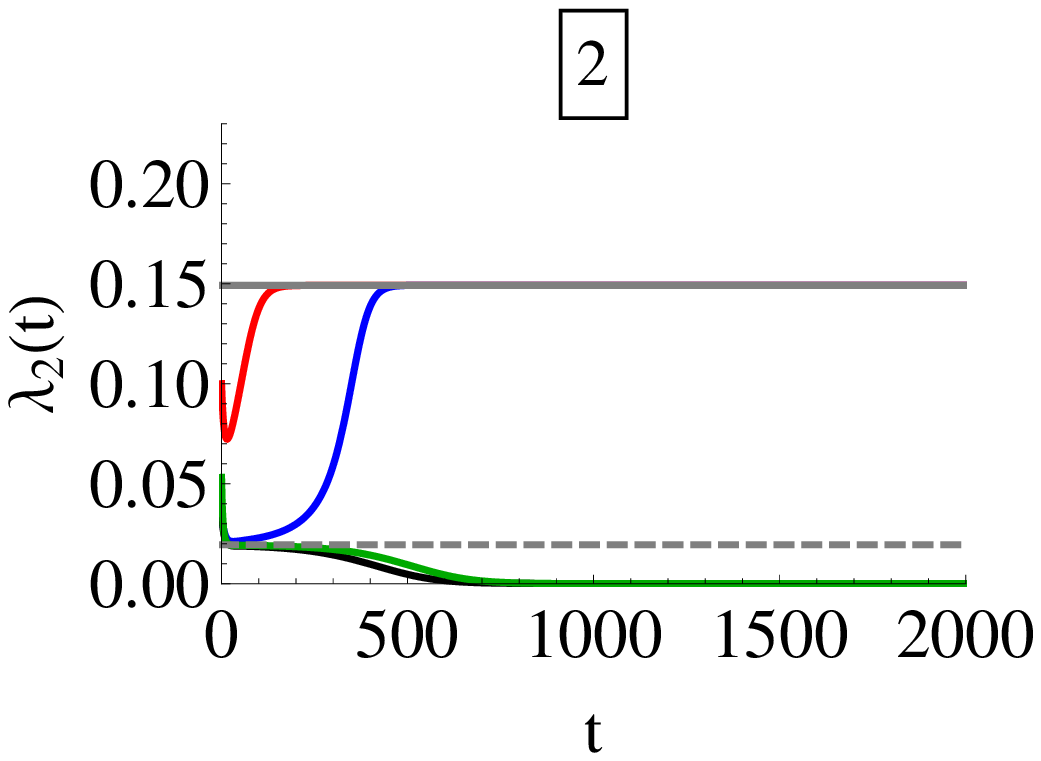} \includegraphics[width=4.8cm]{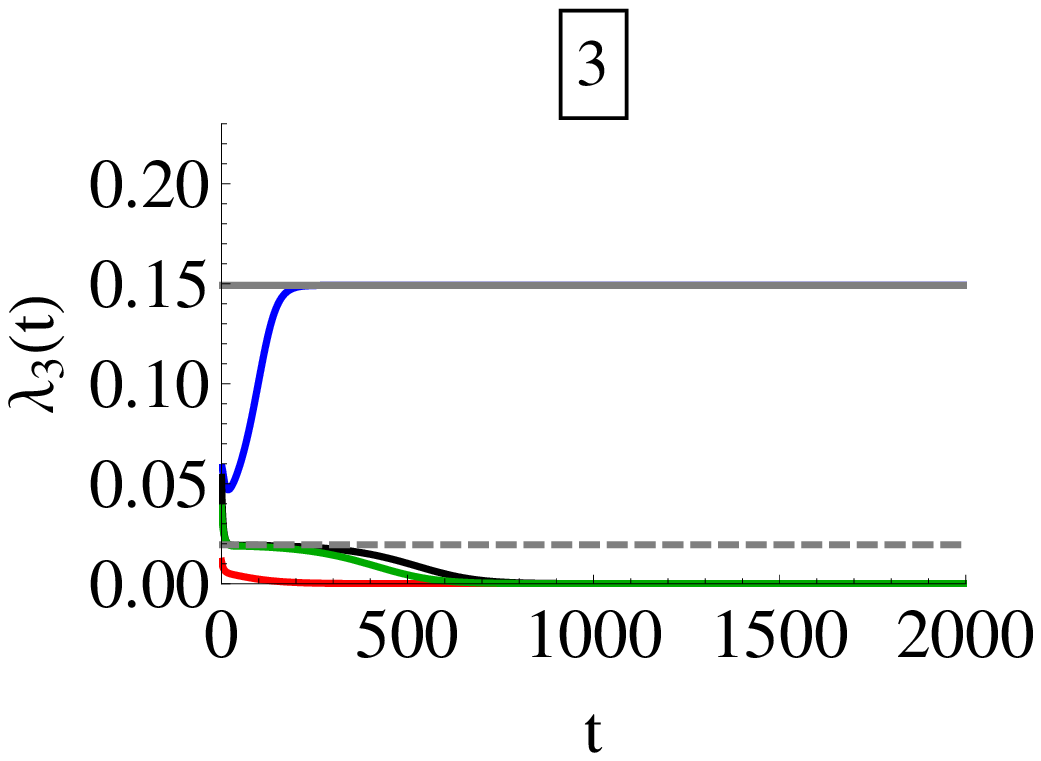}}\\
\subfigure[$\alpha=10^{-3}$]{\includegraphics[width=4.8cm]{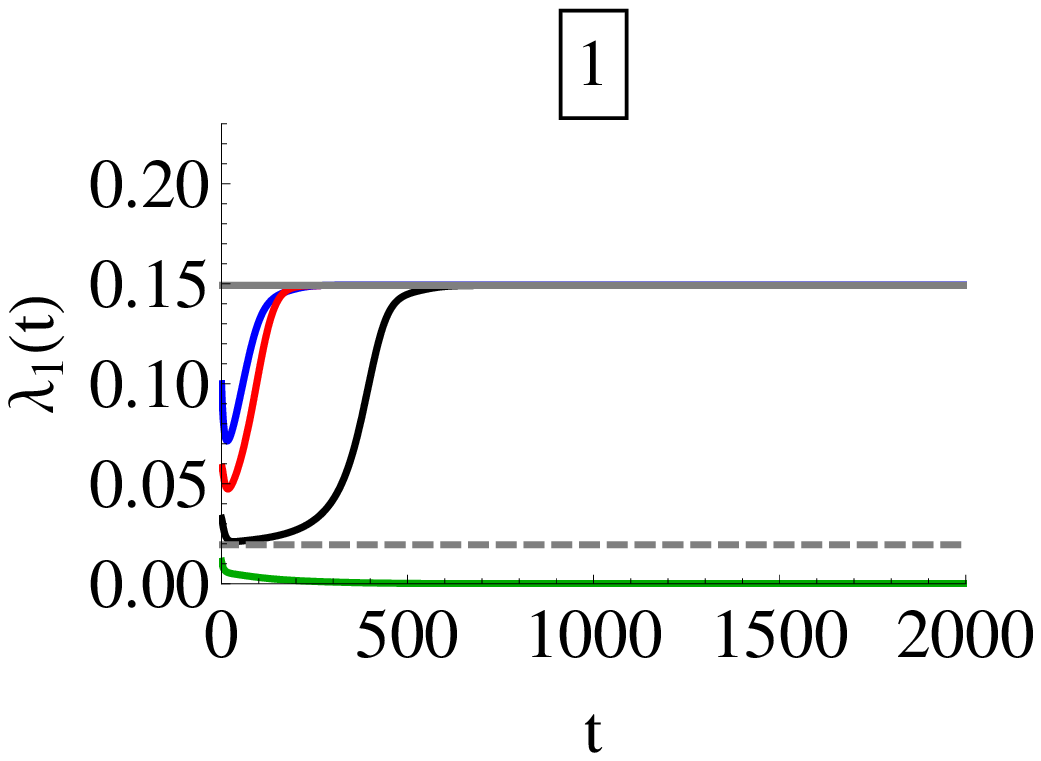} \includegraphics[width=4.8cm]{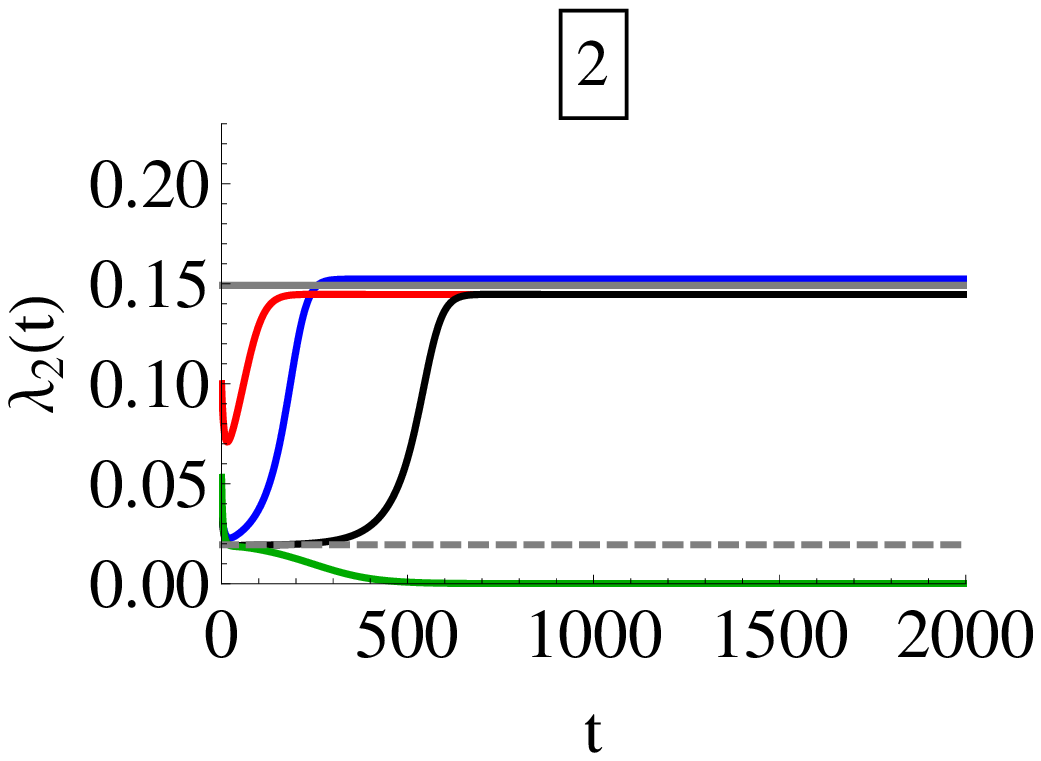} \includegraphics[width=4.8cm]{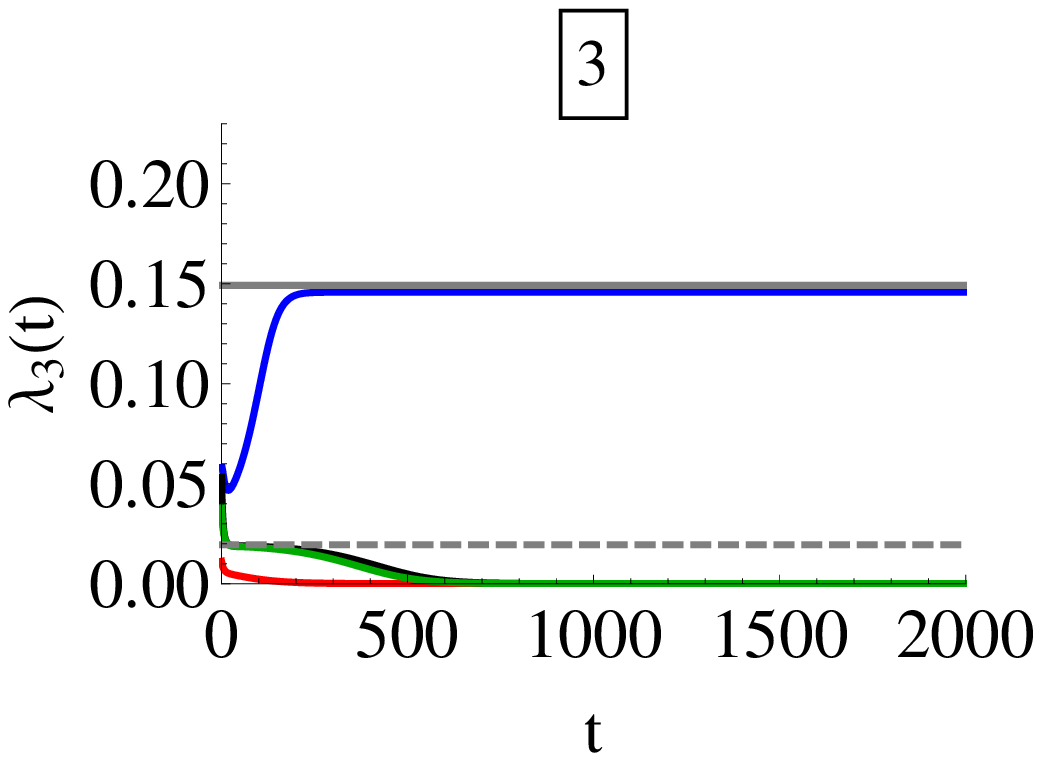}}\\
\subfigure[$\alpha=10^{-1}$]{\includegraphics[width=4.8cm]{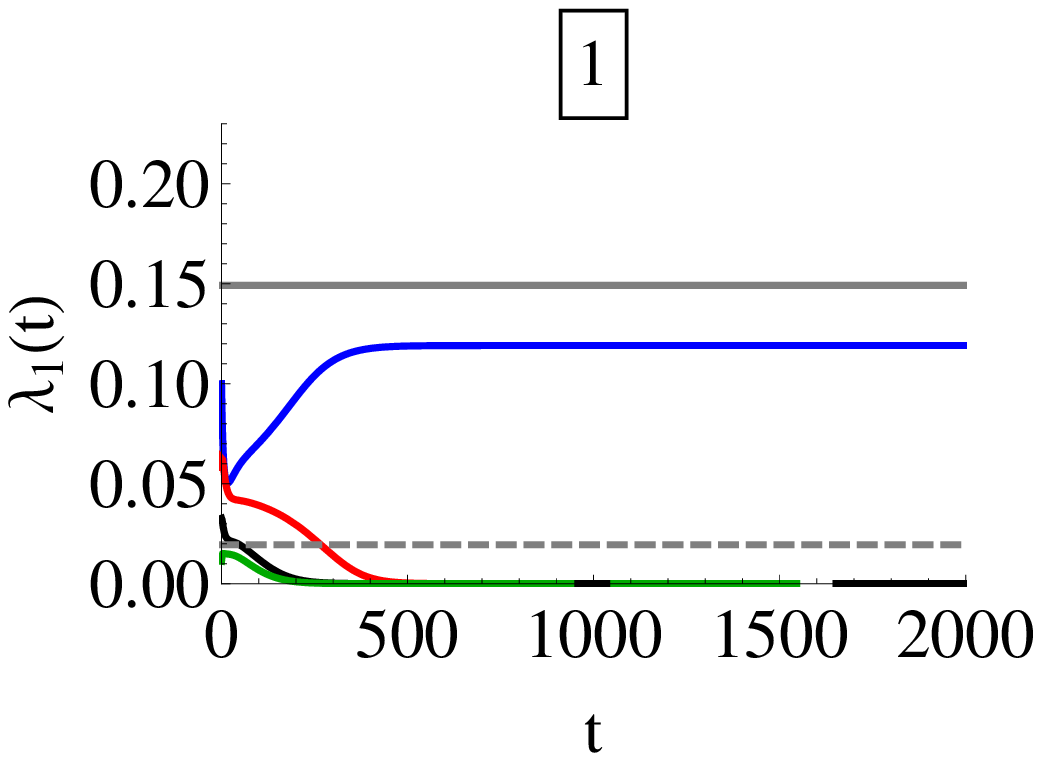} \includegraphics[width=4.8cm]{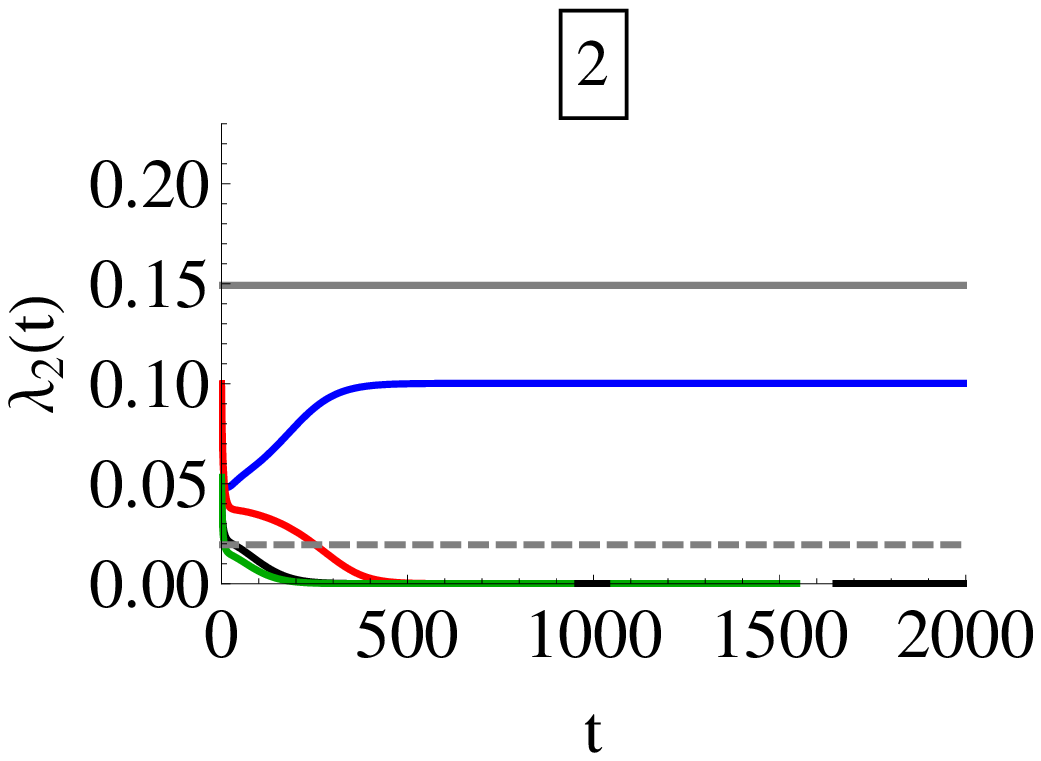} \includegraphics[width=4.8cm]{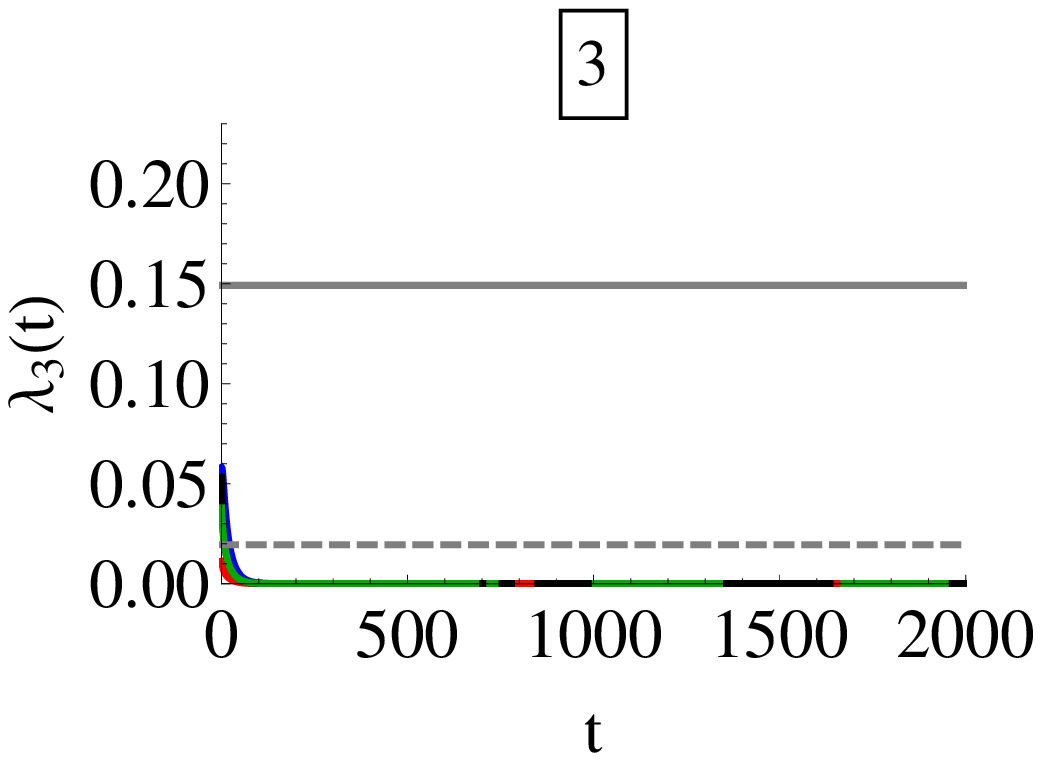}}
\caption{Solutions of system $(T_1)$--$(T_3)$ with HIV dynamics for different travel volumes. The three regions are considered to be symmetric in every parameter value, we applied the parameter set given in section \ref{sec:HIV} so that $\rn_c^{i}<\rn_H^{i}<1$ is satisfied for $i=1, 2, 3$. We use the reducible connection network depicted in Figure \ref{fig:irredrednetwork} (a), where the connectivity potential parameters $c^{12}$, $c^{21}$, $c^{23}$ are equal to one in all model classes.  Solid and dashed gray lines correspond to steady state solutions in the regions in the absence of traveling.}\label{fig:HIVrl1red}
\end{figure}

\begin{figure}[p]
\centering
\subfigure[$\alpha=0$]{\includegraphics[width=4.8cm]{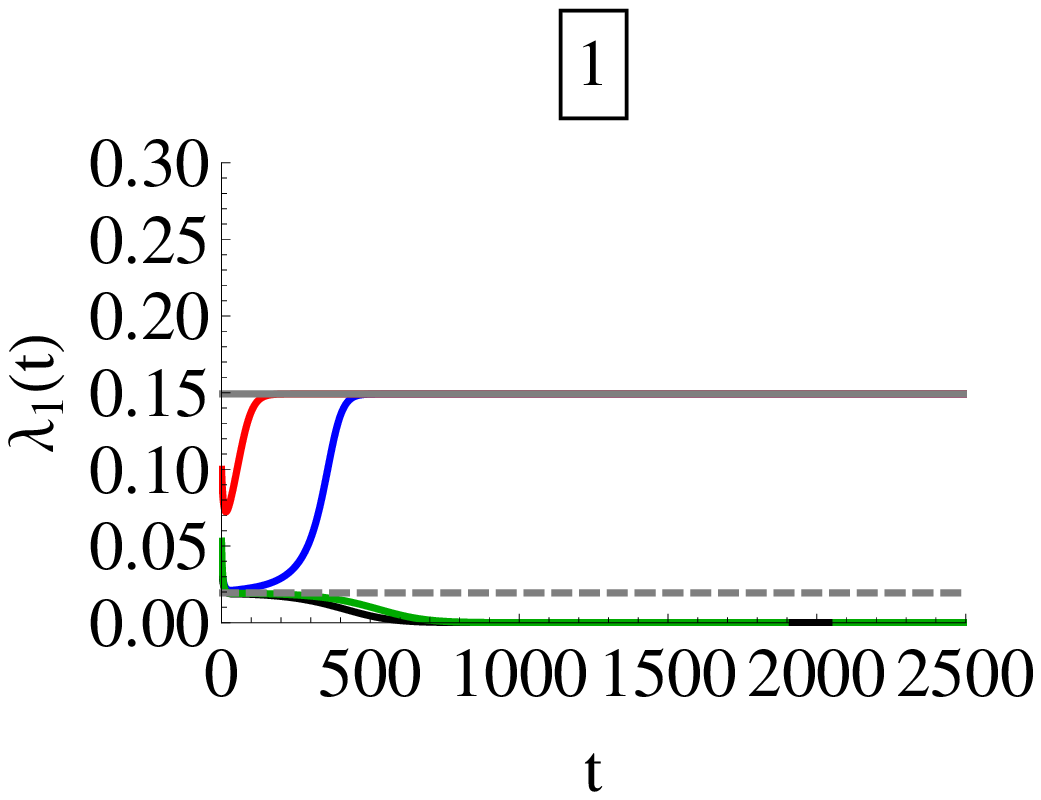} \includegraphics[width=4.8cm]{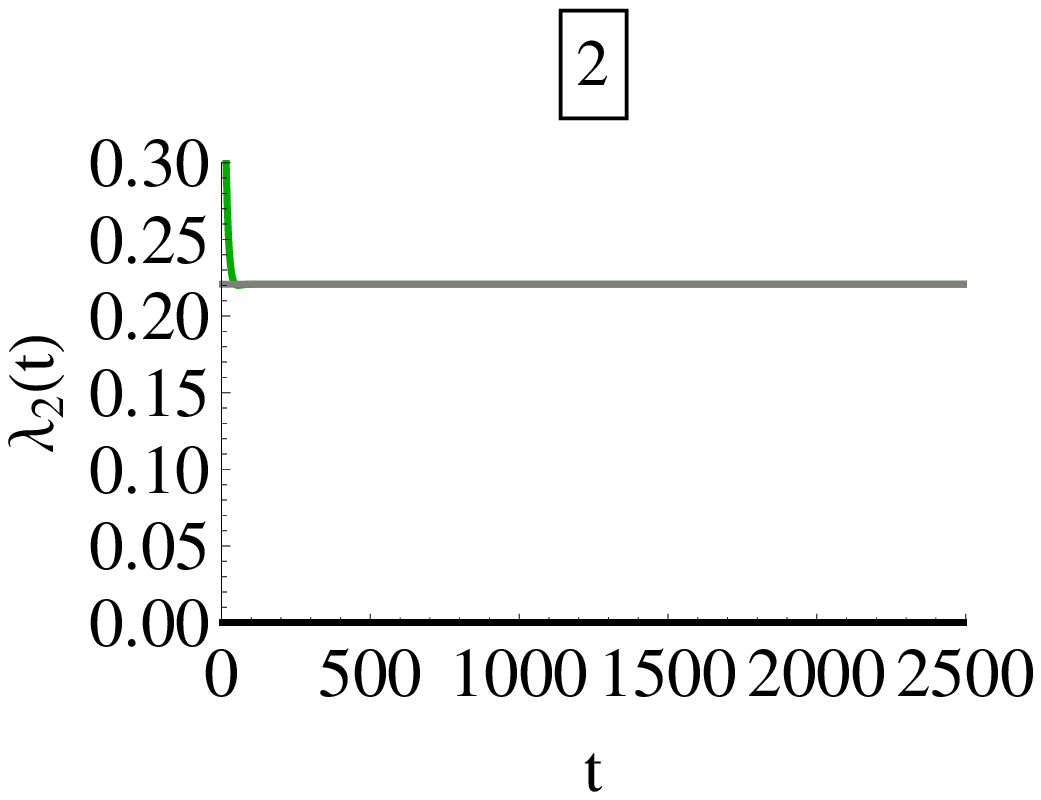} \includegraphics[width=4.8cm]{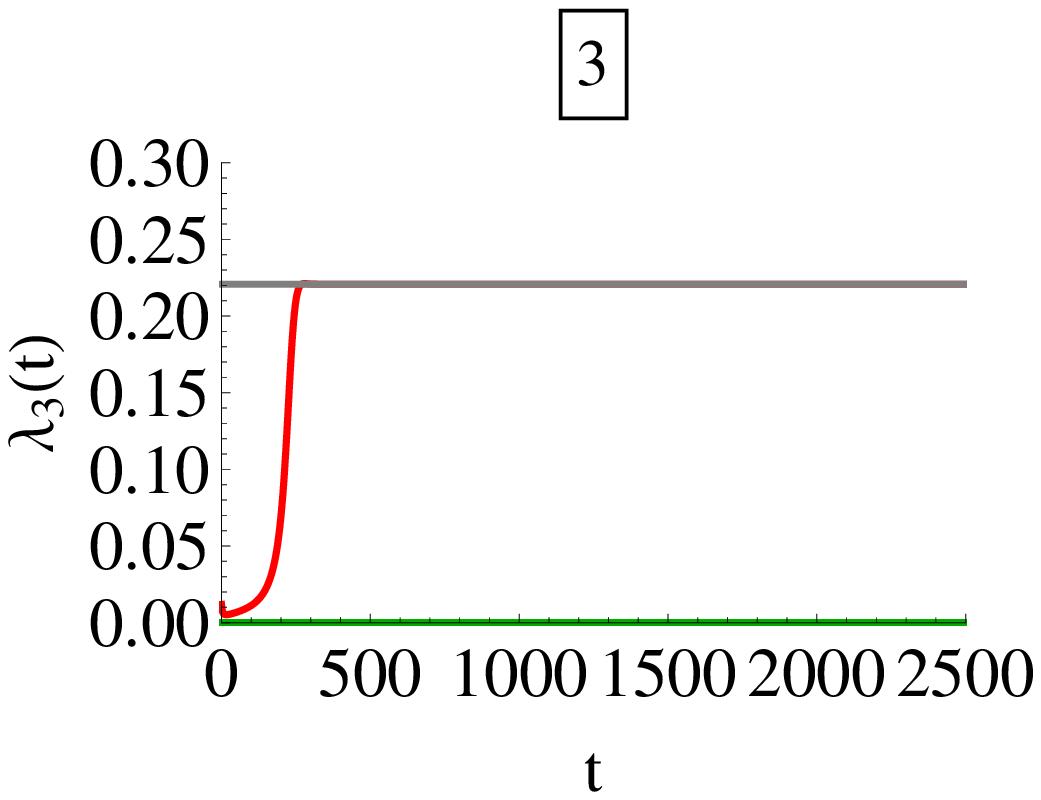}}\\
\subfigure[$\alpha=10^{-5}$]{\includegraphics[width=4.8cm]{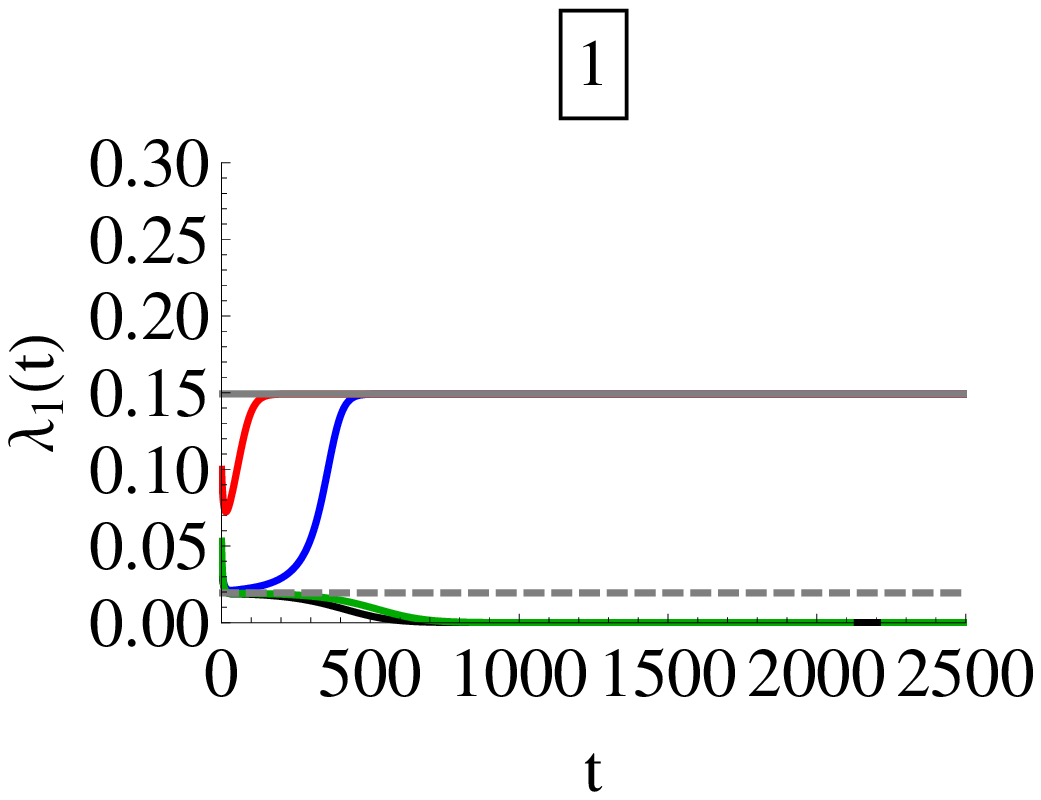} \includegraphics[width=4.8cm]{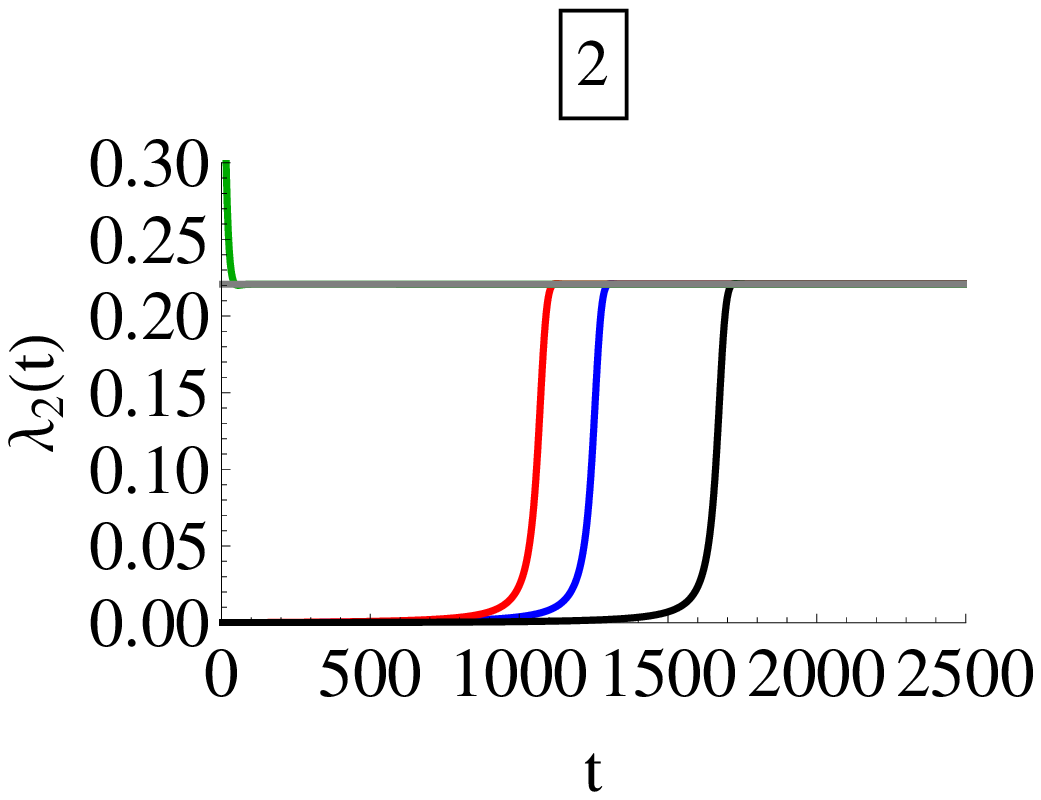} \includegraphics[width=4.8cm]{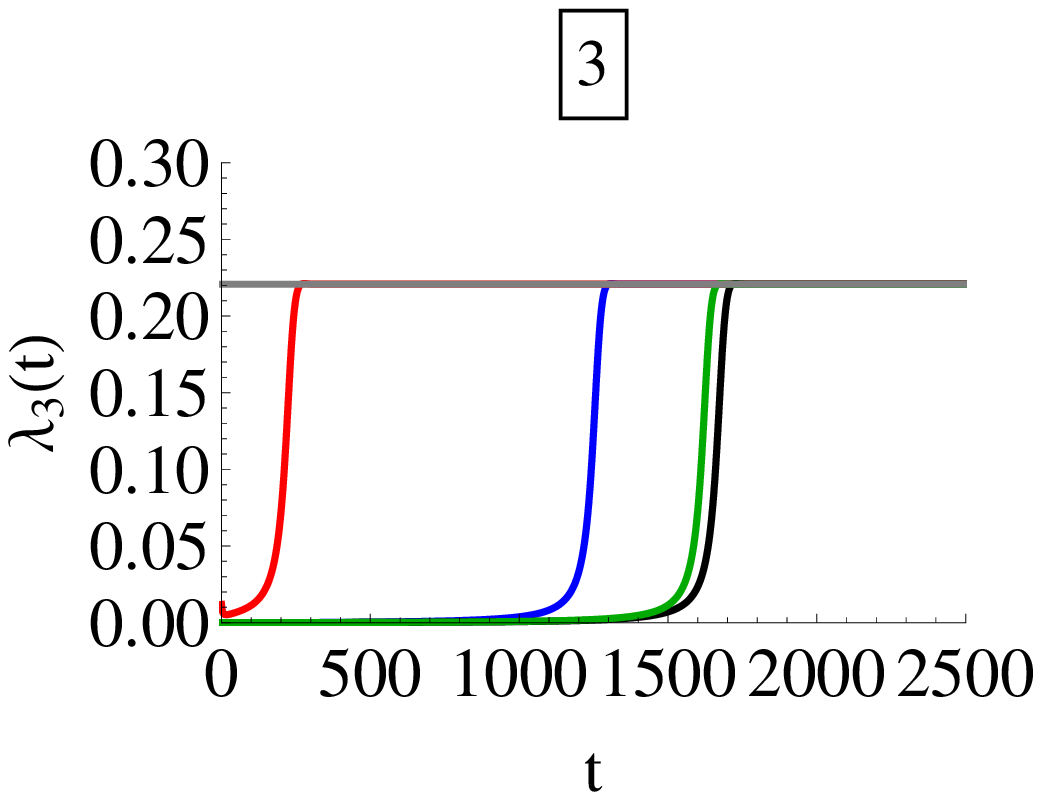}}\\
\subfigure[$\alpha=10^{-3}$]{\includegraphics[width=4.8cm]{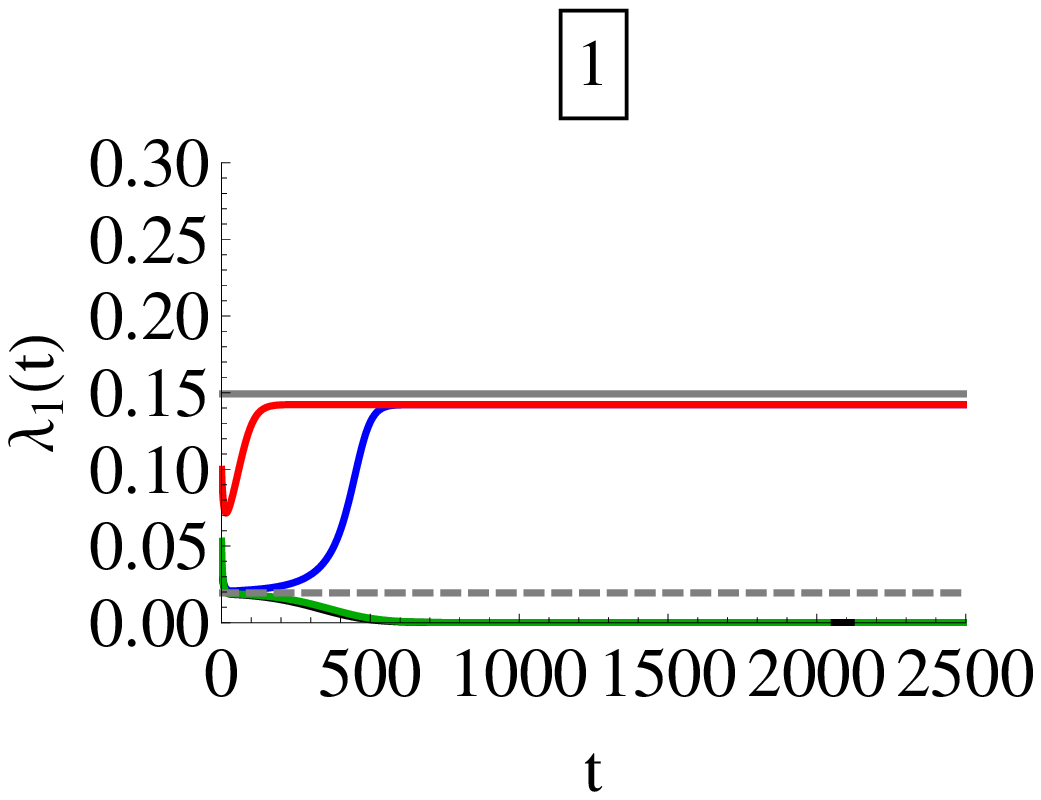} \includegraphics[width=4.8cm]{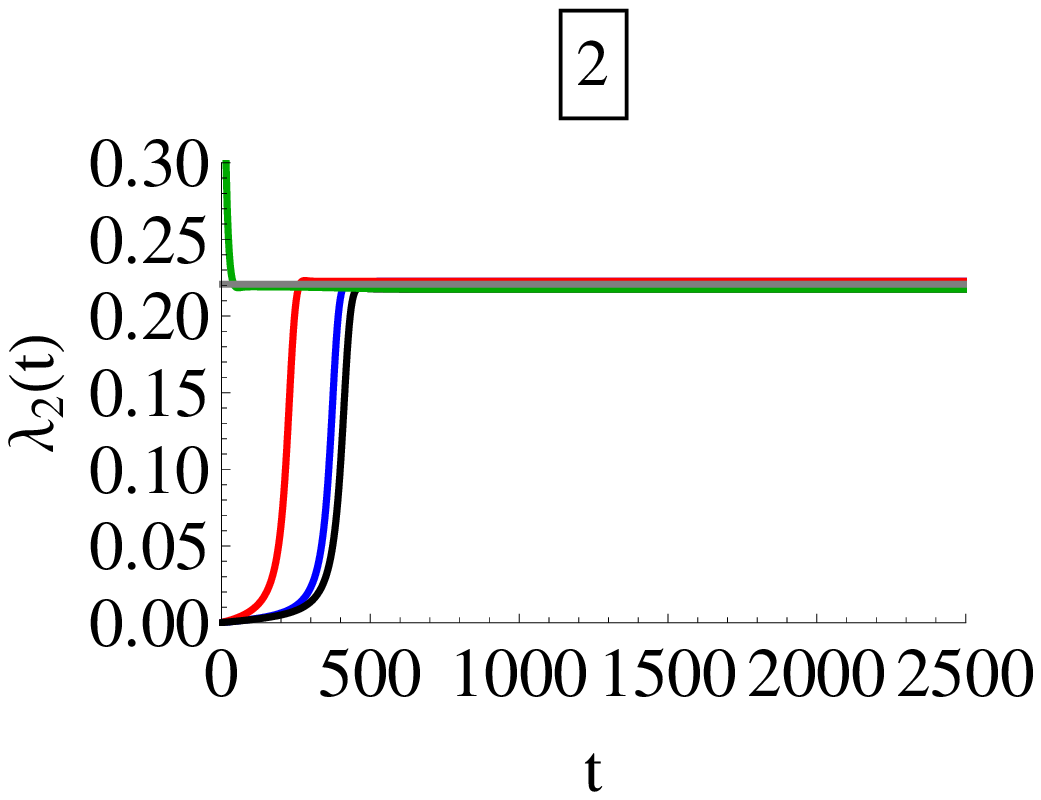} \includegraphics[width=4.8cm]{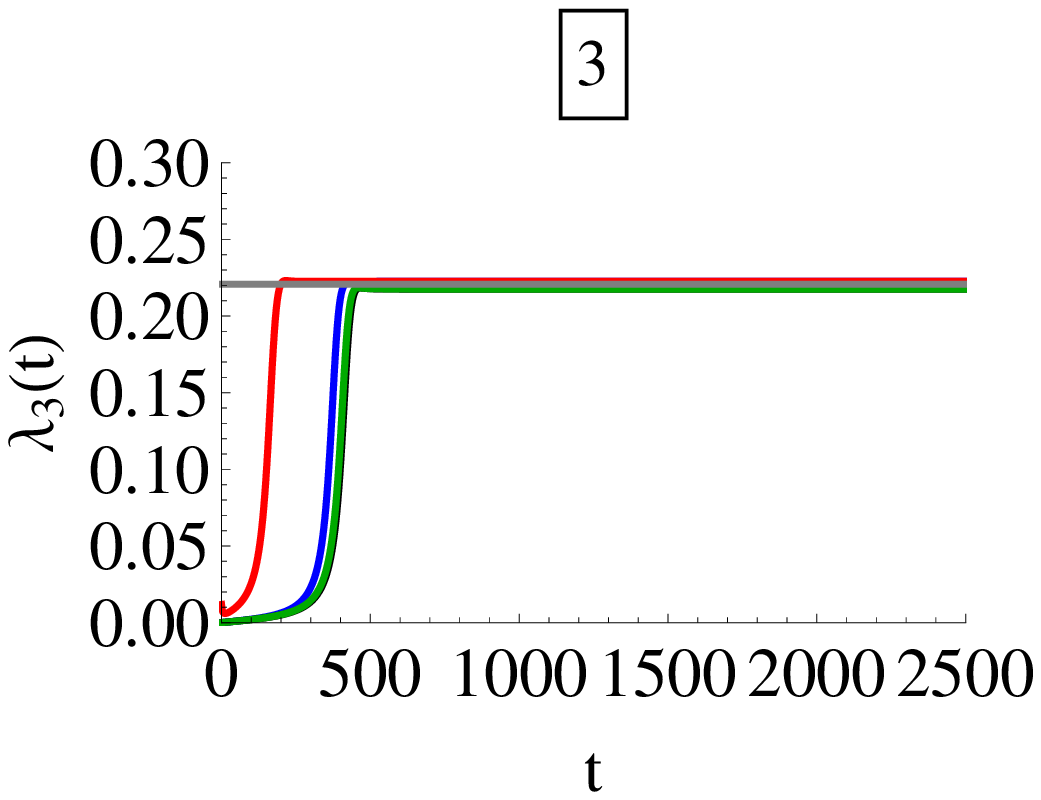}}\\
\subfigure[$\alpha=10^{-1}$]{\includegraphics[width=4.8cm]{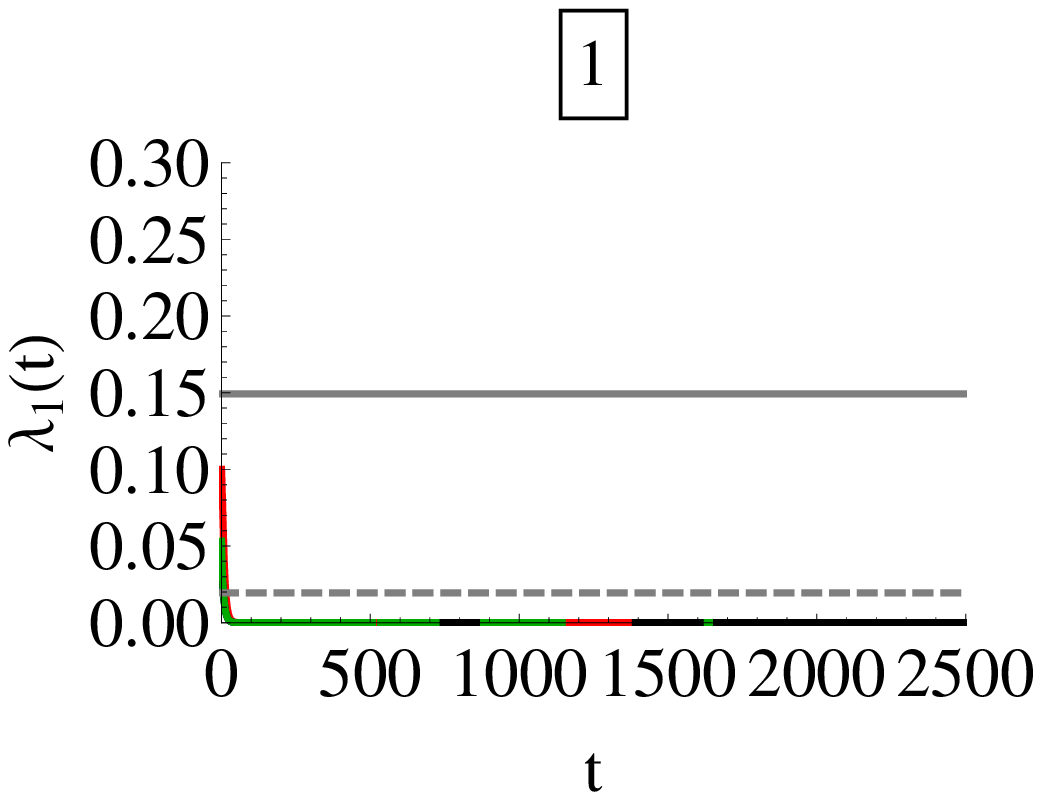} \includegraphics[width=4.8cm]{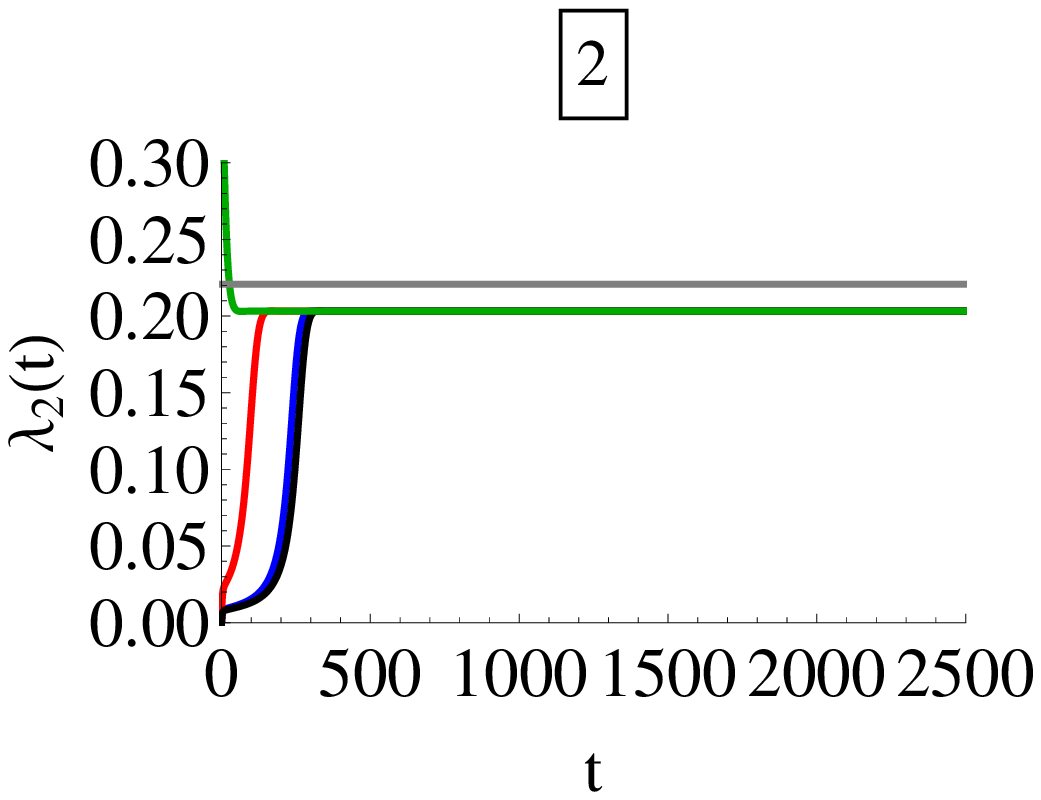} \includegraphics[width=4.8cm]{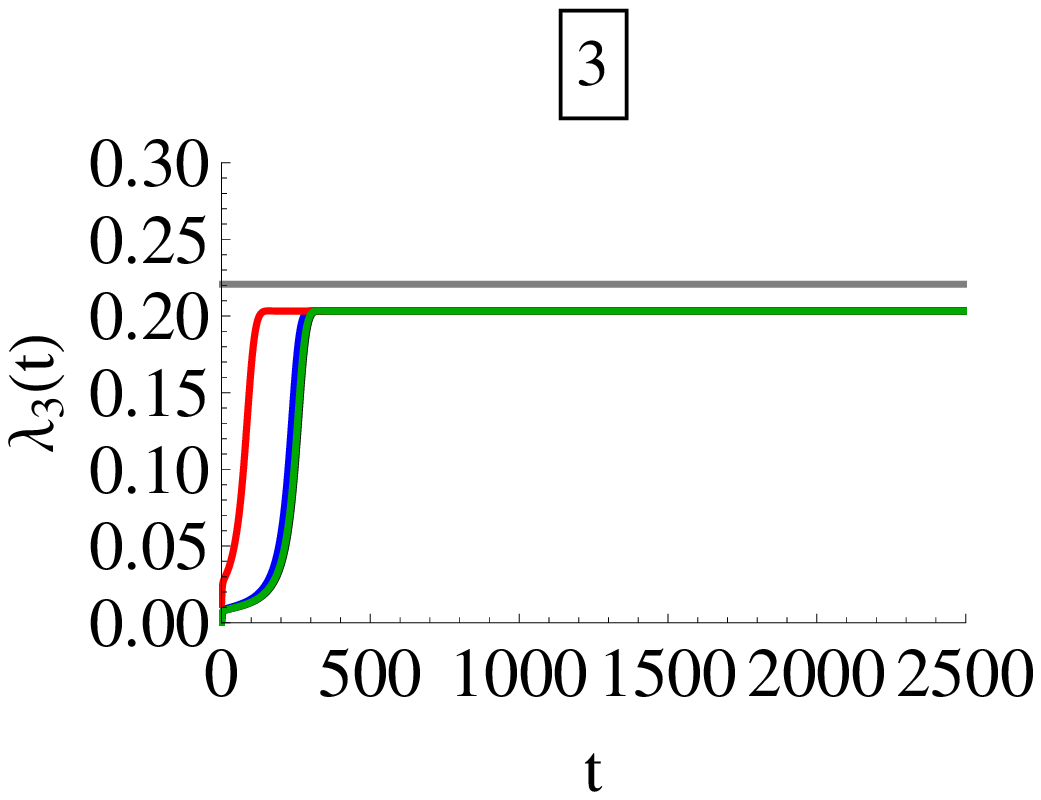}}
\caption{Solutions of system $(T_1)$--$(T_3)$ with HIV dynamics for different travel volumes. The three regions are considered to be symmetric in every parameter value but the local reproduction numbers, we applied the parameter set given in section \ref{sec:HIV} with $\beta^{i}_1=0.85$ and $\beta^{i}_2=\beta^{i}_3=1$ so that $\rn_c^{1}<\rn_H^{1}<1$ and  $\rn_H^{2},\rn_H^{3}>1$ are satisfied. We use the connection network depicted in Figure \ref{fig:4567} (c), where the connectivity potential parameters $c^{21}=c^{31}$ are equal to one in all model classes. Solid and dashed gray lines correspond to steady state solutions in the regions in the absence of traveling. 
}
\label{fig:HIVrg1nofb}
\end{figure}

\begin{figure}[p]
\centering
\subfigure[$\alpha=0$]{\includegraphics[width=4.8cm]{rg1_fig3b_a0_1.eps} \includegraphics[width=4.8cm]{rg1_fig3b_a0_2.eps} \includegraphics[width=4.8cm]{rg1_fig3b_a0_3.eps}}\\
\subfigure[$\alpha=10^{-5}$]{\includegraphics[width=4.8cm]{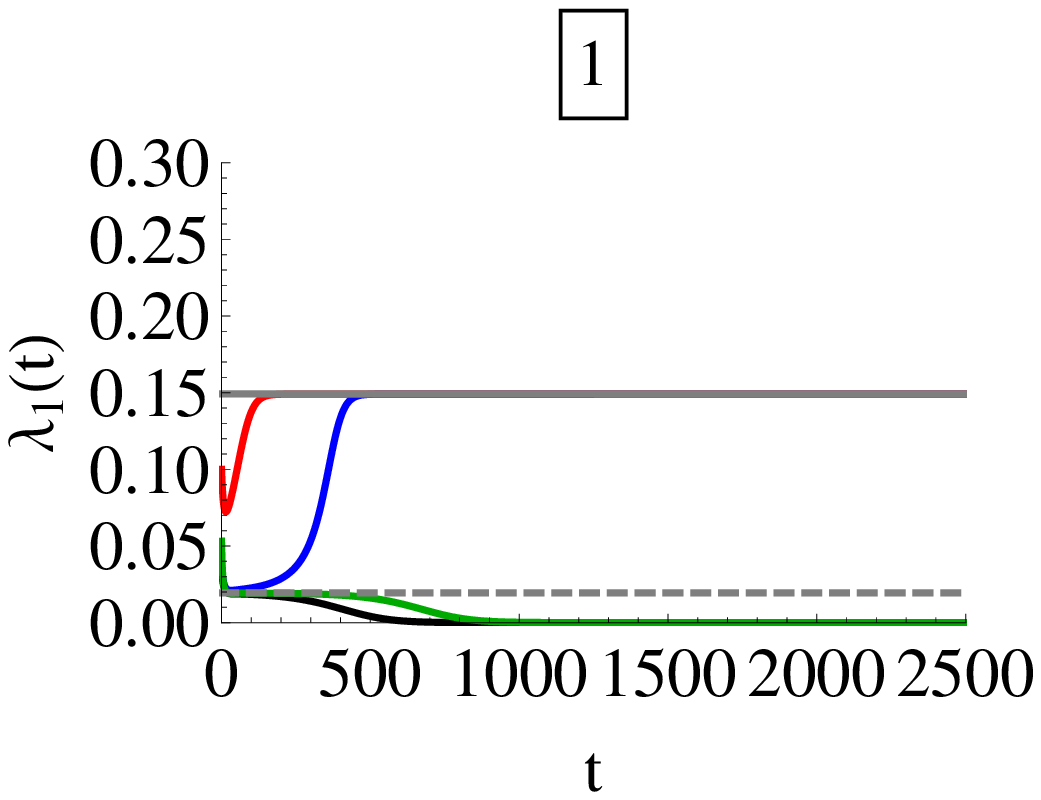} \includegraphics[width=4.8cm]{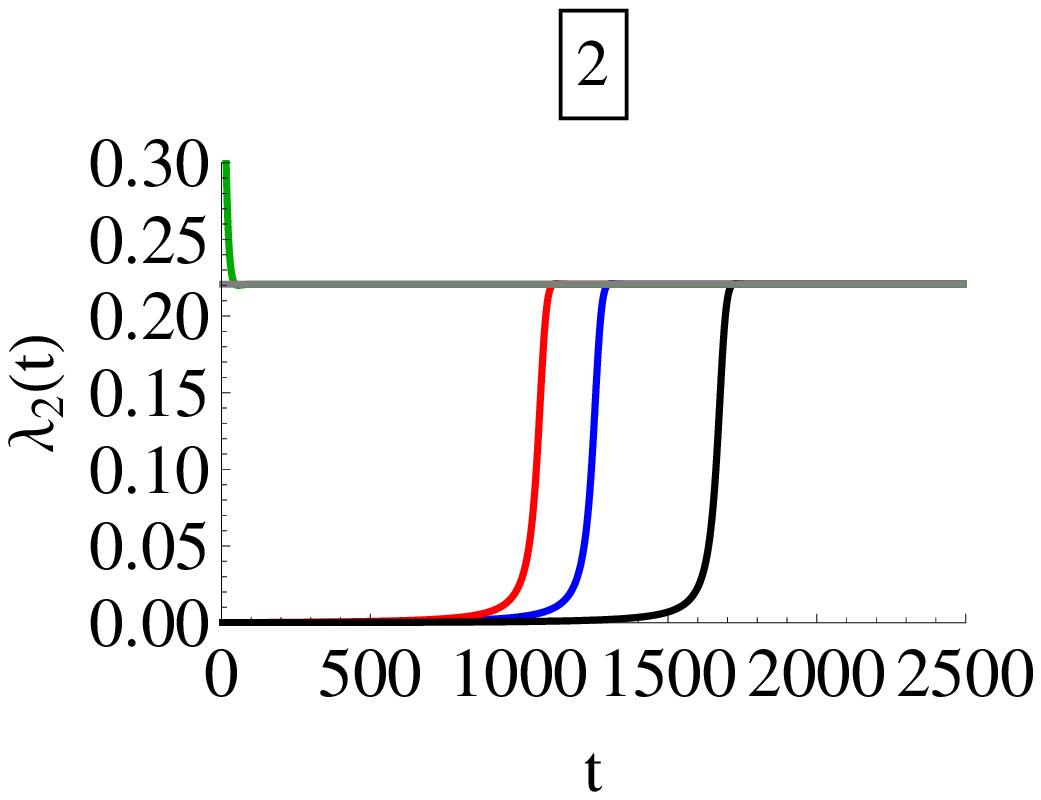} \includegraphics[width=4.8cm]{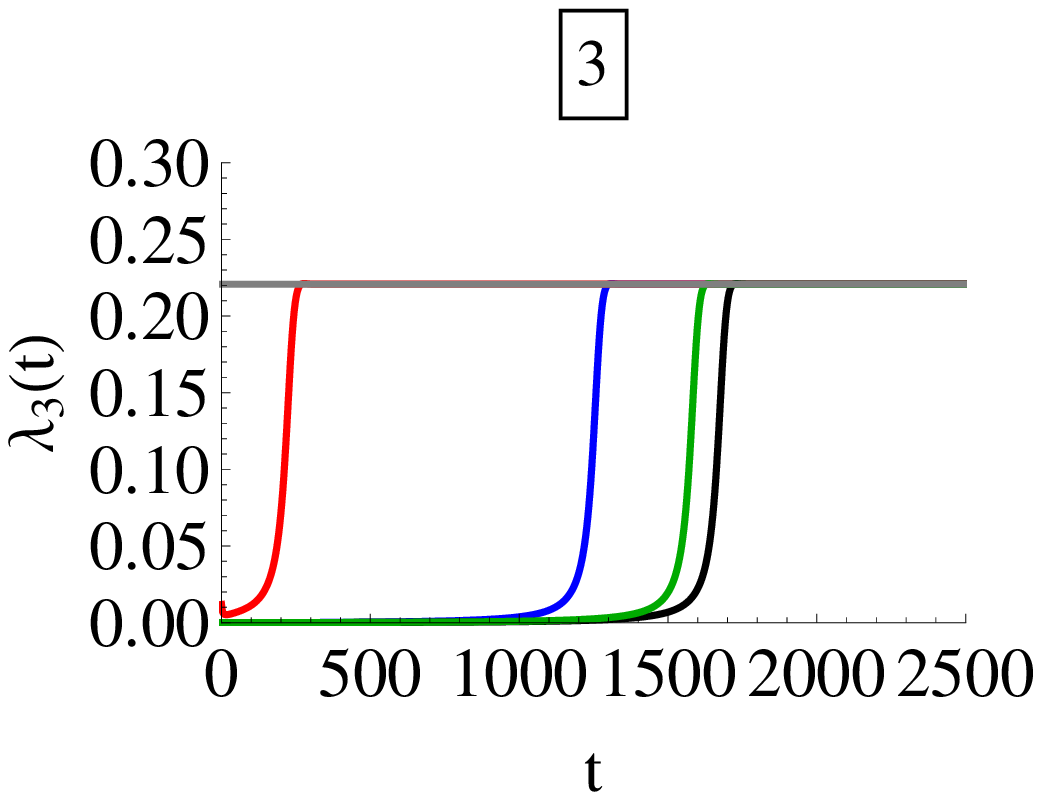}}\\
\subfigure[$\alpha=10^{-3}$]{\includegraphics[width=4.8cm]{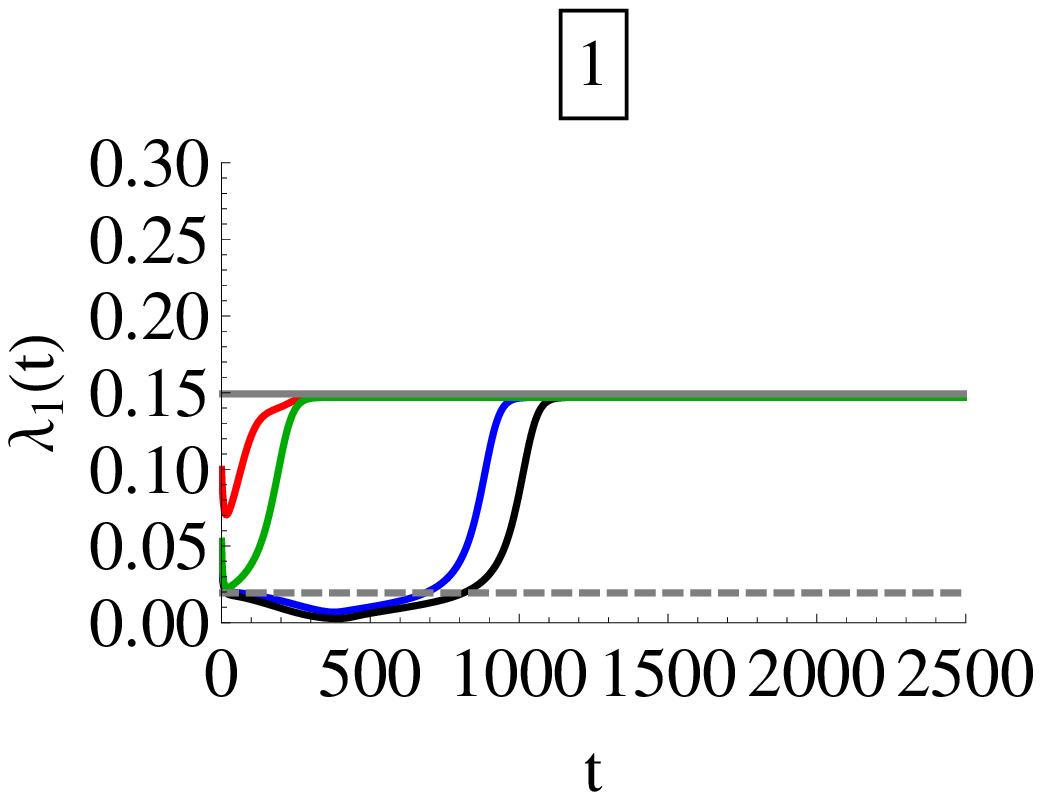} \includegraphics[width=4.8cm]{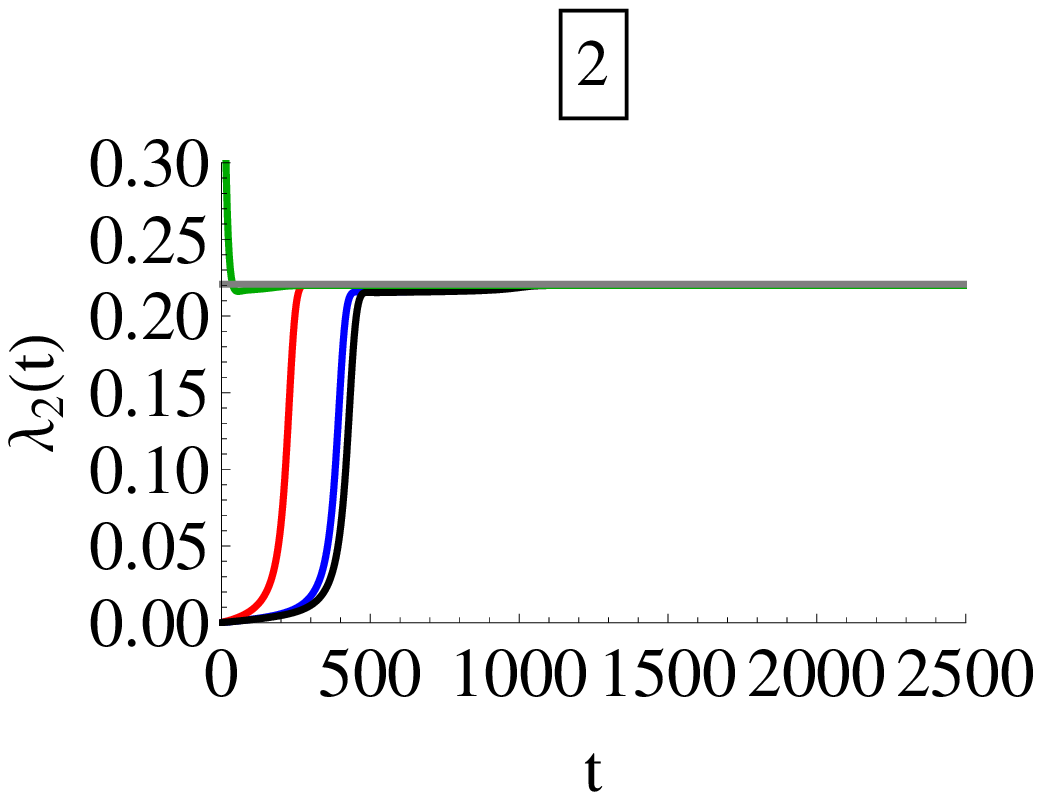} \includegraphics[width=4.8cm]{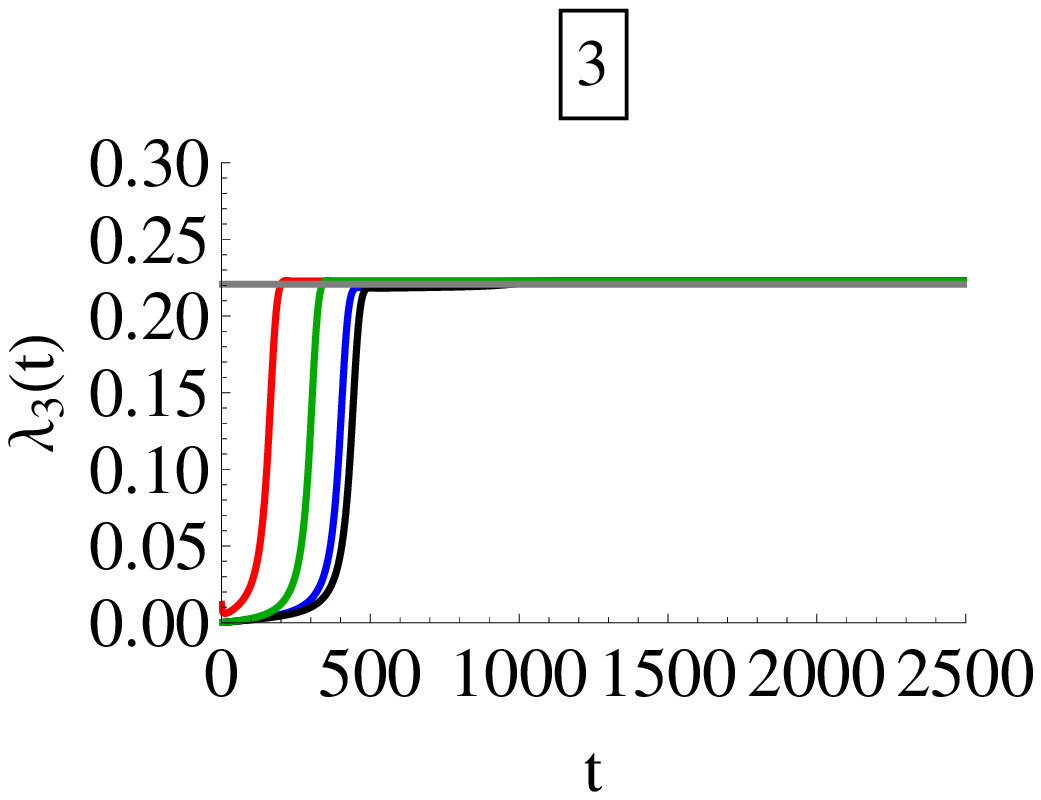}}\\
\subfigure[$\alpha=10^{-1}$]{\includegraphics[width=4.8cm]{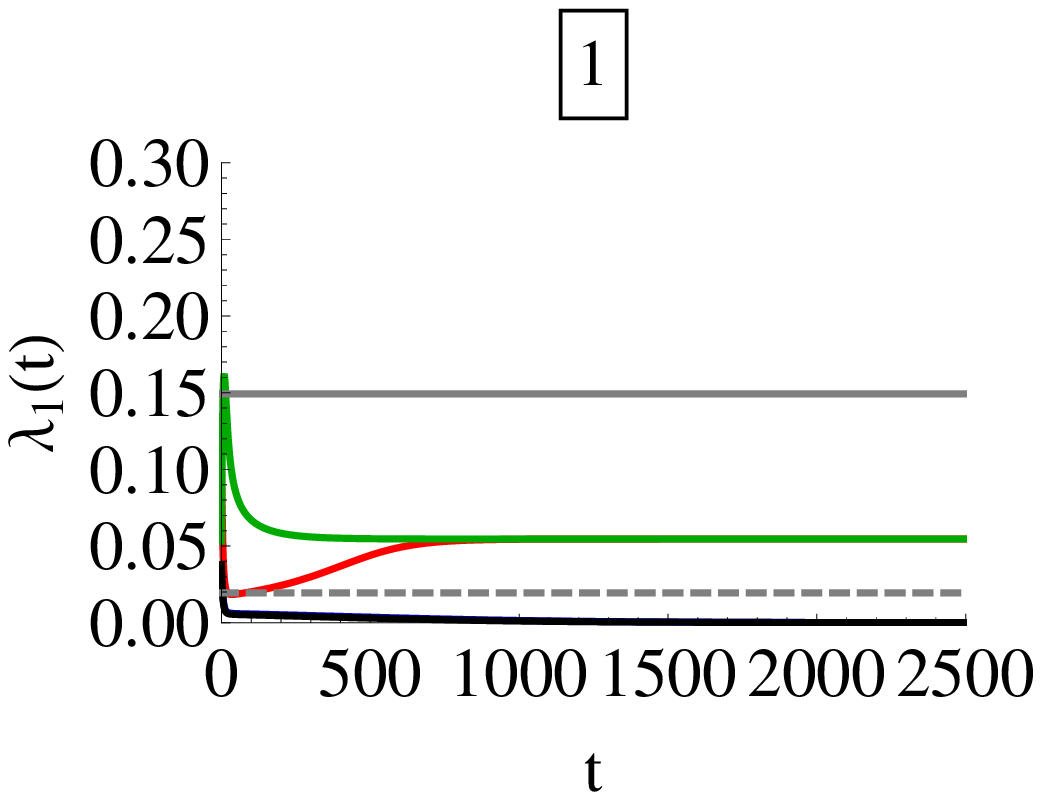} \includegraphics[width=4.8cm]{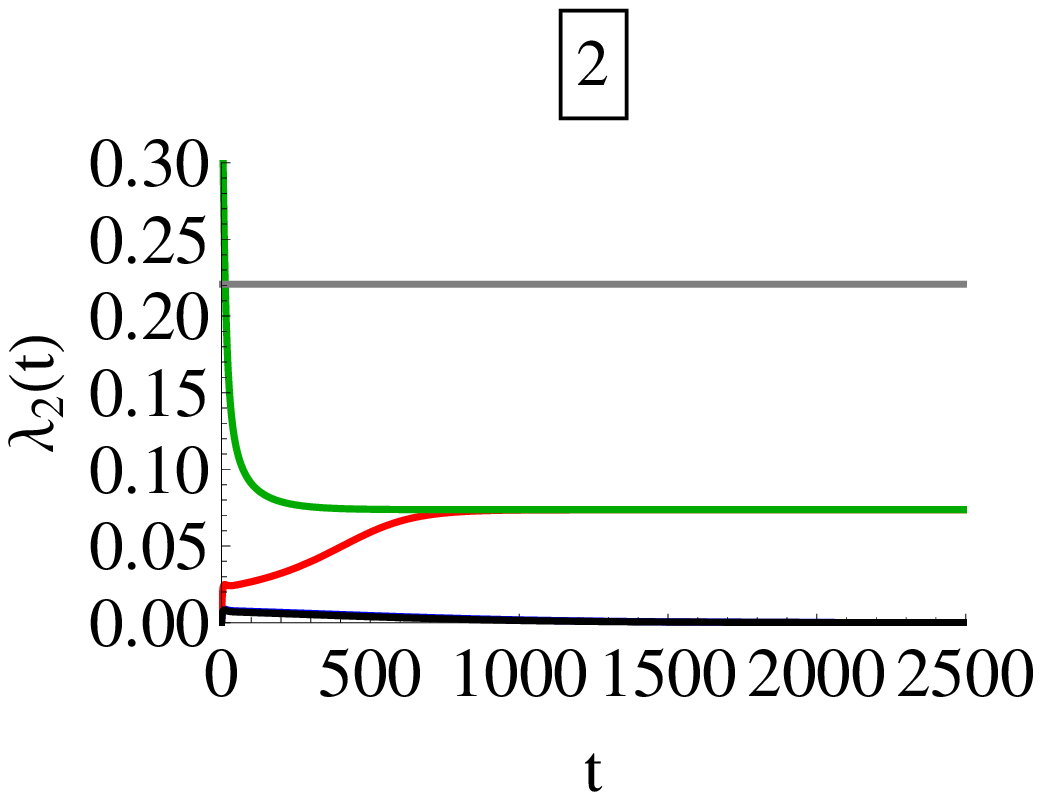} \includegraphics[width=4.8cm]{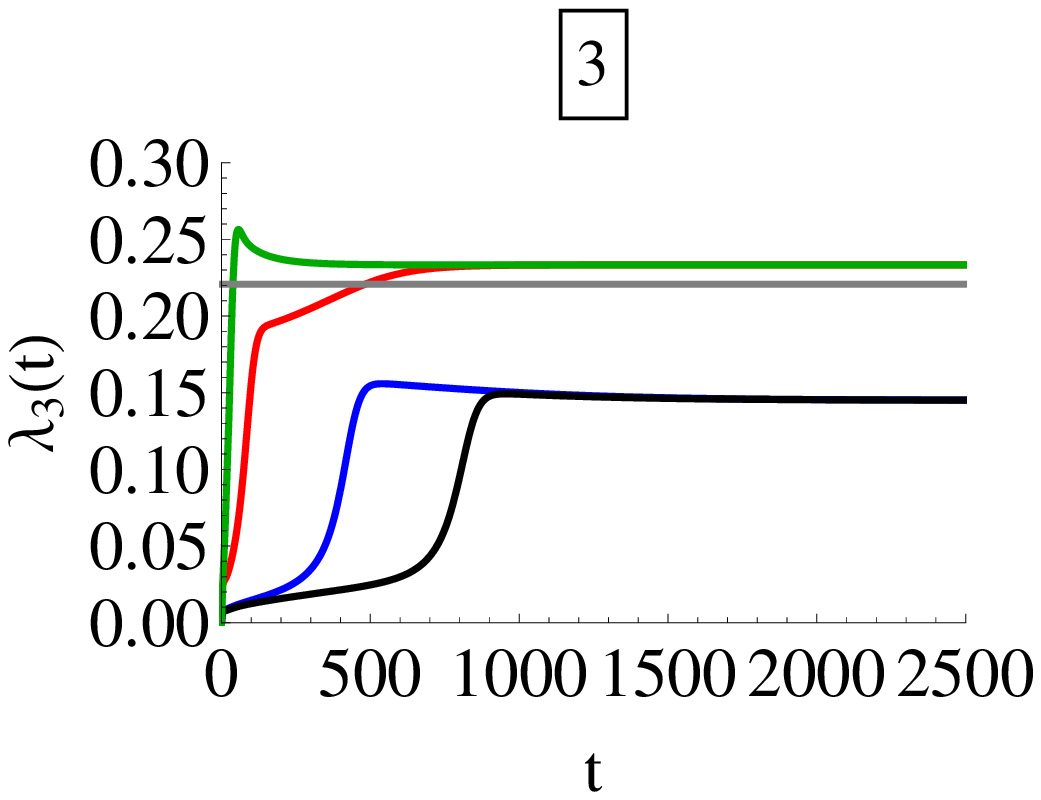}}
\caption{Solutions of system $(T_1)$--$(T_3)$ with HIV dynamics for different travel volumes. The three regions are considered to be symmetric in every parameter value but the local reproduction numbers, we applied the parameter set given in section \ref{sec:HIV} with $\beta^{i}_1=0.85$ and $\beta^{i}_2=\beta^{i}_3=1$ so that $\rn_c^{1}<\rn_H^{1}<1$ and  $\rn_H^{2},\rn_H^{3}>1$ are satisfied. We use the connection network depicted in Figure \ref{fig:4567} (b), where the connectivity potential parameters $c^{12}=c^{21}=c^{31}$ are equal to one in all model classes and zero otherwise. Solid and dashed gray lines correspond to steady state solutions in the regions in the absence of traveling. 
}\label{fig:HIVrg1feedback}
\end{figure}

\begin{figure}[p   ]
\centering
\subfigure[$\alpha=0$]{\includegraphics[width=4.8cm]{rg1_fig3b_a0_1.eps} \includegraphics[width=4.8cm]{rg1_fig3b_a0_2.eps} \includegraphics[width=4.8cm]{rg1_fig3b_a0_3.eps}}\\
\subfigure[$\alpha=10^{-5}$]{\includegraphics[width=4.8cm]{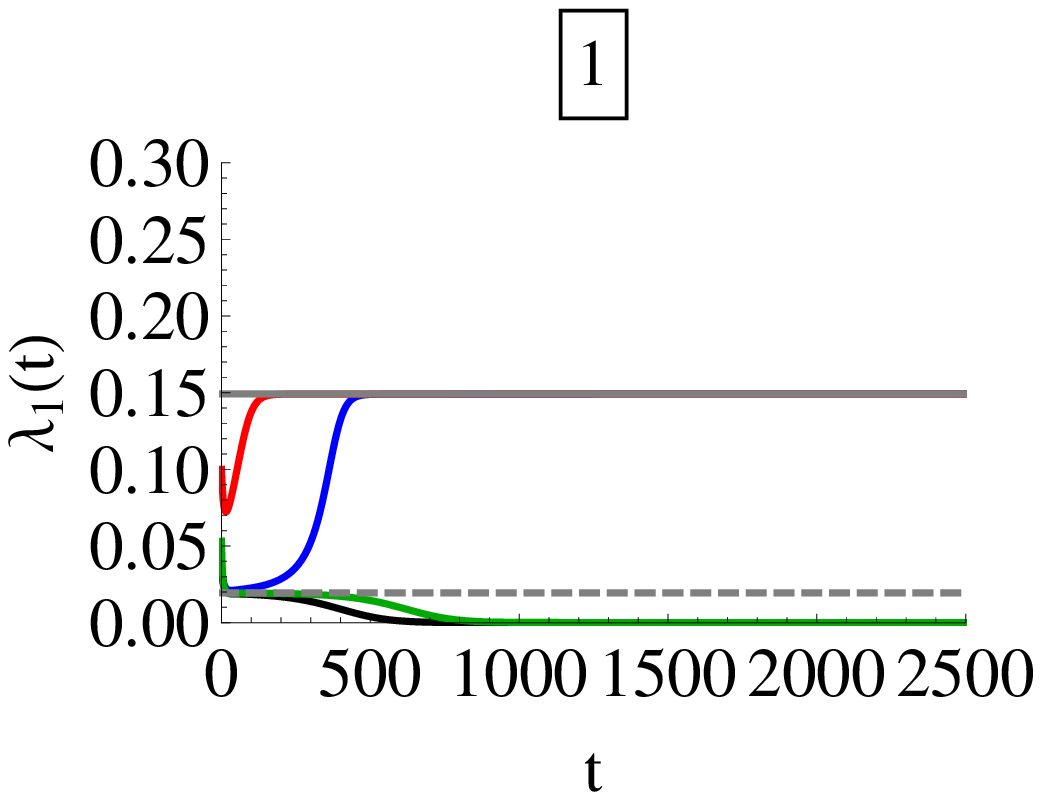} \includegraphics[width=4.8cm]{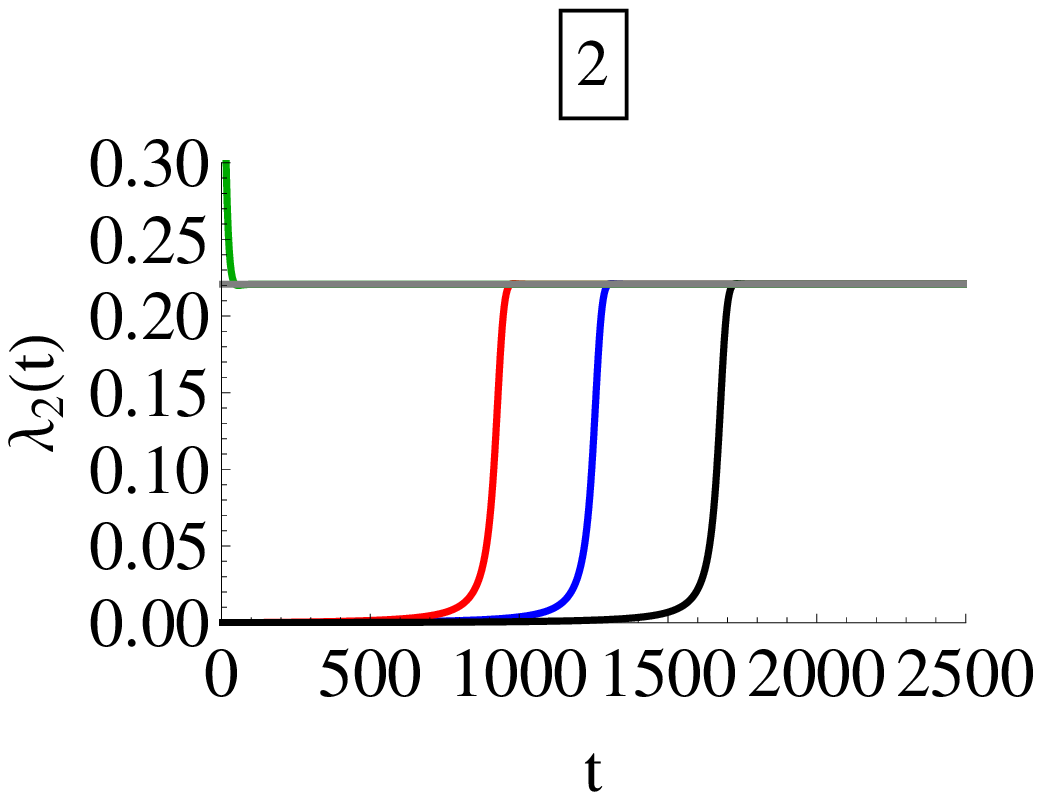} \includegraphics[width=4.8cm]{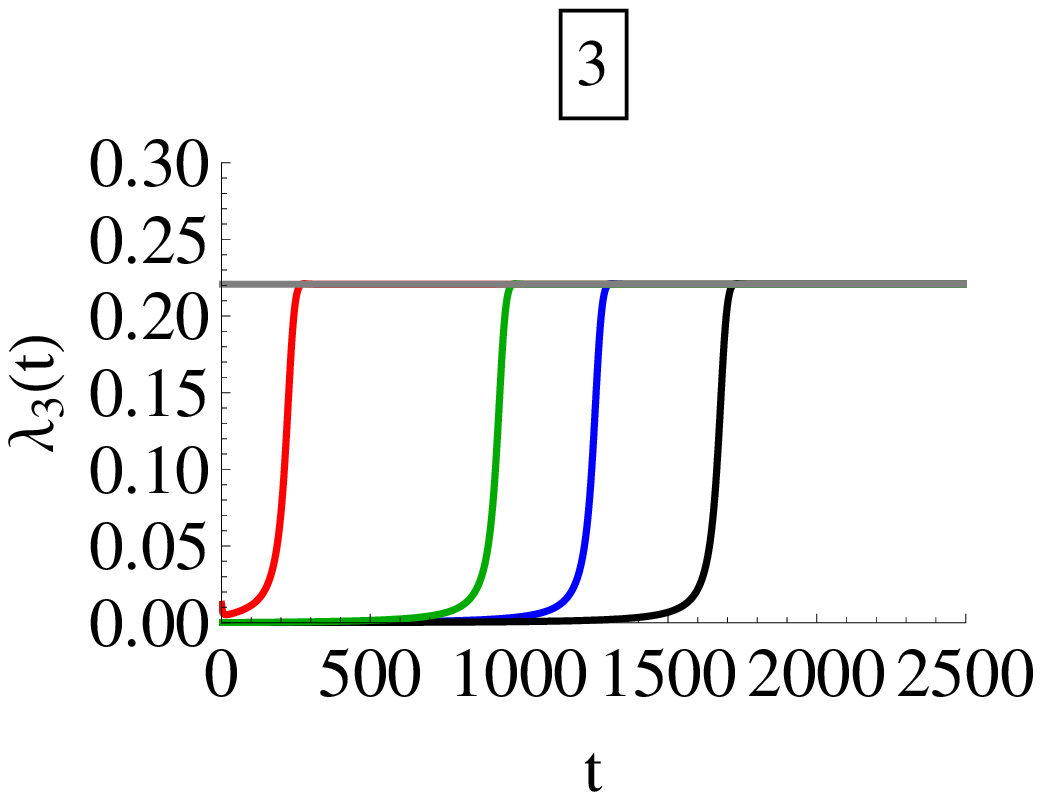}}\\
\subfigure[$\alpha=10^{-3}$]{\includegraphics[width=4.8cm]{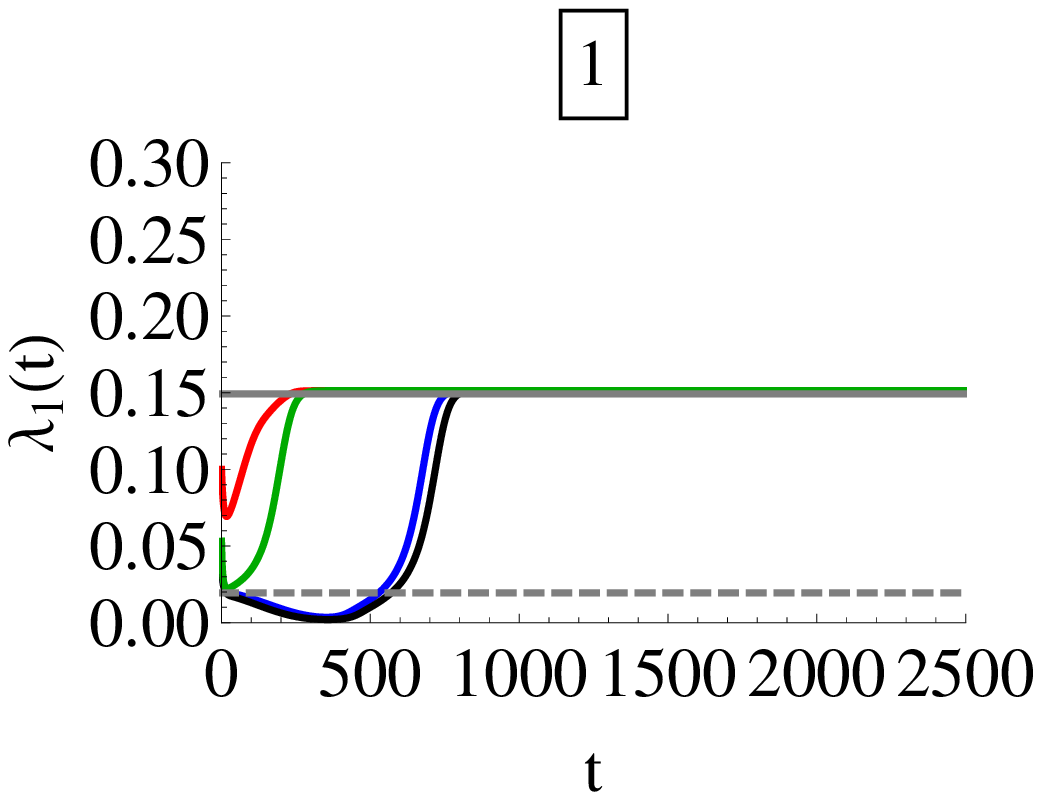} \includegraphics[width=4.8cm]{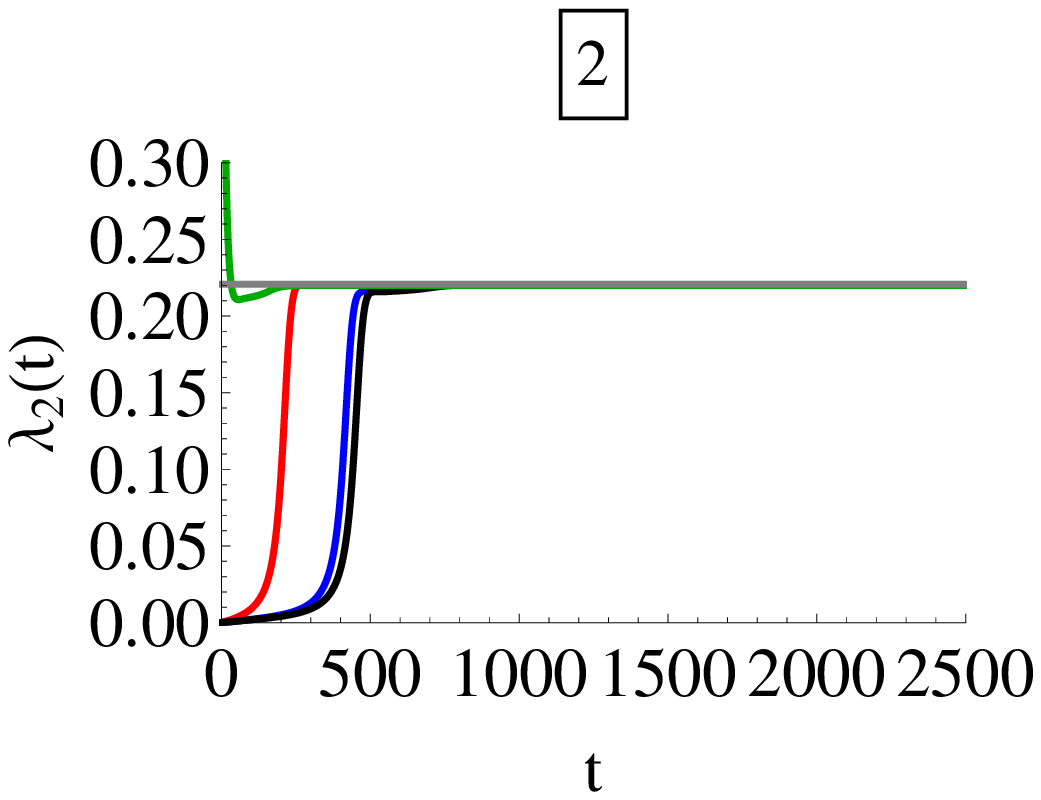} \includegraphics[width=4.8cm]{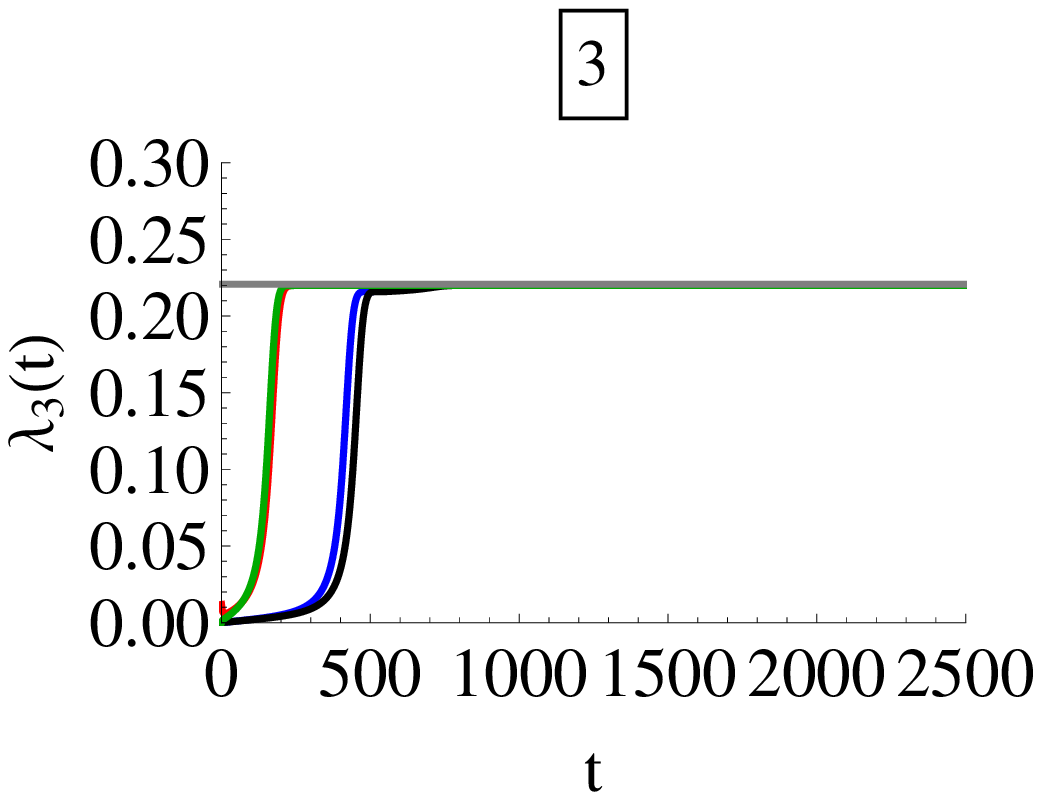}}\\
\subfigure[$\alpha=10^{-1}$]{\includegraphics[width=4.8cm]{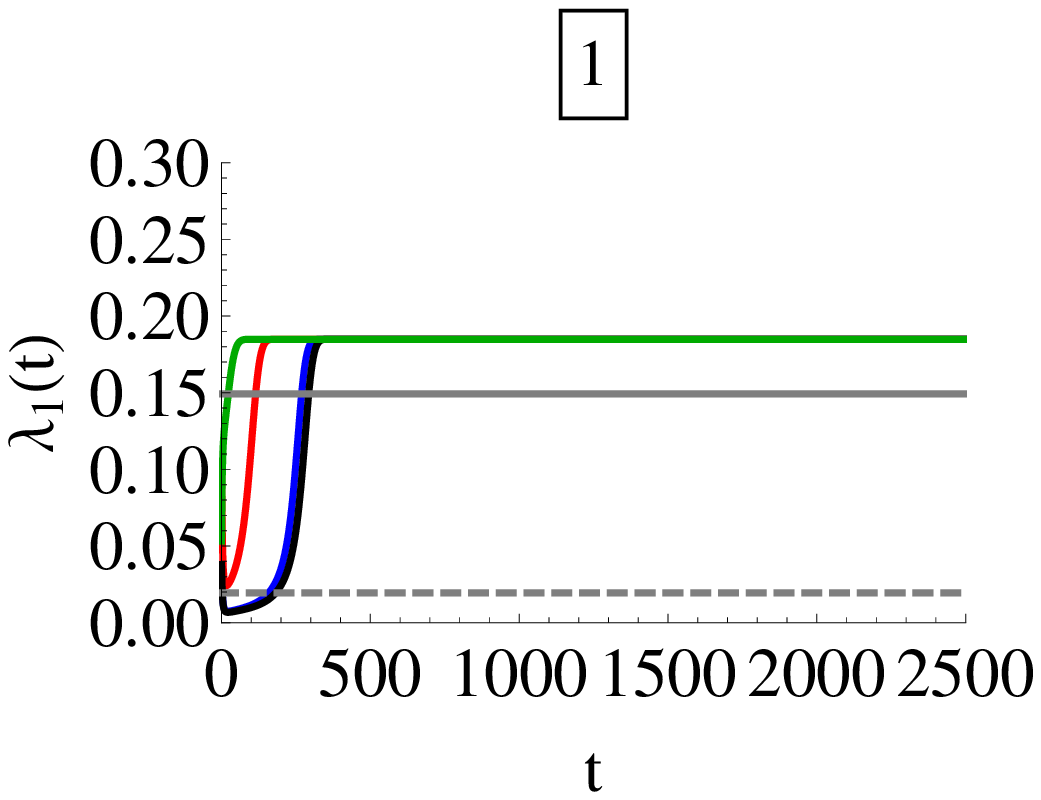} \includegraphics[width=4.8cm]{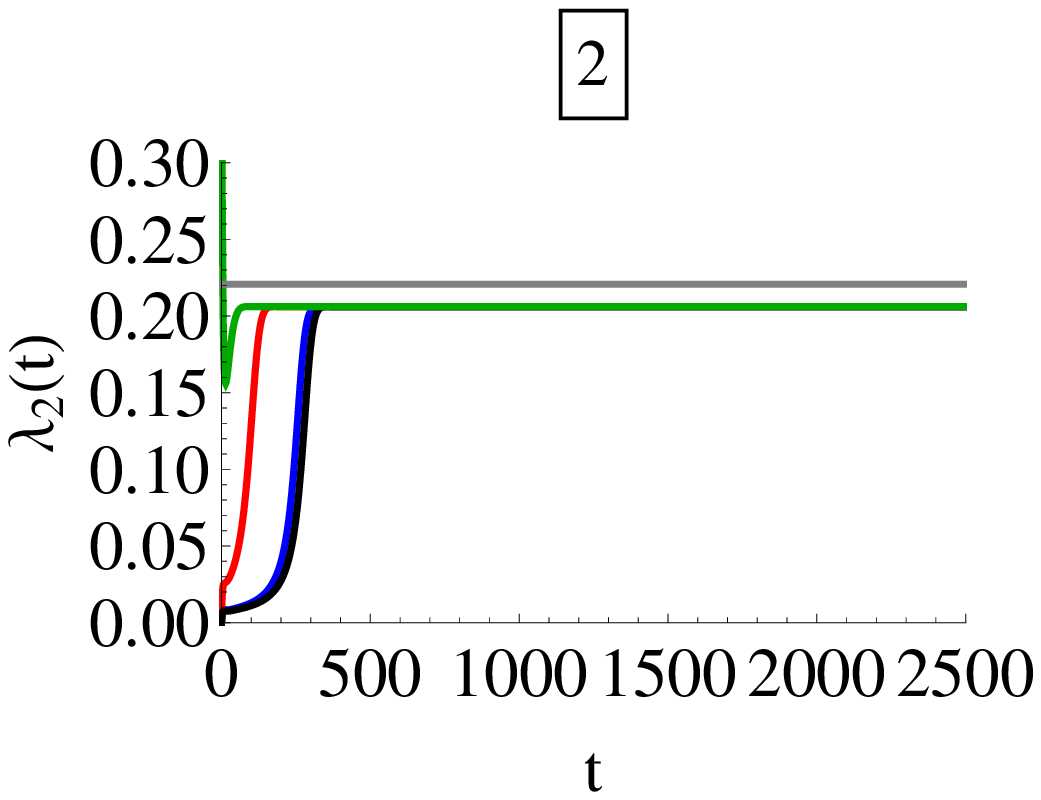} \includegraphics[width=4.8cm]{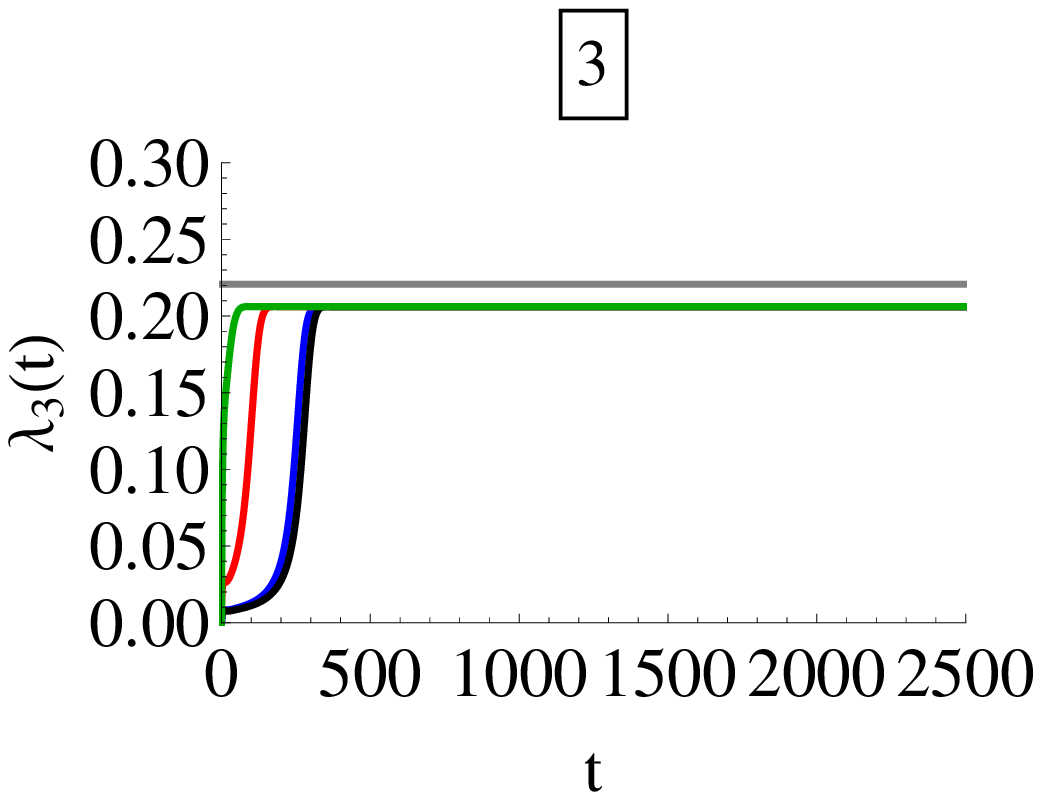}}
\caption{Solutions of system $(T_1)$--$(T_3)$ with HIV dynamics for different travel volumes. The three regions are considered to be symmetric in every parameter value but the local reproduction numbers, we applied the parameter set given in section \ref{sec:HIV} with $\beta^{i}_1=0.85$ and $\beta^{i}_2=\beta^{i}_3=1$ so that $\rn_c^{1}<\rn_H^{1}<1$ and  $\rn_H^{2},\rn_H^{3}>1$ are satisfied. We use the complete connection network depicted in Figure \ref{fig:irredrednetwork} (c), where the connectivity potential parameters $c^{ij}$, $i, j \in \{1,2,3\}$, $i \neq j$, are equal to one in all model classes. Solid and dashed gray lines correspond to steady state solutions in the regions in the absence of traveling.
}\label{fig:HIVrg1full}
\end{figure}
\end{document}